\documentclass[reqno]{amsart}

\usepackage{natbib}
\usepackage{amstext}
\usepackage{geometry}
\usepackage{ts_def_share}

\usepackage{enumerate}

\usepackage[colorlinks,urlcolor=red,citecolor=blue,linkcolor=red]{hyperref}
 
\DeclareMathOperator{\Real}{Re}
\DeclareMathOperator{\supp}{supp}
\DeclareMathOperator{\conv}{conv}
\newcommand{\Id}{\text{Id}}

\RequirePackage{colortbl}
\definecolor{tscolor}{rgb}{1.0,0.6,0.0}
\definecolor{bgcolor}{rgb}{0.5,0.2,0.7}

\usepackage[ulem=normalem,commandnameprefix=ifneeded]{changes}      % do not show them in final version 
\definechangesauthor[name = Thorsten Schmidt, color=tscolor]{TS}  %{green}
\definechangesauthor[name = Boris Günther, color=bgcolor]{BG}

\title{Path-dependent Affine Processes}
\date{\today}
\author{Boris G\"unther}
\address{Institute of Mathematics, University of Gie{\ss}en}

\author{Thomas Kruse}
\address{Department of Mathematics \& Informatics, University of Wuppertal}

\author{Ludger Overbeck}
\address{Institute of Mathematics, University of Gie{\ss}en}

\author{Thorsten Schmidt}
\address{Department of Mathematical Stochastics, University of Freiburg}

\begin{document}

\begin{abstract}
    We extend the classical theory
    of affine processes to a path-dependent setting by introducing path-dependent coefficients and provide analytic formulas for their Fourier--Laplace transform in terms of generalized Riccati-type equations.
    In the proposed framework, we define path-dependent affine processes through their exponential-affine Fourier--Laplace transform on the path space and establish a characterization theorem. 
    Conversely, for path-dependent stochastic differential equations with affine path-dependent coefficients, we also provide explicit exponential-affine representations of the Fourier--Laplace functional in terms of those Riccati equations. 
    Moreover, we derive a condition ensuring non-negativity of the path-dependent diffusion coefficient, guaranteeing well-posedness of the model. Finally, we apply these results to a path-dependent volatility model and a path-dependent extension of the Heston model, including a delayed Heston model as a special case.
\end{abstract}


\maketitle

\section*{Introduction}

Affine stochastic processes are one of the most widely used frameworks for constructing tractable multifactor models.
Their key feature is that, despite being non-Gaussian in general, they remain analytically tractable in the sense that their characteristic function is characterized by ordinary differential or integral equations of Riccati type. This analytical structure enables the efficient computation of many quantities of interest and is one of the core reasons why affine processes are used throughout mathematical finance.

The earliest applications of affine processes can be traced back to the early 1950s in the field of biology, where they appeared in the study of continuous state branching processes and Feller diffusions, \cite{feller_Diffusion_1951, lamperti_Continuous_1967}. Their first use in finance appeared in the modelling of interest rate term structures in \cite{cox_Theory_1985} and later also in stochastic volatility modeling, for example in the celebrated model of \cite{heston_ClosedForm_1993}.
A first systematic treatment of affine processes and their applications for finance was given in the seminal work of \cite{duffie_Affine_2003}.

Since then, the theory of affine processes has been extended in many directions, including
matrix valued processes on more general state spaces such as the cone of positive semidefinite matrices,
as well as infinite-dimensional affine processes and measure-valued extensions
(see e.g. \cite{keller-ressel_Affine_2019,cuchiero_Affine_2011,cuchiero_Affine_2016,schmidt_Infinite_2020, cuchiero_Measurevalued_2024}) to mention a few.

In these approaches,
the underlying process is often assumed to be Markovian. A notable exception %Most notably 
is the recent extension to affine stochastic Volterra equations, initiated in \cite{abijaber_Affine_2019} and extended to the inhomogeneous case in \cite{ackermann_Inhomogeneous_2022}.
In this  setting, which is in general neither Markovian nor a semimartingale, the authors manage to show that if the
coefficients of the underlying Volterra process are affine, then the process still admits a Fourier--Laplace transform of exponential-affine form. This transform is characterized by a system of Riccati-Volterra integral equations. Due
to the path-dependence induced by the Volterra kernel, the Fourier--Laplace transform is affine with respect to the past trajectory of the underlying process rather than only in its current state.

Motivated by these observations, the aim of this paper is to characterize a class of path-dependent processes whose Fourier--Laplace transforms depend on the past trajectory in an affine way. More precisely, we consider a $d$-dimensional path-dependent process of the form
\begin{equation}\label{eq:intro-path-dep}
    X(t) = X(0) + \int_{0}^{t} b(s,X_{[0,s]}) \, ds + \int_{0}^{t} \sigma(s,X_{[0,s]}) \, dW(s),
\end{equation}
where $W$ is an $m$-dimensional Brownian motion;
here $X(t)$ denotes the value of the process at time $t \in [0,T]$, while $X_{[0,t]} = \{X(s) : 0 \leq s \leq t\}$ denotes the past trajectory up to time $t$. In contrast to the classical affine setting, the coefficients $b$ and $\sigma$ are allowed to depend on the entire past path of the process rather than only on its current state. To capture this structure, we  formulate the affine property directly on the path space.

In Section \ref{sec:path-dependent-affine}, we study processes whose Fourier--Laplace transform is exponentially affine with respect to the past trajectory, i.e.
\begin{equation}\label{eq:intro-exp-affine}
    E\Bigl[\exp\Bigl(\int_{[0,T]} X(s) \, \mu(ds)\Bigr) \mid \ccF_t \Bigr] = \exp \Bigl(\phi(t) + \int_{[0,t]} X(s) \, \Psi(t,ds) \Bigr), \quad t \in [0,T],
\end{equation}
for admissible measures $\mu$ and sufficiently regular functions $\phi$ and $\Psi$.
We show that processes with this property necessarily possesses coefficients $b(t,x)$ and $a(t,x) = \sigma(t,x)\sigma(t,x)^\top$ that are affine of the form
\begin{equation}\label{eq:intro-coeff}
    b(t,x_{[0,t]}) = b^0(t) + \sum_{i = 1}^{d} \int_{[0,t]} x^i(s) \, b^i(t,ds), \quad a(t,x_{[0,t]}) = a^0(t) + \sum_{i = 1}^{d} \int_{[0,t]} x^i(s) \, a^i(t,ds),
\end{equation}
for all $t \in [0,T]$ and $x \in C([0,T],\R^d)$.
Here, $b^0$ and $a^0$ are $\R^d$-valued functions on $[0,T]$, while for each $i \in \{1,\dotsc,d\}$, $b^i$ and $a^i$ are finite signed measures taking values in $\R^d$ and in the space of symmetric matrices respectively.
Moreover, as in the classical theory, the functions $\phi$ and $\Psi$ appearing in \eqref{eq:intro-exp-affine} are characterized by solutions to a system of Riccati-type integral equations determined by those coefficients.

In Section \ref{sec:reverse-direction}, we study the reverse direction and show that affine
coefficients of the form \eqref{eq:intro-coeff} are sufficient to establish the exponential-affine
transform formula \eqref{eq:intro-exp-affine}, given that the functions $\phi$ and $\Psi$ solve the
associated Riccati-type equations.
In Section \ref{section:existence}, we turn to the question of existence and well-posedness of the
path-dependent stochastic integral equation \eqref{eq:intro-path-dep}.
First, we establish the existence of weak solutions under linear growth conditions on the coefficients.
In addition, we derive conditions ensuring that the affine diffusion coefficient $\sigma \sigma^\top = a$ remains non-negative, which guarantees well-posedness of the process.
Finally, in Section \ref{sec:examples} we apply our results to two concrete path-dependent models:
first, we construct a model with a genuine path-dependent volatility,
where the instantaneous variance depends on the difference of the current level of the process
from a moving average of its recent trajectory over a rolling window.
We show that well-posedness of this model is ensured under a drift condition derived in Section \ref{section:existence}.
Second, we consider a path-dependent extension of the Heston model in which the stochastic volatility process retains the classical square-root diffusion structure,
but its drift is allowed to depend on the past trajectory.
For both models we explicitly show that the associated Riccati equations do admit unique solutions and derive the specific form of the Fourier--Laplace formula.

\section{Preliminaries}\label{sec:preliminaries}

Consider a finite time horizon  $T \in (0, \infty)$ and a filtered probability space $(\Omega,\ccF, (\ccF_t)_{t \in [0,T]},P)$ satisfying the usual conditions and supporting an  $m$-dimensional Brownian motion $W$. For the tools from semimartingale calculus we refer to \cite{JacodShiryaev}.

We briefly recall the notion of path-dependent processes and introduce the relevant path spaces.
Technical details of the functional It\^o-calculus, such as horizontal and vertical derivatives, are delegated to Appendix~\ref{app:functional_ito}.
While our main tool is a semimartingale representation which we provide in the next section, the functional It\^o formula serves as a guiding example showing how such a semimartingale representation naturally arises in the path-dependent setting.

Let $x \in C([0,T],\R^d)$. For each $t \in [0,T]$, we define the stopped path
$$x_{[0,t]} = (x(s))_{s \in [0,t]} \in C([0,t],\R^d)$$
and analogously for half open intervals by $x_{[0,t)} = (x(s))_{s \in [0,t)} \in C([0,t),\R^d)$.
For continuous paths, $x_{[0,t)}$ uniquely determines $x(t)$ via
$\lim_{s \uparrow t} x(s) = x(t)$, but we keep both notions to distinguish
whether the value at time $t$ is included.
In order for us to consider dual pairings with measures on $[0,T]$, we identify $x_{[0,t]}$ with its zero extension to $[0,T]$ via
$x\Ind_{[0,t]}$,
and analogously with $x\Ind_{[0,t)}$ for $x_{[0,t)}$.

For path-dependent
models on a state space $S\subset \R^d$ it is natural to consider
$$\ccC_T(S)%([0,T],\R^d)
    := \{(t,x): t \in [0,T],\ x \in C([0,t],S)\},$$
i.e.\ the collection of continuous paths stopped at time $t$ together with their time parameter.
If $S=\R^d$ we simply write $\ccC_T^d$.
The space $\ccC_T(S)$ can also be interpreted as a vector bundle over time, i.e.\ $\ccC_T(S)  = \cup_{t \in [0,T]} \{t\} \times C([0,t],S).$
We equip $\ccC_T(S)$ with the metric
$$d_{\ccC}\big( (t,x),(t',x') \big) = |t-t'| + \sup_{u \in [0,T]} \| x(u \wedge t) - x'(u \wedge t') \|,$$
which induces the same topology as the stopped-path space discussed in Appendix \ref{app:functional_ito},
but does not distinguish jumps. Since we will not have to consider vertical perturbations of continuous paths as in the functional It\^o theory, it suffices to work on $C([0,T],S)$ which we view as a subspace of the Skorokhod space $D([0,T],S)$.
A \emph{non-anticipative functional} is a family of measurable functions $F=(F_t)_{t \in [0,T]}$ such that
$$F_t : C([0,t],S) \to \R^d, \quad t \in [0,T].$$

Equivalently, a non-anticipative functional may be viewed as map $F:\ccC_T(S) \to \R^d$ by considering $ F(t,x) = F_t(x). $
Continuity is always understood with respect to the metric $d_{\ccC}$.

Since our aim is to evaluate the Fourier--Laplace transforms of entire paths rather than terminal values, we recall that the dual of continuous functions can be identified with the space of finite signed Borel measures.
The extension to $d$ dimensions and complex values can be done as follows:  consider signed complex vector-valued (Borel) measures of bounded variation on $[0,T]$, i.e.\  countably additive functions on the Borel sigma field over $[0,T]$ with values in $\C^d$ and finite total variation.
We denote the space of these measures by $\ccM([0,T], \C^d)$ and refer to \cite{weaver_Measure_2013} for further details. Equipped with the total variation norm\footnote{Here, the total variation of a complex vector-valued measure $\mu$ on $[0,T]$ is defined as
$\|\mu\|_{TV} := \sup_{\pi} \sum_{A \in \pi} \|\mu(A)\|,$
where the supremum is taken over all countable, measurable, disjoint partitions $\pi$ of $[0,T]$.},
the space $\ccM([0,T], \C^d)$ is a Banach space.

For $x \in C([0,T],\R^d)$ and $\mu \in \ccM([0,T],\C^d)$ the dual pairing is defined for any $t \in [0,T]$
by
\begin{equation*}
    \mu \cdot x_{[0,t]} := \mu \cdot \bigl(\Ind_{[0,t]}x\bigr) := \sum_{i=1}^{d} \int_{[0,T]} \Ind_{[0,t]}x^i(s) \, \mu^i(ds),
\end{equation*}
where $x^i$ and $\mu^i$ denote the respective components.

Let $F$ be an $\R^d$-valued non-anticipative continuous functional. We call $F$ \emph{affine} if there exists functions
\begin{equation*}
    F^0 : [0,T] \to \R^d, \quad L_t : C([0,T],\R^d) \to \R^d,
\end{equation*}
such that for every $t \in [0,T]$, $L_t$ is linear and $L_t(x)$ depends only on $x_{[0,t]}$, understood under the zero extension, and
\begin{equation*}
    F(t,x)=F^0(t)+L_t(x), \quad (t,x) \in \ccC_T(\R^d).
\end{equation*}
By the Riesz--Markov representation theorem, each scalar component of $L_t$ can be represented by integration against a finite $\R$-valued Borel measure on $[0,T]$ (see e.g. \cite[Theorem 4.35]{weaver_Measure_2013}). Hence, for every $t \in [0,T]$, there exists a matrix of finite signed measures
$\mu(t)=(\mu_{ij}(t))_{1\leq i,j \leq d}$ such that
\begin{equation*}
    (L_t(x))^i = \sum_{j=1}^d \int_{[0,t]} x^j(s)\,\mu_{ij}(t,ds), \quad i=1,\dotsc,d.
\end{equation*}

\section{Path-dependent affine processes}\label{sec:path-dependent-affine}

In this section we develop the notation of an affine path-dependent process $X$ as a process whose conditional Fourier--Laplace transform is exponentially affine in the past path (see Definition \ref{def:path-dependent-affine} below). Before this, we detail the required prerequisites.

\begin{definition}
    An $\R^d$-valued stochastic process $X$ 
    is called a $d$-dimensional \emph{path-dependent process}, if it is a continuous semimartingale and there exist continuous $b : \ccC_T^d \to \R^d$ and measurable $\sigma : \ccC_T^d \to \R^{d \times m}$, such that $a(t,x) := \sigma(t,x) \sigma(t, x)^\top$ is continuous and $X$ satisfies
    \begin{align}\label{path_dependent_sde}
        X(t) & = X(0) + \int_{0}^{t} b(s,X_{[0,s]}) \, ds + \int_{0}^{t} \sigma(s,X_{[0,s]}) \, dW(s), \quad 0 \leq t \leq T,
    \end{align}
    with deterministic initial value $X(0) \in \R^d$.
\end{definition}
This framework naturally extends the classical time-inhomogeneous Markovian case, where coefficients only depend on time and the current state, to situations in which path-dependence is present. Note that this definition implicitly assumes that $b$ and $\sigma$ are non-anticipative and satisfy the necessary integrability conditions.

In Section~\ref{section:existence} we show the existence of weak solutions to the path-dependent SDE~\eqref{path_dependent_sde} under the following additional assumption.
\begin{assumption}\label{LG-Condition}
    There exists a constant $C_{LG} \in (0,\infty)$ such that the coefficients $b$ and  $\sigma$ of the path-dependent SDE \eqref{path_dependent_sde}
    satisfy for all $(t,x) \in \ccC_T^d$ the linear growth condition
    \begin{equation*}
        \left\|b(t,x)\right\|+\left\|\sigma(t,x)\right\| \leq C_{LG} \Bigl(1+ \sup_{s \in [0,t]} \left\|x(s)\right\| \Bigr).
    \end{equation*}
\end{assumption}

Next, we introduce an important condition which states that at every time $t$ the support of $X$ is rich enough to determine the affine structure.
To this end, let $S \subseteq \R^d$ be a closed convex cone of full dimension, i.e.\ a convex set, closed under multiplication with positive scalars and with linear hull equal to $\R^d$.
A classical example is of course $\R^m_{\geq 0} \times \R^n$ with $m+n=d$, which is often used as the canonical state space for affine processes, see e.g.\ \cite{duffie_Affine_2003}.
Recall that the support of a generic random variable $\xi$ is the smallest closed set $A$ such that $P(\xi \in A)=1$; we use the notation $\supp(\xi)$ for this set. If $\xi$ takes values in $C([0,T],\R^d)$, we consider the topology induced by the supremums norm $\|\cdot\|_\infty$.
For a set $A$, we write $\conv(A)$ for its convex hull, that is, the smallest convex set containing $A$.

\begin{condition}\label{cond:full support_inhom}
    We say that a path-dependent process $X$ has \emph{full path-dependent support}
    $$\conv(\supp(X_{[0,t]}))=C([0,t],S)$$
    for all $t\in [0,T]$.
\end{condition}
This condition ensures that the affine representation of the process is uniquely determined by its law.
We note that this condition differs from the Condition 2.3 in \cite{keller-ressel_Affine_2019}, where the full support condition is only required for $X(t)$, i.e.\ at each fixed point in time. Here, however we need to exclude cases where $X$ is constant over time, which will not allow to uniquely identify the affine parameters (see Lemma \ref{lem:phi and psi uniquely determined}).

Not every measure yields a well-defined Fourier–Laplace transform. We therefore restrict to the dual cone of the state space $S$
$$ \cU(T):= \{ \mu \in \ccM([0,T], \C^d) \colon \Real(\mu \cdot x) \leq 0\text{ for all }x \in C([0,T],S) \}. $$

To ease the notation, we consider in the following the terminal time $T$ as a variable parameter rather than fixed.
In the Markovian affine case, continuous differentiability in $t$ of $\Psi(t,T,\mu)$ is essential to derive the associated Riccati equations. In the path-dependent setting, however, the assumption of continuous differentiability is in practice often too strong. Therefore, we introduce a weaker notion, which we call \emph{semimartingale differentiability}, which requires that the derivative of $\Psi$ can be represented in a specific way as a semimartingale.

\begin{definition}\label{def:semimartingale-diff}
    Let $\Gamma : [0,T]\to \ccM([0,T], \C^d)$, and let $X$ be a $d$-dimensional path-dependent process.
    We say that the pair $(\Gamma,X)$ is \emph{semimartingale differentiable} if the process
    \begin{equation*}
        \bigl(\Gamma(t)\cdot X_{[0,t]}\bigr)_{0\leq t\leq T}
    \end{equation*}
    is a continuous semimartingale and admits a representation of the form
    \begin{equation}\label{eq:representation-general}
        \Gamma(t)\cdot X_{[0,t]} = \Gamma(0)\cdot X_{[0,0]} - \int_0^t B\bigl(s,X_{[0,s]}\bigr)\,ds + \int_0^t C(s)\,dX(s),
    \end{equation}
    where $B\colon \ccC_T^d\to \C$ is an affine non-anticipative functional and $C\colon [0,T]\to \C^d$ deterministic.
    We call $B$ and $C$ the coefficients of the semimartingale differentiability of the pair $(\Gamma,X)$.
\end{definition}

Note that this definition implicitly assumes that all integrals appearing in
\eqref{eq:representation-general} are well-defined.
Sufficient conditions are that
$\int_0^T \bigl|B(s,x_{[0,s]})\bigr|\,ds < \infty$ for every $x \in C([0,T],\R^d)$,
and that $C$ is measurable with
$\int_0^T C(s)^\top \, d\langle X\rangle(s)\, C(s) < \infty$ almost surely.

\begin{remark}
    Definition~\ref{def:semimartingale-diff} states that the semimartingale
    $(\Gamma(t) \cdot X_{[0,t]})_{t \in [0,T]}$ admits a decomposition
    into a stochastic integral with respect to $X$ and an additional drift term.
    More precisely, the stochastic integral part has a deterministic integrand $C$,
    while the drift contribution is given by a non-anticipative functional $B$.
    The latter should be understood as an additional drift on top of the drift already present in $X$.
    The linearity in the path $X$ of $\Gamma$ as well as the linearity of the stochastic integral in the integrator $X$ indicate that the drift functional $B$ is affine.
\end{remark}

In our concrete setting, the measure-valued map $\Gamma$ arises from the Fourier--Laplace transform and therefore depends additionally on a terminal time $T<\infty$ and a parameter in the dual cone $\mu \in \cU(T)$.
This leads us to the following assumption on the semimartingale differentiability of the pair $(\Psi, X)$.

\begin{assumption}\label{ass:sm-diff-Psi}
    Let $\Psi : \{(t,T,\mu)\in[0,T]^2\times\,\cU(T):\, t\leq T\}\to \ccM([0,T], \C^d)$
    and let $X$ be a $d$-dimensional path-dependent process.
    We assume that, for every $T < \infty$ and every $\mu\in\cU(T)$, the pair $(\Psi(\cdot,T,\mu),\,X)$
    is semimartingale differentiable with coefficients $(\alpha, \psi)$.
    That is, for all $t\in[0,T]$, the process
    \begin{equation*}
        \bigl(\Psi(t,T,\mu)\cdot X_{[0,t]}\bigr)_{0\leq t\leq T}
    \end{equation*}
    is a continuous semimartingale and admits the representation
    \begin{equation}\label{eq:representation alpha beta_inhom}
        \Psi(t,T,\mu)\cdot X_{[0,t]} = \Psi(0,T,\mu)\cdot X_{[0,0]} - \int_0^t \alpha(s,X_{[0,s]},T,\mu)\,ds + \int_0^t \psi(s,T,\mu)\,dX(s),
    \end{equation}
    where $\alpha=\alpha(T,\mu): \ccC_T^d \to \C $ is an affine non-anticipative functional and
    $\psi = \psi(T,\mu): [0,T] \to \C^d$. Moreover, we assume that $t \to \psi(t,T,\mu)$ is c\`agl\`ad and of finite variation.
\end{assumption}

The notation in Equation \eqref{eq:representation alpha beta_inhom} has been chosen deliberately. In the affine framework developed below, the function $\psi$ will be the fundamental object, satisfying a system of Riccati-type equations. Later we will also see that, similarly as $\phi$ in the Markovian affine case, once $\psi$ is known, $\alpha$ can be determined by a simple integration. Thus, the measure-valued function $\Psi$ is solely determined by $\psi$.
For notational simplicity, we may write $\psi(s)$ or $\psi(s,\mu)$ for $\psi(s,T,\mu)$ and similarly for $\alpha$, depending on the context.
The above condition is similar to absolute continuity. The following example puts the semimartingale differentiability assumption in relation to the functional It\^o calculus. The precise definitions and technical details of the functional It\^o calculus can be found in Appendix~\ref{app:functional_ito}.

\begin{example}\label{ex:functional-ito-semimartingale-diff}
    Let us consider the linear non-anticipative functional $\Gamma_t(x_{[0,t]})=\Gamma(t)\cdot x_{[0,t]}$ for all $(t,x) \in \ccC_T^d$.
    A natural candidate for the drift term in Definition~\ref{def:semimartingale-diff} is the horizontal derivative
    \begin{align*}
        \cD_t\Gamma(x) & = \lim_{h \to 0+} \frac{\Gamma_{t+h}(x_{t,h}) - \Gamma_t(x_t)}{h}                                    \\
                       & =\lim_{h \to 0+} \frac{\int_0^{t}x(s)(\Gamma(t+h)(ds) - \Gamma(t)(ds))+x(t)\Gamma(t+h)((t,t+h])}{h}.
    \end{align*}
    Hence, if we assume that the measures
    $$\frac{\partial \Gamma}{\partial t} := \lim_{h \to 0+} \frac{(\Gamma(t+h)(\cdot)- \Gamma(t)(\cdot))}{h}$$
    converges in the weak topology (for finite measures) to a finite measure
    $$\frac{\partial \Gamma(t)}{\partial t}(\cdot),$$
    and that the limit
    $$\lim_{h \to 0+} \frac{\Gamma(t+h)((t,t+h])}{h}$$
    exists and equals a function $\tilde{\Gamma}(t)$ which is continuous in $t$, and that the mapping $t\to\frac{\partial \Gamma(t)}{\partial t}(\cdot)$ is continuous.
    Then we may conclude that $\Gamma_\cdot \in \C^1$.

    Since in the vertical derivative only the path $x$ is perturbed at time $t$, while the integrator $\Gamma(t)(ds)$ remains fixed, we obtain that
    \begin{equation*}
        \nabla_x \Gamma_t \ \text{exists, does not depend on $x$, and equals the vector} \
        \big( \Gamma_i(t)(\{t\}) \big)_{i=1,\dots,d}.
    \end{equation*}
    Since the vertical derivative does not depend on $x$ anymore the second vertical derivative vanishes.

    Therefore, under the above assumptions (ensuring the existence of a continuous horizontal differential), we have
    $\Gamma_\cdot \in \C^{1,2}_b$, provided that $t \mapsto \Gamma_i(t)(\{t\})$ is continuous.

    Applying the functional It\^o formula then yields
    \begin{equation}
        \Gamma_t = \Gamma_0 + \int_0^t \cD_s \Gamma_s\, ds + \int_0^t \nabla_x \Gamma_s\, dX(s),
    \end{equation}
    with the non-random vector
    \begin{equation*}
        \nabla_x \Gamma_s = \big( \Gamma_i(s)(\{s\}) \big)_{i=1,\dots,d}.
    \end{equation*}
    Hence, in this case we may identify
    \begin{align*}
        B(s,x_{[0,s]}) = -\cD_s \Gamma_s(x_{[0,s]}), \quad
        C(s) = \big( \Gamma_i(s)(\{s\}) \big)_{i=1,\dots,d}.
    \end{align*}
\end{example}

\bigskip

The above example shows that every linear non-anticipative functional admits a first and second vertical derivative, but the horizontal derivative may fail to exist.

\begin{example}\label{example:semimartingale-functionals}
    The semimartingale differentiability condition from Definition~\ref{def:semimartingale-diff}
    is strictly weaker than the existence of a horizontal derivative in the sense of
    Example~\ref{ex:functional-ito-semimartingale-diff}. We illustrate this with two
    classes of linear non-anticipative functionals $\Gamma_t(x_{[0,t]}) = \Gamma(t) \cdot x_{[0,t]}$, $x \in C([0,T],\R)$,
    which satisfy the representation in Definition~\ref{def:semimartingale-diff}, but
    for which the horizontal derivative $D_t\Gamma$ fails to exist at some times.

    \begin{itemize}
        \item[(i)] Fix $a \in (0,T)$ and define
              \begin{equation*}
                  g(t) = \sqrt{(t-a)^+}, \quad \Gamma(t,ds) = g(t)\delta_t(ds).
              \end{equation*}
              Then
              \begin{equation*}
                  \Gamma_t(x_{[0,t]})=g(t)x(t), \quad x\in C([0,T],\R).
              \end{equation*}
              Now let $X$ be a continuous semimartingale. By integration by parts,
              \begin{equation*}
                  \Gamma_t(X_{[0,t]}) = g(t)X(t) = \int_0^t g(s)\,dX(s)+\int_0^t g'(s)X(s)\,ds,
              \end{equation*}
              where
              \begin{equation*}
                  g'(s)= \frac{\Ind_{\{s>a\}}}{2\sqrt{s-a}} \quad \text{for a.e. } s \in [0,T].
              \end{equation*}
              Hence, the pair $(\Gamma,X)$ is semimartingale differentiable with coefficients
              \begin{equation*}
                  C(s)=\sqrt{(s-a)^+}, \quad B(s,x_{[0,s]}) = -\frac{\Ind_{\{s>a\}}}{2\sqrt{s-a}}\,x(s).
              \end{equation*}
              On the other hand, the horizontal derivative does not exist at $t=a$.
              Indeed, for $x\in C([0,T],\R)$ and $h>0$,
              \begin{equation*}
                  \Gamma_{a+h}(x_{a,h}) = \sqrt{h}\,x(a), \quad \Gamma_a(x_a) = 0.
              \end{equation*}
              Therefore,
              \begin{equation*}
                  \frac{\Gamma_{a+h}(x_{a,h})-\Gamma_a(x_a)}{h} = \frac{x(a)}{\sqrt{h}},
              \end{equation*}
              which diverges whenever $x(a) \neq 0$.
              Thus, the horizontal derivative fails to exist at $t = a$.
        \item[(ii)] Fix $t_0\in[0,T)$ and define
              \begin{equation*}
                  k(s) := \sum_{n=1}^\infty (-1)^n \Ind_{(t_0 + 2^{-n-1}, t_0 + 2^{-n}]}(s), \quad s \in [0,T],
              \end{equation*}
              where we set $k(s) = 0$ for $s \leq t_0$ if necessary. Define
              \begin{equation*}
                  \Gamma(t,ds) = \delta_t(ds) + k(s)\Ind_{[0,t)}(s)\,ds, \quad \Gamma_t(x_{[0,t]}) = x(t) + \int_0^t k(s)x(s) \, ds.
              \end{equation*}
              Then, for every continuous semimartingale $X$,
              \begin{equation*}
                  \Gamma_t(X_{[0,t]}) = \Gamma_0(X_{[0,0]}) - \int_0^t B(s,X_{[0,s]}) \, ds + \int_0^t C(s) \, dX(s),
              \end{equation*}
              with
              \begin{equation*}
                  B(s,x_{[0,s]}) = -k(s)x(s) = B^0(s) + \int_{[0,s]} x(u) \, B^1(s,du), \quad (t,x) \in \ccC_T^1,
              \end{equation*}
              where
              \begin{equation*}
                  B^0(s)\equiv 0, \quad B^1(s,du) = -k(s)\delta_s(du), \quad C(s)\equiv 1.
              \end{equation*}
              However, the horizontal derivative fails at $t=t_0$. Indeed, for $h>0$ we have
              \begin{equation*}
                  \Gamma_{t_0+h}(x_{t_0,h}) - \Gamma_{t_0}(x_{t_0}) = x(t_0)\int_{t_0}^{t_0+h} k(s) \, ds.
              \end{equation*}
              Hence,
              \begin{equation*}
                  \frac{\Gamma_{t_0+h}(x_{t_0,h}) - \Gamma_{t_0}(x_{t_0})}{h}
                  = x(t_0) \frac{1}{h} \int_{t_0}^{t_0+h} k(s) \, ds.
              \end{equation*}
              Along the sequence $h_m = 2^{-m}$, $m \in \N$, one computes
              \begin{equation*}
                  \frac{1}{h_m}\int_{t_0}^{t_0+h_m} k(s) \, ds = \frac{(-1)^m}{3},
              \end{equation*}
              so the difference quotient oscillates and therefore has no limit whenever
              $x(t_0)\neq 0$. Consequently, the horizontal derivative fails to exist at that point.
    \end{itemize}
\end{example}

With those regularity conditions we are now ready to state the main definition.

\begin{definition}\label{def:path-dependent-affine}
    Let $X$ be a path-dependent process taking values in $S$.
    We call  $X$ \emph{affine} (up to time $T$), if there exist
    \begin{align*}
        \phi & : \{(t,T,\mu) \in [0,T]^2 \times \cU(T) : t \leq T\} \to \C,                \\
        \Psi & : \{(t,T,\mu) \in [0,T]^2 \times \cU(T) : t \leq T\} \to \ccM([0,T], \C^d),
    \end{align*}
    with $\phi(\cdot,T,0) = 0$ and $\Psi(\cdot,T,0) =0$
    such that
    \begin{enumerate}
        \item[(i)] for each fixed $\mu$, $t \mapsto \phi(t,T,\mu)$ is absolutely continuous on $[0,T]$ and the pair $(\Psi, X)$ satisfies Assumption~\ref{ass:sm-diff-Psi},
        \item[(ii)] for each fixed $(t,T)$, the mappings $\mu \mapsto \phi(t,T,\mu)$ and $\mu \mapsto \Psi(t,T,\mu)$ are continuous in the $weak^\ast$-topology, and
        \item[(iii)] the following affine transform formula holds:
              \begin{equation}\label{eq:affine-path}
                  E[e^{\, \mu \cdot X_{[0,T]} } \mid \ccF_t] = \exp\bigl( \phi(t,T,\mu) + \Psi(t,T,\mu) \cdot X_{[0,t]} \bigr)
              \end{equation}
              for all $ 0 \leq t \leq T $ and all $\mu \in \cU(T)$.
    \end{enumerate}
\end{definition}

As seen in Example~\ref{ex:functional-ito-semimartingale-diff}, a sufficient condition for the semimartingale differentiability assumption to hold is that $\Psi(t,T,\mu) \cdot X_{[0,t]} \in \C^{1,2}_{b}$ for all $t \in [0,T]$.
Note that the left-hand side of Equation \eqref{eq:affine-path} is always well-defined, due to the definition of the dual cone $\cU(T)$.

Although the definition above is given for a fixed time horizon $T$, it immediately implies the affine property for shorter horizons.
That is, the exponential-affine transform also holds for any $0 \leq T' < T$, as stated in the following lemma.

\begin{lemma}\label{lm:also-path-dependent}
    If $X$ is an affine process on $S$, then for all $0 \leq T' \leq T$, there exist
    $\C$- and $\ccM([0,T], \C^d)$
    %was: $\ccM^d([0,T'])$
    -valued functions $\phi(t,T',\mu)$ and $\Psi(t,T',\mu)$, satisfying $(i)$ and $(ii)$ of Definition~\ref{def:path-dependent-affine}, such that
    \begin{align}
        E[e^{ \, \mu \cdot X_{[0,T']} } \mid \ccF_t] = \exp\bigl( \phi(t,T',\mu) + \Psi(t,T',\mu) \cdot X_{[0,t]} \bigr)
    \end{align}
    for all $0 \leq t \leq T'$ and all $\mu \in \cU(T')$.
\end{lemma}
\begin{proof}
    The claim follows from a simple truncation argument, due to the structural assumptions on $S$.
    For $\mu \in \cU(T)$, define the measure $\tilde \mu := \Ind_{[0,T']} \mu$.
    By construction, $\tilde \mu \in \cU(T)$ and
    \begin{equation*}
        \tilde \mu \cdot X_{[0,T]} = \mu \cdot X_{[0,T']},
    \end{equation*}
    such that
    \begin{align}
        \phi(t,T',\mu) & := \phi(t,T,\tilde \mu) \quad \text{ and } \quad \Psi(t,T',\mu) :=\Psi(t,T,\tilde \mu)
    \end{align}
    have the desired properties.
\end{proof}

The following lemma shows that the full-support condition ensures uniqueness
of the affine representation, in the sense that $\phi(t,T,\mu)$ and the
restriction of $\Psi(t,T,\mu)$ to $[0,t]$ are uniquely determined.
In particular, no uniqueness is asserted for the part of $\Psi(t,T,\mu)$
on $(t,T]$.

\begin{lemma}\label{lem:phi and psi uniquely determined}
    Let $X$ be an affine process satisfying Condition~\ref{cond:full support_inhom}.
    Then $\phi(t,T,\mu)$ and the restricted measure $ \Psi(t,T,\mu)|_{[0,t]} \in \ccM([0,t],\C^d)$
    are uniquely determined by \eqref{eq:affine-path} for all $0 \leq t \leq T$ and every $\mu \in \cU(T)$.
\end{lemma}
\begin{proof}
    Consider fixed $0 \leq t \leq T$ and $\mu \in \cU(T)$. Suppose that $\widetilde{\phi}(t,T,\mu)$ and $\widetilde{\Psi}(t,T,\mu)$ also satisfy
    \eqref{eq:affine-path} and are also continuous in $\mu$. Set $p(t,T,\mu) := \widetilde{\phi}(t,T,\mu) - \phi(t,T,\mu)$ and $q(t,T,\mu) := \widetilde{\Psi}(t,T,\mu) - \Psi(t,T,\mu)$. Since $X_{[0,t]}$ only depends on the path on $[0,t]$, we have
    $q(t,T,\mu)\cdot X_{[0,t]} = q(t,T,\mu)|_{[0,t]}\cdot X_{[0,t]}$.
    Then,
    $$ 1 =  \exp\big(p(t,T,\mu) +  q(t,T,\mu)|_{[0,t]} \cdot  X_{[0,t]}  \big) \quad P\text{-a.s.}$$
    This implies that
    $$p(t,T,\mu) +  q(t,T,\mu)|_{[0,t]} \cdot  X_{[0,t]}  \in 2\pi i \Z \quad P\text{-a.s.}$$
    However, since $\cU(T)$ is convex and hence connected, its image under a continuous function, in the weak$^\ast$-topology, is also connected and a connected subset of a discrete set is a singleton.
    Therefore,
    $\mu \mapsto p(t,T,\mu) +  q(t,T,\mu)|_{[0,t]} \cdot X_{[0,t]}$ is constant on $\cU(T)$ and, therefore, equal to $p(t,T,0) +  q(t,T,0)|_{[0,t]} \cdot  X_{[0,t]} = 0$.
    Hence,
    \begin{equation*}
        p(t,T,\mu) + q(t,T,\mu)|_{[0,t]} \cdot x = 0,
    \end{equation*}
    for all $x \in \supp(X_{[0,t]})$ and $\mu \in \cU(T)$. Taking convex combinations, this equality can be extended to $x \in C([0,t],S)$. Since $S$ has full linear span, we conclude that $p(t,T,\mu)=0$ and $q(t,T,\mu)|_{[0,t]} = 0$ for all $\mu \in \cU(T)$.
\end{proof}
Note that by \eqref{eq:affine-path}, $\phi$ and $\Psi$ also satisfy the terminal conditions $\phi(T,T,\mu)=0$ and $\Psi(T,T,\mu)=\mu$.

\begin{lemma}\label{lem:beta_inhom}
    Let $X$ be an affine process (up to time $T$).
    Then, for the function $\psi$, given by the semimartingale differentiability Assumption~\ref{ass:sm-diff-Psi}, it holds that
    $$ \psi(T,T,i u \delta_{T}) = iu^\top$$
    for all $u \in \R^d$ and all $T \geq 0$.
\end{lemma}
\begin{proof}
    Indeed, by Definition~\ref{def:path-dependent-affine}, $\Psi(T,T,\mu) = \mu$, such that
    \begin{align}
        \Psi(T,T,\mu) \cdot X_{[0,T]} & = \int_{0}^{T}X(s) \, \mu(ds)
    \end{align}
    holds for all $T \geq 0$.
    When $\mu = iu\delta_{T} \in \cU(T)$ for a $u \in \R^d$, this reduces to
    \begin{align}
        \int_{0}^{T}X(s) \, \mu(ds) = iu^\top X(T) = iu^\top X(0) + \int_{[0,T]} iu^\top dX(s).
    \end{align}
    Uniqueness of the semimartingale characteristics together with representation \eqref{eq:representation alpha beta_inhom} then yields
    $$\psi(T,T,iu\delta_{T})=iu^\top$$
    as claimed.
\end{proof}

We now turn to the general structural result which characterizes the drift and diffusion coefficients of a path-dependent affine process. % taking values in $S$ with its associated Riccati equations. 
To simplify notation we introduce the following $d \times d$ matrix-valued function
\begin{equation}\label{notation:B}
    B(t,x) = \left( ( b^1(t) \cdot x^1 ) \dots (b^d(t) \cdot x^d)\right), \quad (t,x) \in \ccC_T(S).
\end{equation}
Furthermore, for any row vector $v \in (\C^d)^\ast$ we define the $(\C^d)^\ast$-valued function
\begin{equation}\label{notation:A}
    A(t,x,v) = \left(v (a^1(t) \cdot x^1) v^\top, \dots, v  (a^d(t)\cdot x^d) v^\top\right), \quad (t,x) \in \ccC_T(S).
\end{equation}

\begin{theorem}\label{thm:affine1_inhom}
    Let $X$ be an affine process on $S$ satisfying the full-support condition, Condition \ref{cond:full support_inhom} and Assumption~\ref{LG-Condition}.
    Then, there exist $b^0: [0,T]  \to \R^d$,\ $a^0 : [0,T] \to \R^{d \times d}$ and $b^i: [0,T] \to \ccM([0,T],\R^d)$, $a^i : [0,T] \to \ccM([0,T],\R^{d\times d})$, $i=1,\dots,d$, such that
    \begin{align}\label{eq:affine representation_inhom}
        b(t,x_{[0,t]}) & = b^0(t) + \sum_{i=1}^d  b^i(t) \cdot x^i_{[0,t]}, \\
        \label{eq:affine representation2_inhom}
        a(t,x_{[0,t]}) & = a^0(t) + \sum_{i=1}^d  a^i(t) \cdot x^i_{[0,t]},
    \end{align}
    for all $(t,x) \in \ccC_T(S)$. Moreover, the functions $\phi$ and $\psi$ satisfy the following generalized Riccati equations% for all $(t,x) \in \ccC_T(S)$
    \begin{align}\label{eq:riccati_phi_thm}
        \phi'(t,T,\mu) & = -\psi(t,T,\mu) b^0(t) - \frac{1}{2} \psi(t,T,\mu) a^0(t) \psi(t,T,\mu)^\top, \quad dt\text{-a.e. on } t \in [0,T],                   \\
        \label{eq:riccati_beta_thm}
        \psi(t,T,\mu)  & = \mu([t,T]) + \int_{t}^{T} \psi(s,T,\mu)B(s,\Ind_{[t,s]}) + \frac{1}{2} A(s,\Ind_{[t,s]}, \psi(s,T,\mu)) \, ds, \quad t \in (0,T],
    \end{align}
    with $\phi(T,T,\mu)=0$ and $\Psi(T,T,\mu)=\mu$,
    where $\psi$ is given by Assumption~\ref{ass:sm-diff-Psi} and $\phi'(t,T,\mu)$ denotes a Lebesgue density of $\phi$.
    Furthermore,
    \begin{align}\label{eq:main-thm-fnctional-psi}
        \begin{split}
            \Psi(t,T,\mu) \cdot X_{[0,t]} & = \Psi(0,T,\mu) \cdot X_{[0,0]} + \int_{0}^{t} \psi(s,T,\mu) \, dX(s)                                                                                             \\
                                          & \quad - \int_{0}^{t} \sum_{i = 1}^{d} \psi(s,T,\mu) ( b^i(s) \cdot X^i_{[0,s]}) + \frac{1}{2} \psi(s,T,\mu) ( a^i(s) \cdot X^i_{[0,s]}) \psi(s,T,\mu)^\top \, ds,
        \end{split}
    \end{align}
    with
    \begin{equation*}
        \Psi(0,T,\mu)(\{0\}) = \mu([0,T]) + \int_{0}^{T} \psi(s,T,\mu)B(s,\Ind_{[0,s]}) + \frac{1}{2} A(s,\Ind_{[0,s]}, \psi(s,T,\mu)) \, ds.
    \end{equation*}
\end{theorem}

\bigskip

Note that $b(t,x)$, $a(t,x)$ are non-anticipative functionals and therefore  in the affine representation \eqref{eq:affine representation_inhom} it holds that
$$ b^i(t) \cdot x^i  = b^i(t) \cdot x_{[0,t]}^i = \int_0^t x^i(s) \, b^i(t,ds) = \sum_{j=1}^d e_j \int_0^t x^i(s) \, b^{i,j}(t,ds),$$
$0 \leq t \leq T$ and $i=1,\dots,d$.
This holds similarly for \eqref{eq:affine representation2_inhom}.

For the following, we define the complex-valued process $G$ by
\begin{align}
    G(t,T,\mu) :=  \psi(t,T,\mu) b(t,X_{[0,t]}) + \frac{1}{2}  \psi(t,T,\mu)a(t,X_{[0,t]}) \psi(t,T,\mu)^\top,
\end{align}
for $0 \leq t \leq T$ and $\mu \in \cU(T)$. Note that since $t \to \psi(t,T,\mu)$ is by assumption c\`agl\`ad that $t \mapsto G(t,T,\mu)$ is also c\`agl\`ad.
The next lemma shows that already $G$  possesses the affine property.

\begin{lemma}\label{lem:dG_inhom}
    Let $X$ be an affine process. Then, for $\mu \in \cU(T)$, and $ 0 \leq t \leq T$,
    \begin{align}\label{eq:13_inhom}
        G(t,T,\mu) + \phi'(t,T,\mu) - \alpha(t,X_{[0,t]},T,\mu)= 0 \quad dt \otimes dP\text{-a.s},
    \end{align}
    where $\phi'$ is a Lebesgue density of $\phi$.
\end{lemma}
\begin{proof}
    We fix $\mu \in \cU(T)$ and consider the martingale
    $$ M_t := E\big[ e^{ \, \mu \cdot X_{[0,T]}} \mid \ccF_t \big] = \exp\big( \phi(t,T,\mu) + \Psi(t,T,\mu) \cdot X_{[0,t]} \big), \quad 0 \leq t \leq T.$$
    The classical It\^o-formula together with the semimartingale representation of $(\Psi,X)$, shows that
    \begin{align}
        \begin{split}
            \frac{dM_t}{M_t} & = \big( \phi'(t,T,\mu) -  \alpha(t,X_{[0,t]},T,\mu) \big)dt + \psi(t,T,\mu)\, dX(t) \\
                             & + \frac 1 2 \sum_{i,j=1}^d\psi^i(t,T,\mu)\psi^j(t,T,\mu)d \langle X^{i,j}\rangle(t).
        \end{split}
    \end{align}
    Since $X$ is a path-dependent process, Equation \eqref{path_dependent_sde} yields that $dX(t) = b(t,X_{[0,t]})dt + \sigma(t,X_{[0,t]}) dW(t)$ and hence
    $d \langle X \rangle(t) = a(t,X_{[0,t]}) \, dt$. Moreover, since $M$ is a martingale, its finite variation part vanishes almost-surely which
    implies that Equation \eqref{eq:13_inhom} holds $dt \otimes dP$-almost surely.
\end{proof}

\begin{proof}[Proof of Theorem \ref{thm:affine1_inhom}]
    We begin by observing that, by Lemma \ref{lem:dG_inhom}, Equation \eqref{eq:13_inhom} holds $dt\otimes dP$-a.s. on $[0,T]$ and, since $t \mapsto G(t,T,\mu)$ is c\`agl\`ad, the processes $t \mapsto \phi'(t,T,\mu) - \alpha(t,X_{[0,t]},T,\mu)$ admits a c\`agl\`ad version. Replacing $\phi' - \alpha$ by this version, we may, by left-continuity, therefore assume that \eqref{eq:13_inhom} holds $P$-almost surely for all $t \in (0,T]$. In particular at time $t = T$ this yields
    \begin{align}\label{eq:proof-alpha-phi}
        \psi(T,T,\mu) b(T,X_{[0,T]}) + \frac{1}{2} \psi(T,T,\mu)a(T,X_{[0,T]}) \psi(T,T,\mu)^\top = -\phi'(T,T,\mu) + \alpha(T,X_{[0,T]},T,\mu), \quad P\text{-a.s.}
    \end{align}
    Let $u \in \R^d$. By Lemma~\ref{lem:beta_inhom}, for $\mu \in iu\delta_T$ we have for all $x \in \supp(X_{[0,T]})$
    \begin{equation}\label{eq:affine?}
        iu^\top b(T,x) - \frac{1}{2}  u^\top a(T,x) u = -\phi'(T,T,i u \delta_{T}) + \alpha(T,x,T,i u \delta_{T}).
    \end{equation}
    Since $\alpha$ is affine in $x$, it follows that for every $u \in \R^d$ the functional
    \begin{equation}\label{eq:F_u}
        F_u(x) := iu^\top b(T,x) - \frac{1}{2}  u^\top a(T,x) u, \quad x \in \supp(X_{[0,T]})
    \end{equation}
    is also affine.
    The support Condition \ref{cond:full support_inhom} now implies that $F_u(x)$ is affine for all $x \in C([0,T],S)$. We now separate the drift and diffusion terms and conclude that they are affine. Since sums and differences of affine functionals remain affine,
    \begin{align*}
        F_{u}(x)-F_{-u}(x) & = 2iu^\top b(T,x) \\
        F_{u}(x)+F_{-u}(x) & = -u^\top a(T,x)u \\
    \end{align*}
    are affine in $x$.
    Choosing the $k$-th unit vector for $u$, we conclude that each component $b_k(T,x)$ of the vector $b(T,x)$ is affine.
    For the matrix-valued functional $a(T,\cdot)$ we identify the entries by the polarization identity
    \begin{equation*}
        u^\top a(T,x) v = \frac{1}{4}\bigl((u+v)^\top a(T,x)(u+v) - (u-v)^\top a(T,x)(u-v)\bigr), \quad u,v \in \R^d.
    \end{equation*}
    Choosing now the $k$-th unit vector for $u$ and the $j$-th unit vector for $v$ this identity implies that the matrix-entries
    $u^\top a(T,x)v = a_{ij}(T,x)$ are affine. Therefore, $b(T,\cdot)$ and $a(T,\cdot)$ are affine functionals on $C([0,T],\R^d)$.
    This yields representation \eqref{eq:affine representation_inhom} for the terminal time $T$ by letting the measures $b^i(T,\cdot)$ and $a^i(T,\cdot)$ denote the representing measures by the Riesz--Markov theorem.
    Having established the affine property of $b(T,\cdot)$ and $a(T,\cdot)$ for every fixed $T$, we now extend this property to all $t \in (0,T]$.
    Fix such a $t$ and consider the truncated process $X_{[0,t]}$.
    By Lemma~\ref{lm:also-path-dependent}, this process is again affine up to time $t$ for all $\cU(t)$, so the argument above applies with $T$ replaced by $t$. We conclude with this that the representation \eqref{eq:affine representation_inhom} holds for all $t \in (0,T]$ and $x \in C([0,T],S)$. Finally, since $b$ and $a$ were assumed to be continuous and $X(0)$ is a deterministic initial condition, the affine representation extends to $t = 0$ as well.

    Assume for the remainder of the proof for all $(t,x) \in \ccC_T(S)$ that $b(t,x)$ and $a(t,x)$ satisfy the relation as given in \eqref{eq:affine representation_inhom} and \eqref{eq:affine representation2_inhom}, respectively. Then, as before, it follows from
    Equation \eqref{eq:13_inhom} that for all $(t,x) \in \ccC_T(S)$
    \begin{align*}
        \phi'(t,T,\mu)            & = -\psi(t,T,\mu)b^0(t) - \frac{1}{2} \psi(t,T,\mu)a^0(t) \psi(t,T,\mu)^\top, \quad dt\text{-a.s.}                                                               \\
        % \alpha(t,x,T,\mu) & = \sum_{i=1}^d \psi(t,T,\mu) (b^i(t) \cdot x^i) + \frac{1}{2}\sum_{i=1}^d \psi(t,T,\mu) (a^i(t) \cdot x^i) \psi(t,T,\mu)^\top,
        \alpha(t,x_{[0,t]},T,\mu) & = \sum_{i = 1}^{d} \psi(t,T,\mu) ( b^i(t) \cdot x^i_{[0,t]}) + \frac{1}{2} \psi(t,T,\mu) ( a^i(t) \cdot x^i_{[0,t]}) \psi(t,T,\mu)^\top, \quad dt\text{-a.s.},
    \end{align*}
    which establishes \eqref{eq:riccati_phi_thm} and identifies $\alpha$. A simple integration now yields
    \begin{equation*}
        \int_{0}^{t} \alpha(s,x_{[0,s]},T,\mu) \, ds = \int_{0}^{t} \sum_{i = 1}^{d} \psi(s,T,\mu) ( b^i(s) \cdot x^i_{[0,s]}) + \frac{1}{2} \psi(s,T,\mu) ( a^i(s) \cdot x^i_{[0,s]}) \psi(s,T,\mu)^\top \, ds.
    \end{equation*}
    Combining this result with the semimartingale representation given by Assumption~\ref{ass:sm-diff-Psi} establishes Equation \eqref{eq:main-thm-fnctional-psi}.
    At time $t = T$ and by assumption that $t \mapsto \psi(t,T,\mu)$ is of finite variation, Equation \eqref{eq:main-thm-fnctional-psi} reads with the integration by parts formula
    \begin{align*}
        \Psi(T,T,\mu) \cdot x_{[0,T]} & = \Psi(0,T,\mu) \cdot x_{[0,0]} + \psi(T,T,\mu)x(T) - \psi(0,T,\mu)x(0) - \int_{0}^{T} x(s) \, d\psi(s,T,\mu)                                                     \\
                                      & \quad - \int_{0}^{T} \sum_{i = 1}^{d} \psi(s,T,\mu) ( b^i(s) \cdot x^i_{[0,s]}) + \frac{1}{2} \psi(s,T,\mu) ( a^i(s) \cdot x^i_{[0,s]}) \psi(s,T,\mu)^\top \, ds,
    \end{align*}
    for all $x_{[0,T]} \in \supp(X_{[0,T]})$. Taking convex combinations, this can be extended to hold for all $x \in C([0,T],S)$ by Condition~\ref{cond:full support_inhom}.

    Moreover, by integration by parts, the identity depends only on $x$ through point evaluations and integrals against finite signed measures, whose total variation is integrable in time by Assumption~\ref{LG-Condition} together with \eqref{eq:affine representation_inhom} and \eqref{eq:affine representation2_inhom}.
    Let $0 \leq s_1 \leq s_2 \leq T$ and choose continuous functions $x_n \in C([0,T],\R)$ with $0 \leq x_n \leq 1$ such that
    \begin{equation*}
        x_n(u)\to \Ind_{[s_1,s_2]}(u)\quad \text{for all } u \in [0,T],
    \end{equation*}
    in particular $x_n(0)\to \Ind_{[s_1,s_2]}(0)$ and $x_n(T)\to \Ind_{[s_1,s_2]}(T)$.
    The endpoint terms converge by construction. For the other integral terms
    we have to be careful, since the measures which appear in the
    identity may have atoms. We therefore choose the
    approximating functions $x_n$ in such a way that they converge to
    $\Ind_{[s_1,s_2]}$ in $L^1$ with respect to one finite measure which
    dominates all measures appearing in the identity. For example, we can take
    a finite measure which dominates $|\mu|$, the total variation of
    the Stieltjes measure $d\psi(\cdot,T,\mu)$, the endpoint atoms, and also
    the measures \begin{equation*}
        A\mapsto
        \int_0^T
        \left(
        |b^i(r)|(A\cap[0,r]) + |a^i(r)|(A\cap[0,r])
        \right)dr,
        \quad i=1,\dots,d,
    \end{equation*}
    where $|b^i(r)|$ and $|a^i(r)|$ denote the total variation measures of
    the signed vector-valued and matrix-valued measures $b^i(r,\cdot)$ and
    $a^i(r,\cdot)$, respectively.
    Such an approximation is possible by regularity of finite Borel measures on $[0,T]$. With this choice, all
    integrals against the signed measures appearing above converge, and hence
    the identity also holds for $x=1_{[s_1,s_2]}$.
    Since $\Psi(T,T,\mu) = \mu$, we get for $t \in (0,T]$, $j \in \{1,\dotsc,d\}$ and $x^j = e_j\Ind_{[t,T]}$
    \begin{equation*}
        \mu^j([t,T]) = \psi^j(t,T,\mu) - \int_t^T \psi(s,T,\mu)\bigl(b^j(s)\cdot \Ind_{[t,s]}\bigr) + \frac{1}{2} \psi(s,T,\mu)\bigl(a^j(s)\cdot \Ind_{[t,s]}\bigr)\psi(s,T,\mu)^\top \,ds.
    \end{equation*}
    Since $j$ was arbitrary, this yields the Riccati equation
    \begin{equation*}
        \psi(t,T,\mu) = \mu([t,T]) + \int_{t}^{T} \psi(s,T,\mu)B(s,\Ind_{[t,s]}) + \frac{1}{2} A(s,\Ind_{[t,s]}, \psi(s,T,\mu)) \, ds, \quad t \in (0,T].
    \end{equation*}
    For time $t = 0$, the same argument with $x^j = e_j \Ind_{[0,T]}$ yields the boundary condition
    \begin{equation*}
        \Psi(0,T,\mu)(\{0\}) = \mu([0,T]) + \int_{0}^{T} \psi(s,T,\mu)B(s,\Ind_{[0,s]}) + \frac{1}{2} A(s,\Ind_{[0,s]}, \psi(s,T,\mu)) \, ds.
    \end{equation*}
\end{proof}

The next lemma shows that $\phi$ and $\Psi$ satisfy the semi-flow property.

\begin{lemma}\label{lm:semi-flow}
    Let $X$ be an $S$-valued affine process satisfying the support Condition~\ref{cond:full support_inhom}. Then the following holds:
    \begin{enumerate}[(i)]
        \item the function $\mu \mapsto \phi(t,T,\mu)$ maps $\cU(T)$ to $\C_{\leq 0}$ and $\mu \mapsto \Psi(t,T,\mu)$ maps $\cU(T)$ to $\cU(T)$,
        \item the functions $\phi$ and $\Psi$ satisfy the semi-flow property, i.e.,
              \begin{align}
                  \phi(s,T,\mu) & = \phi(t,T,\mu) + \phi(s,t, \Psi(t,T,\mu)), &  & \phi(T,T, \mu) = 0,  \\
                  \Psi(s,T,\mu) & = \Psi(s,t,\Psi(t,T,\mu)),                  &  & \Psi(T,T,\mu) = \mu,
              \end{align}
              for all $0 \leq s \leq t \leq T$.
    \end{enumerate}

\end{lemma}
\begin{proof}
    \begin{enumerate}[(i)]
        \item By equation \eqref{eq:affine-path} we have for all $\mu \in \cU(T)$
              \begin{equation*}
                  E \left[e^{\mu\cdot X_{[0,T]}} \mid \ccF_t\right] = \exp \left(\phi(t,T,\mu)+\Psi(t,T,\mu) \cdot X_{[0,t]}\right), \quad 0 \leq t \leq T.
              \end{equation*}
              Since $\Real(\mu \cdot X) \leq 0$ $P$-a.s., we have that the left-hand side is bounded by one in absolute value.
              Thus,
              \begin{equation*}
                  \Real\left(\phi(t,T,\mu)+\Psi(t,T,\mu)\cdot X_{[0,t]}\right) \leq 0 \quad P\text{-a.s.}
              \end{equation*}
              and also $\Real \left(\phi(t,T,\mu)+\Psi(t,T,\mu) \cdot x\right) \leq 0$ for all $x \in \supp(X_{[0,t]})$.
              Taking now convex combinations of these inequalities and using the support Condition \ref{cond:full support_inhom} we obtain that the inequality must hold for all $x \in C([0,t],S)$. Since $S$ is a cone this in turn implies that
              $\Real \phi(t,T,\mu) \leq 0$ and $\Psi(t,T,\mu) \in \cU(T)$.
        \item Let $0 \leq s \leq t \leq T$ and $\mu \in \cU(T)$.
              By the affine transform formula,
              \begin{equation}\label{eq:semi-flow-affine}
                  E\left[e^{\mu \cdot X_{[0,T]}} |\, \ccF_{s}\right] = \exp(\phi(s,T,\mu) + \Psi(s,T,\mu) \cdot X_{[0,s]}).
              \end{equation}
              On the other hand, applying iterated expectations to the left-hand side of \eqref{eq:semi-flow-affine}
              \begin{align*}
                  E\left[e^{\mu \cdot X_{[0,T]}} |\, \ccF_{s}\right] & = E\left[E\left[e^{\mu \cdot X_{[0,T]}} |\, \ccF_{t}\right] |\, \ccF_{s}\right]                  \\
                                                                     & = E\left[\exp\left(\phi(t,T,\mu)+ \Psi(t,T,\mu) \cdot X_{[0,t]} \right)|\, \ccF_{s}\right]       \\
                                                                     & = \exp\big(\phi(t,T,\mu)\big)E\left[e^{\Psi(t,T,\mu) \cdot X_{[0,t]}}|\, \ccF_{s}\right]         \\
                                                                     & =\exp\Bigl(\phi(t,T,\mu)+\phi(s,t,\Psi(t,T,\mu)) + \Psi(s,t,\Psi(t,T,\mu))\cdot X_{[0,s]}\Bigr).
              \end{align*}
              Note that the exponent on the right hand side is continuous in $\mu$ (under the weak$^\ast$-topology) and that the same holds true for \eqref{eq:semi-flow-affine}. By the same argument as in the proof of Lemma~\ref{lem:phi and psi uniquely determined}, we conclude that
              \begin{equation*}
                  \phi(s,T,\mu) + \Psi(s,T,\mu) \cdot x_{[0,s]} = \phi(t,T,\mu) + \phi(s,t,\Psi(t,T,\mu)) + \Psi(s,t,\Psi(t,T,\mu))\cdot x_{[0,s]},
              \end{equation*}
              for all $x \in C([0,s],S)$. Since the linear hull of $S$ is $\R^d$ the semi-flow property follows.
    \end{enumerate}
\end{proof}

\section{Affine Characterization of Path-Dependent SDEs}\label{sec:reverse-direction}

In this section we show that if the coefficients of the path-dependent SDE are affine, then the corresponding Fourier-Laplace transform admits an exponential-affine representation.
Let $S\subseteq \R^d$ be a closed state space and let
$ b:\ccC_T(S)\to \R^d$, $a:\ccC_T(S)\to \bS_d^+$
be continuous non-anticipative functionals. Let $\sigma:\ccC_T(S)\to\R^{d\times m}$
be a measurable square root of $a$, i.e.
$\sigma(t,x)\sigma(t,x)^\top = a(t,x)$ for all $(t,x)\in\ccC_T(S)$.
For instance, one may take $m=d$ and $\sigma(t,x)= \sqrt{a(t,x)}$.
Let $X$ be an $S$-valued
continuous semimartingale satisfying
\begin{equation}\label{path_dependent_sde_reverse_dir}
    X(t) = X(0) + \int_{0}^{t} b(s,X_{[0,s]}) \, ds + \int_{0}^{t} \sigma(s,X_{[0,s]}) \, dW(s), \quad 0 \leq t \leq T,
\end{equation}
By $\lVert \cdot\rVert_{TV}$ we denote the total variation norm on $\ccM([0,T],\R^d)$ or on $\ccM([0,T],\R^{d\times d})$, respectively.

\begin{assumption}\label{ass:affine-coeff}
    Suppose there exists $b^0: [0,T]  \to \R^d$,\ $a^0 : [0,T] \to \bS_d$ and $b^i: [0,T] \to \ccM([0,T],\R^d)$, $a^i : [0,T] \to \ccM([0,T],\bS_d)$, $i=1,\dots,d$, such that
    \begin{align}
        b(t,x) & = b^0(t) + \sum_{i=1}^d  b^i(t) \cdot x^i, \\
        a(t,x) & = a^0(t) + \sum_{i=1}^d  a^i(t) \cdot x^i,
    \end{align}
    for all $(t,x) \in \ccC_T(S)$.
    Assume additionally that
    $$\int_0^T \|b^i(s)\|_{TV}\,ds + \int_0^T \|a^i(s)\|_{TV}\,ds <\infty,\quad \forall i = 1,\dotsc,d.$$
\end{assumption}

Note that positivity of $a$ is only imposed on the state space $\ccC_T(S)$.
This is essential: if positivity were imposed on the whole linear path space
$\ccC_T^d$, then the affine dependence of $a$ on the path would necessarily
be trivial; see Remark~\ref{rem:nontrivial-affine-diffusion-support}.

\begin{lemma}\label{product-rule-lemma}
    Let $X$ be a continuous semimartingale.
    Suppose $f : [0,T] \to (\C^d)^\ast$ is bounded and measurable and let $\nu : [0,T] \to \ccM([0,T], \C^{d \times d})$ be measurable such that $\int_0^T \lVert\nu(r)\rVert_{TV}\,dr < \infty$. For $B \in \cB([0,T])$ write $\nu(r,B) = \bigl(\nu^1(r,B),\dots,\nu^d(r,B)\bigr)$,
    where $\nu^j(r,B) \in \C^d$ denotes the $j$-th column.
    Then, for every $t \in [0,T]$, the following integration by parts formula holds:
    \begin{align*}
        \int_0^t \int_s^T f(r) \nu\big(r,[s,r]\big) \, dr \, & dX(s) - \int_0^t f(r) (\nu(r) \cdot X_{[0,r]})\, dr                                                                              \\
                                                             & = \int_t^T f(r) \Big(\nu\big(r,(t,r]\big)X(t) + \nu(r) \cdot X_{[0,t]} \Big) dr - \int_0^T f(r) \nu\big(r,[0,r]\big) \, dr X(0).
    \end{align*}
\end{lemma}

\begin{proof}
    Define
    \begin{equation*}
        G(s) = \int_s^T f(r) \nu(r,[s,r]) \, dr, \quad s \in [0,T].
    \end{equation*}
    Since $f$ is bounded and $\int_0^T \|\nu(r)\|_{TV}\,dr < \infty$,
    we have for every $s \in [0,T]$
    \begin{equation*}
        \|G(s)\| \leq \sup_{s \in [0,T]} \|f(s)\| \int_s^T \|\nu(r)\|_{TV} \, dr
        \leq \sup_{s \in [0,T]} \|f(s)\| \int_0^T \|\nu(r)\|_{TV} \, dr < \infty.
    \end{equation*}
    Moreover, since $G$ is deterministic, measurable and bounded and $X$ is a continuous semimartingale, the stochastic integral
    $$\int_0^t G(s) \, dX(s), \quad t \in [0,T],$$
    is well-defined.
    Note that $G$ is of finite variation, and an application of the integration-by-parts formula for continuous semimartingales yields
    \begin{equation}\label{eq:ibp-GX}
        \int_0^t G(s) \, dX(s) = G(t)X(t) - G(0)X(0) - \int_0^t X(s) \, dG(s).
    \end{equation}
    Next, note that
    \begin{equation*}
        G(t) = \int_t^T f(r)\nu(r,[t,r]) \, dr, \quad G(0) = \int_0^T f(r)\nu(r,[0,r]) \, dr.
    \end{equation*}
    Therefore, it remains to identify the term $\int_0^t X(s) \, dG(s)$.
    To this end, Fubini's theorem yields
    \begin{align*}
        \int_0^t X(s) \, dG(s) & = -\int_0^t f(r)(\nu(r) \cdot X_{[0,r]}) \, dr - \int_t^T f(r)(\nu(r) \cdot X_{[0,t]}) \, dr \\
                               & \quad +X(t) \int_t^T f(r)\nu(r,[t,r]) \, dr.
    \end{align*}
    Substituting this into \eqref{eq:ibp-GX} leads us to
    \begin{align*}
        \int_0^t G(s) \, dX(s) & = \int_0^t f(r)(\nu(r)\cdot X_{[0,r]}) \, dr + \int_t^T f(r)(\nu(r)\cdot X_{[0,t]}) \, dr \\
                               & \quad -X(t) \int_t^T f(r)\nu(r,[t,r]) \, dr + G(t)X(t) - G(0)X(0).
    \end{align*}
    Since
    $$G(t)X(t) - X(t)\int_t^T f(r)\nu(r,[t,r]) \, dr = 0,$$
    we arrive at
    \begin{align*}
        \int_0^t G(s) \, dX(s) - \int_0^t f(r)(\nu(r)\cdot X_{[0,r]}) \, dr & = \int_t^T f(r)(\nu(r)\cdot X_{[0,t]}) \, dr - \int_0^T f(r) \nu(r,[0,r]) \, dr X(0).
    \end{align*}
    This can now be rewritten as
    \begin{align*}
        \int_0^t\int_s^T f(r) \nu(r,[s,r]) \, dr \, dX(s) -\int_0^t f(r)(\nu(r)\cdot X_{[0,r]})\,dr & = \int_t^T f(r)\Bigl(\nu(r,(t,r])X(t)+\nu(r)\cdot X_{[0,t]}\Bigr)\,dr \\
                                                                                                    & \quad - \int_0^T f(r)\nu(r,[0,r]) \, dr X(0),
    \end{align*}
    which proves the claim.
\end{proof}

\subsection{Conditional Fourier--Laplace functional of path-dependent affine processes}

Recall the notation on $B$ and $A$ from Equations \eqref{notation:B} and \eqref{notation:A}, i.e.
\begin{equation*}
    B(t,x) = \left( ( b^1(t) \cdot x^1 ), \dots, (b^d(t) \cdot x^d)\right),
\end{equation*}
and
\begin{equation*}
    A(t,x,v) = \left(v (a^1(t) \cdot x^1) v^\top, \dots, v  (a^d(t)\cdot x^d) v^\top\right), \end{equation*}
for $ (t,x) \in \ccC_T(S)$ and $v \in (\C^d)^\ast$. Although $B$ and $A$ are induced by the coefficient representation on
$\ccC_T(S)$, expressions such as $B(s,\Ind_{[t,s]})$ and
$A(s,\Ind_{[t,s]},v)$ are understood purely in terms of the representing
measures. Namely,
\begin{equation*}
    B(s,\Ind_{[t,s]}) = \left( b^1(s,[t,s]),\dots,b^d(s,[t,s])\right),
\end{equation*}
and
\begin{equation*}
    A(s,\mathbf 1_{[t,s]},v) = \left( v a^1(s,[t,s])v^\top,\dots,v a^d(s,[t,s])v^\top\right).
\end{equation*}

\begin{theorem}\label{main_thm} 
    Assume that Assumptions \ref{LG-Condition} and \ref{ass:affine-coeff} hold.
    Let $\mu \in \ccM([0,T], \C^d)$
    and let $\psi \in L^2([0,T], (\C^d)^\ast)$
    solve the Riccati-type integral equation
    \begin{equation}\label{riccati-ode}
        \psi(t) = \mu([t,T]) +
        \int_t^T \psi(s) B(s, \Ind_{[t,s]}) + \frac{1}{2} A(s, \Ind_{[t,s]}, \psi(s))\, ds, \quad t\in[0,T]
    \end{equation}
    and define
    \begin{equation}
        \phi(t) = \int_t^T \psi(s) b^0(s) + \frac{1}{2}\psi(s)a^0(s)\psi(s)^\top \, ds, \quad t\in[0,T].
    \end{equation}
    Then the $\C$-valued process $Y$,
    %$= (Y(t))_{t\in[0,T]}$ 
    defined by
    \begin{align}\label{Process-Y}
        \begin{split}
            Y(t) & = \phi(t) + \psi(0)X(0) +\int_0^t \psi(s) dX(s)                                                                                                                       \\
                 & \quad -\int_0^t\sum_{i=1}^{d}\Bigl( \psi(s) ( b^i(s) \cdot X^i_{[0,s]}) + \frac{1}{2} \psi(s) ( a^i(s) \cdot X^i_{[0,s]}) \psi(s)^\top \Bigr)ds, \quad 0 \leq t \leq T,
        \end{split}
    \end{align}
    satisfies
    \begin{align}\label{Y-tilde}
        \begin{split}
            Y(t) & = \phi(t) + \psi(t)X(t)+ \mu \cdot X_{[0,t)}                                                                                                                            \\
                 & \quad + \int_t^T \sum_{i=1}^{d} \Bigl( \psi(s) ( b^i(s) \cdot X^i_{[0,t)}) + \frac{1}{2}\psi(s) (a^i(s) \cdot X^i_{[0,t)}) \psi(s)^\top \Bigr) ds, \quad 0 \leq t \leq T.
        \end{split}
    \end{align}
    Moreover, the process $\exp(Y)$ is a local martingale and if it is a true martingale, then %one has the exponential-affine transform formula
    \begin{equation}\label{exp-affine-formula}
        E\left[e^{\mu \cdot X_{[0,T]}} |\, \ccF_t\right] = \exp\left(Y(t)\right), \quad 0 \leq t \leq T.
    \end{equation}
\end{theorem}

\begin{remark}
    \begin{enumerate}
       \item[(i)] In the setting of Theorem~\ref{main_thm} any $L^2$-solution $\psi$ of the Riccati equation~\eqref{riccati-ode}
        admits a c\`agl\`ad version, and we always work with this version. Indeed, for every finite
        signed measure $\mu$, the map $t \mapsto \mu([t,T])$ is c\`agl\`ad and similarly, for fixed $s\in[0,T]$, the maps $t \mapsto b^i(s,[t,s])$ and $t\mapsto a^i(s,[t,s])$. In particular, since every c\`agl\`ad function on the
        compact interval $[0,T]$ is bounded, this version of $\psi$ belongs to
        $L^\infty([0,T],(\C^d)^\ast)$.
        \item[(ii)] Unlike in the Markovian case, the function $\phi$ is not necessarily continuously differentiable.
              Since the coefficients $b$ and $a$ are continuous affine functionals, this implies that their constant parts $b^0$ and $a^0$ are continuous. Moreover, as $\psi$ is c\`agl\`ad and belongs to $L^2([0,T])$, the function $t \mapsto \phi(t)$
              is absolutely continuous on $[0,T]$, and thus of bounded variation.
        \item[(iii)] Note that the right-hand side of \eqref{Process-Y} cannot be formulated for a general path $x \in C([0,T], \R^d)$ due to the integral with respect to $x$.
              In contrast, the alternative representation~\eqref{Y-tilde} is well-defined for every $x \in C([0,T], \R^d)$. Moreover, the identity \eqref{Y-tilde}
              clearly shows which term appears due to the path-dependency of the coefficients, since in the classical Markovian case we only have $Y(t) = \phi(t) + \psi(t)X(t)$.

    \end{enumerate}
\end{remark}

The following corollary shows that $Y(t)$ depends on the path $X_{[0,t]}$ in an affine way.

\begin{corollary}
    In the setting of Theorem~\ref{main_thm}, define the $\ccM([0,T],\C^d)$-valued measure $\Psi$ by
    \begin{align}\label{eq:Psi-measure}
        \Psi^j(t,T,\mu,A) & = \psi^j(t)\delta_t(A) + \mu^j(A \cap [0,t)) \notag            \\
                          & \quad + \int_t^T \Bigl( \psi(u) b^j\bigl(u, A \cap [0,t)\bigr)
        + \frac{1}{2} \psi(u)\, a^j\bigl(u, A \cap [0,t)\bigr)\, \psi(u)^\top\Bigr)\, du, \quad j = 1,\dotsc, d,
    \end{align}
    for all Borel sets $A \subseteq [0,T]$.
    Then
    \begin{equation*}
        Y(t)=\phi(t)+\Psi(t,T,\mu)\cdot X_{[0,t]}
    \end{equation*}
    and if $e^Y$ is a true martingale, then the conditional Fourier--Laplace transform admits the exponential-affine representation.
    Moreover, the pair $(\Psi,X)$ satisfies the semimartingale representation of
    Assumption~\ref{ass:sm-diff-Psi} with
    \begin{align*}
        \psi(t,T,\mu)             & = \psi(t)                                                                \\
        \alpha(t,X_{[0,t]},T,\mu) & = \psi(t,T,\mu)B(t,X_{[0,t]}) + \frac{1}{2}A(t,X_{[0,t]},\psi(t,T,\mu)).
    \end{align*}
\end{corollary}

\begin{proof}
    By \eqref{eq:Psi-measure}, the measure $\Psi(t,T,\mu)$ decomposes into an atom at $\{t\}$
    and a finite signed measure supported on $[0,t)$. Hence,
    \begin{align*}
        \int_{[0,t]} X(u) \, \Psi(t,T,\mu,du) & = X(t)\Psi(t,T,\mu,\{t\}) + \int_{[0,t)} X(u) \, \Psi(t,T,\mu,du) \\
         &= \psi(t)X(t) + \int_{[0,t)} X(u) \, \mu(du)          \\
         & \quad + \int_{[0,t)} X(u)
        \left( \int_t^T \sum_{i=1}^d \Bigl[\psi(s)\, b^i(s,du) + \frac{1}{2} \psi(s) a^i(s,du) \psi(s)^\top \Bigr] ds\right).
    \end{align*}
    Since $X$ is continuous and the measures $b^i(s,\cdot)$ and $a^i(s,\cdot)$ have finite variation,
    Fubini's theorem yields
    \begin{align*}
        \int_{[0,t]} X(u)\,\Psi(t,T,\mu,du) & = \psi(t)X(t) + \mu \cdot X_{[0,t)} \\
         & \quad + \int_t^T \sum_{i=1}^d \left( \psi(s)(b^i(s)\cdot X^i_{[0,t)}) + \frac{1}{2}\psi(s)(a^i(s)\cdot X^i_{[0,t)})\psi(s)^\top \right) ds.
    \end{align*}
    Therefore, $\phi(t) + \Psi(t, T, \mu) \cdot X_{[0,t]}$ equals \eqref{Y-tilde}, which in turn equals to \eqref{Process-Y} with Theorem~\ref{main_thm}.
    The second statement now directly follows from \eqref{Process-Y}, since $\Psi(0,T,\mu,\{0\}) = \psi(0)$ and therefore we can identify for every $t \in [0,T]$
    \begin{align*}
        \psi(t,T,\mu)             & = \psi(t),                                                                          \\
        \alpha(t,X_{[0,t]},T,\mu) & = \sum_{i=1}^d \left( \psi(t,T,\mu)(b^i(t)\cdot X_{[0,t]}^i) + \frac{1}{2} \psi(t,T,\mu)(a^i(t)\cdot X_{[0,t]}^i)\psi(t,T,\mu)^\top \right)
    \end{align*}
\end{proof}

\begin{proof}[Proof of Theorem~\ref{main_thm}]
    We begin by proving that equation \eqref{Process-Y} equals \eqref{Y-tilde}.
    Let $Y$  be given by equation \eqref{Process-Y}. With the Riccati equation $\eqref{riccati-ode}$,
    \begin{align*}
        Y(t) & = \phi(t) + \psi(0)X(0) - \int_0^t\left( \psi(r)\sum_{i=1}^{d} (b^i(r) \cdot X^i_{[0,r]}) + \frac{1}{2} \sum_{i = 1}^{d}\psi(r) (a^i(r) \cdot X^i_{[0,r]}) \psi(r)^\top \right)dr \\
             & \quad +\int_0^t \left(\mu([s,T]) + \int_s^T \psi(r) B(r, \Ind_{[s,r]}) + \frac{1}{2} A(r, \Ind_{[s,r]}, \psi(r)) \, dr\right) dX(s)                                               \\
             & = \phi(t) + \psi(0)X(0) + \int_0^t \mu([s,T]) dX(s) + \frac{1}{2} \int_0^t \int_s^T A(r, \Ind_{[s,r]}, \psi(r)) \, dr \, dX(s)                                                    \\
             & \quad - \frac{1}{2} \int_0^t \sum_{i = 1}^{d}\psi(r) (a^i(r) \cdot X^i_{[0,r]}) \psi(r)^\top dr + \int_0^t \int_s^T \psi(r) B(r, \Ind_{[s,r]}) \, dr \, dX(s)                     \\
             & \quad - \int_0^t \psi(r) \sum_{i=1}^{d} (b^i(r) \cdot X^i_{[0,r]}) dr.
        \intertext{Applying Lemma~\ref{product-rule-lemma} twice, once on the pair $\psi$ and $B$ and once on $A$, we get}
        Y(t) & = \phi(t) + \psi(0)X(0) + \int_0^t \mu([s,T]) \, dX(s) - \frac{1}{2} \int_0^T A(r, \Ind_{[0,r]}, \psi(r)) \, dr X(0)                                                              \\
             & \quad + \int_{t}^{T} \psi(r) \left(B(r, \Ind_{(t,r]})X(t) + \sum_{i=1}^{d} (b^i(r) \cdot X^i_{[0,t]})\right)dr - \int_0^T \psi(r) B(r, \Ind_{[0,r]}) \, dr X(0)                   \\
             & \quad + \frac{1}{2} \int_{t}^{T} A(r, \Ind_{(t,r]}, \psi(r))X(t) + \sum_{i = 1}^{d}\psi(r) (a^i(r) \cdot X^i_{[0,t]}) \psi(r)^\top \,dr.
    \end{align*}
    A further application of the usual integration-by-parts formula gives us
    \begin{equation*}
        \int_0^t \mu([s,T]) dX(s) = \mu((t,T])X(t) - \mu([0,T])X(0) + \mu \cdot X_{[0,t]}.
    \end{equation*}
    Therefore, we obtain
    \begin{align*}
        Y(t) & = \phi(t) + \psi(0)X(0) + \mu \cdot X_{[0,t)}                                                                                                          \\
             & \quad + \left(\mu([t,T]) + \int_{t}^{T} \psi(r)B(r, \Ind_{[t,r]}) + \frac{1}{2} A(r, \Ind_{[t,r]}, \psi(r))\, dr\right)X(t)                            \\
             & \quad - \left(\mu([0,T]) + \int_{0}^{T} \psi(r) B(r, \Ind_{[0,r]}) + \frac{1}{2} A(r, \Ind_{[0,r]}, \psi(r))\, dr\right)X(0)                           \\
             & \quad + \int_{t}^{T} \psi(r)\sum_{i=1}^{d} (b^i(r) \cdot X^i_{[0,t)}) +\frac{1}{2} \sum_{i = 1}^{d}\psi(r) (a^i(r) \cdot X^i_{[0,t)}) \psi(r)^\top dr,
        \intertext{which in turn equals with the Riccati equation \eqref{riccati-ode}}
             & = \phi(t) + \psi(0)X(0) + \mu \cdot X_{[0,t)} + \psi(t)X(t) - \psi(0)X(0)                                                                              \\
             & \quad + \int_{t}^{T} \sum_{i=1}^{d} \left(\psi(r) (b^i(r) \cdot X^i_{[0,t)}) +\frac{1}{2} \psi(r) (a^i(r) \cdot X^i_{[0,t)}) \psi(r)^\top \right) dr.
    \end{align*}
    Therefore, the process from equation \eqref{Process-Y} is equal to the one given by \eqref{Y-tilde}.

To prove that $\exp(Y)$ is a local martingale, we rewrite the process $Y$ given by \eqref{Process-Y} in a stochastic
exponential form. 
Thus, let $Y$ be given by \eqref{Process-Y}. Then,
    \begin{align}
        \begin{split}
            Y(t) & = \phi(t) + \psi(0)X(0) +\int_0^t \psi(s) dX(s)                                                                                                     \\
                 & \quad -\int_0^t \sum_{i=1}^{d} \left( \psi(s) ( b^i(s) \cdot X^i_{[0,s]}) + \frac{1}{2} \psi(s) ( a^i(s) \cdot X^i_{[0,s]}) \psi(s)^\top \right)ds.
        \end{split}
    \end{align}
    Hence,
    \begin{align*}
        \begin{split}
            Y(t) & = \phi(t) + \psi(0)X(0) + \int_0^t \psi(s) b(s,X) \, ds + \int_0^t \psi(s)\sigma(s,X) \, dW(s)                                                     \\
                 & \quad -\int_0^t \sum_{i=1}^{d} \left( \psi(s)( b^i(s) \cdot X^i_{[0,s]}) + \frac{1}{2} \psi(s) ( a^i(s) \cdot X^i_{[0,s]}) \psi(s)^\top \right)ds  \\
                 & = \phi(0) - \int_0^t \left( \psi(s) b^0(s) + \frac{1}{2}\psi(s)a^0(s)\psi(s)^\top \right) ds + \psi(0)X(0)                                         \\
                 & \quad + \int_0^t \psi(s) b(s,X_{[0,s]}) \, ds + \int_0^t \psi(s)\sigma(s,X_{[0,s]}) \, dW(s)                                                       \\
                 & \quad -\int_0^t \sum_{i=1}^{d} \left( \psi(s) ( b^i(s) \cdot X^i_{[0,s]}) + \frac{1}{2} \psi(s) ( a^i(s) \cdot X^i_{[0,s]}) \psi(s)^\top \right)ds \\
                 & = \phi(0)+ \psi(0)X(0) + \int_0^t \psi(s)\sigma(s,X_{[0,s]}) \, dW(s) - \frac{1}{2}\int_0^t \psi(s)a(s,X_{[0,s]})\psi(s)^\top \, ds.
        \end{split}
    \end{align*}
Set
\[
    M(t):=\int_0^t \psi(s)\sigma(s,X_{[0,s]})\,dW(s).
\]
Then $M$ is a continuous complex-valued local martingale with quadratic
variation
\[
    \langle M\rangle(t)
    =
    \int_0^t
    \psi(s)a(s,X_{[0,s]})\psi(s)^\top\,ds .
\]
Consequently,
\[
    Y(t)
    =
    Y(0)+M(t)-\frac{1}{2} \langle M \rangle(t),
    \quad
    Y(0)=\phi(0)+\psi(0)X(0).
\]
 Note now that $Y + \frac{1}{2}\langle Y \rangle$ is a local martingale
    and hence $\exp(Y)$ is a local martingale. For the true martingale situation, the remaining claim follows upon observing that by \eqref{Y-tilde},
    $Y(T) = \phi(T) + \mu(\{T\})X(T) + \mu \cdot X_{[0,T)} = \mu \cdot X_{[0,T]}$. Therefore,
    $E\left[\exp\left( \mu \cdot X_{[0,T]} \right) |\, \ccF_t \right] = \exp(Y(t))$ holds for all $t \in [0,T]$.
\end{proof}

The following remarks illustrate how Theorem~\ref{main_thm} relates to the corresponding results in the literature on affine diffusions and affine Volterra processes.

\begin{remark}
    We explain here how Theorem~\ref{main_thm} covers the well-known result on affine diffusions as special case. Let $u \in (\C^d)^\ast$ and assume that $\mu(dr) = u \delta_T(dr)$. Moreover, assume that the measures in Assumption~\ref{ass:affine-coeff} are of the form
    \begin{equation*}
        b^i(t,dr) = b^i(t)\delta_t(dr),\quad a^i(t,dr) = a^i(t)\delta_t(dr), \quad i = 1,\dotsc, d,
    \end{equation*}
    where $b^i : [0,T] \to \R^d$, $a^i : [0,T] \to \bS_d$ and $b^0 : [0,T] \to \R^d$, $a^0 : [0,T] \to \bS_d$.
    As the coefficients depend only on the current state $x(t)$ at each time $t \in [0,T]$ the path-dependent SDE 
    \eqref{path_dependent_sde_reverse_dir} reduces to the classical affine diffusion
    \[
        X(t) = X(0) + \int_0^t \left( b^0(s)+\sum_{i=1}^d X^i(s)b^i(s) \right)ds + \int_0^t \sigma(s,X(s))\,dW(s),
    \]
    with
    \[
        \sigma(s,x)\sigma(s,x)^\top = a^0(s)+\sum_{i=1}^d x^i a^i(s).
    \]
    Denote $B(s,\Ind_{[t,s]}) = B(s) := \big(b^1(s),\dots,b^d(s)\big)$.
    In this case, the Riccati equation \eqref{riccati-ode} reduces to the standard system of Riccati ODEs
    $(\psi^i)_{i \in \{1,\dotsc,d\}}$ in the Markovian setting
    \begin{equation*}
        \psi^i(t) = u^i + \int_{t}^{T} \biggl(\sum_{j = 1}^{d} \psi^j(s)B^{ji}(s) + \frac{1}{2} \psi(s)a^i(s)\psi(s)^\top\biggr) ds, \quad t \in [0,T].
    \end{equation*}
    This is now the usual system of backward Riccati equations for
    time-inhomogeneous affine diffusions.

    Finally, the exponential-affine representation from Theorem~\ref{main_thm}
    also reduces to the classical one
    \[
        Y(t)=\phi(t)+\psi(t)X(t), \quad t\in[0,T],
    \]
    and the conditional Fourier--Laplace transform takes the standard form
    \[
        \E\Bigl[ \exp\big(uX(T)\big) \mid \ccF_t \Bigr] = \exp\bigl(\phi(t)+\psi(t)X(t)\bigr),
    \]
    whenever the corresponding exponential local martingale is a true
    martingale.
\end{remark}

\begin{remark}
    We now explain how Theorem~\ref{main_thm} also contains as a special case the constant volatility inhomogeneous affine Volterra setting considered in \cite{ackermann_Inhomogeneous_2022}.
    To this end, assume that $S = \R^d$, $b^0(t) = 0$, $a^0(t) = \Id_d$, and $a^i \equiv 0$ for all $i = 1,\dots,d$. Moreover, assume that the measures $b^i(t,\cdots)$ are absolutely continuous with respect to the Lebesgue measure, i.e.
    \begin{equation*}
        b^i(t,dr) = \beta^i(t,r)dr,
    \end{equation*}
    where $\beta^i(t,r) \in \R^d$ for all $0 \leq r \leq t \leq T$. Consider also for $u \in (\C^d)^\ast$, $\mu(dr) = u \delta_T(dr)$.

    Then the drift in Assumption~\ref{ass:affine-coeff} reduces to
    \begin{equation*}
        b(t,x) = \sum_{i = 1}^{d} \int_{0}^{t} x^i(r) \beta^i(t,r) \, dr, \quad (t,x) \in \ccC_T^d,
    \end{equation*}
    and the path-dependent SDE \eqref{path_dependent_sde_reverse_dir} becomes
    \begin{equation}\label{eq:volterra-sde-remark}
        X(t) = X(0) + \int_{0}^{t} b(s,X_{[0,s]}) ds + W(t), \quad t \in [0,T].
    \end{equation}
    With these assumptions, the Riccati equation \eqref{riccati-ode} reads
    \begin{equation}\label{riccati-volterra-remark}
        \psi(t) = u + \int_t^T \psi(s) B(s, \Ind_{[t,s]}) \, ds, \quad t \in [0,T],
    \end{equation}
    and the process $Y$ \eqref{Process-Y} takes the form
    \begin{align}\label{process-Y-volterra-remark}
            Y(t) & = \phi(0)+ \psi(0)X(0) + \int_0^t \psi(s)\, dW(s) - \frac{1}{2}\int_0^t \psi(s)\psi(s)^\top \, ds.
    \end{align}
    We can now rewrite \eqref{eq:volterra-sde-remark} as a Volterra equation. Define the matrix-valued kernel $K$ by
$$K(t,r) = \left(\int_{r}^{t} \beta^1(s,r) \,ds \ \cdots \ \int_{r}^{t} \beta^d(s,r) \,ds\right) \in \R^{d \times d}, \quad 0 \leq r \leq t \leq T.$$
With this, the path-dependence of \eqref{eq:volterra-sde-remark} is encoded by the Volterra kernel $K$ and an application of Fubini's theorem yields
\begin{equation*}
    X(t) = X(0) + \int_{0}^{t} K(t,r)X(r) \, dr + W(t).
\end{equation*}
To compare this with \cite{ackermann_Inhomogeneous_2022}, introduce the process $Z(t) := (X(t), W(t))^\top$ with the matrices
    \[
        \overline{K}(t,s) :=
        \begin{pmatrix}
            K(t,s) & \Id_d \\
            0      & \Id_d
        \end{pmatrix},
        \quad B :=
        \begin{pmatrix}
            \Id_d & 0 \\
            0     & 0
        \end{pmatrix},
        \quad \sigma :=
        \begin{pmatrix}
            0   \\
            \Id_d
        \end{pmatrix}.
    \]
The pair $Z(t) := (X(t), W(t))^\top$ can now be written in affine Volterra form as
    \begin{align*}
        Z(t) & =
        \begin{pmatrix}
            X(0) \\
            0
        \end{pmatrix}
        + \int_{0}^{t} \bar{K}(t,s)BZ(s) \,ds + \int_{0}^{t} \bar{K}(t,s) \sigma \, dW(s).
    \end{align*}
    Applying \cite[Theorem 2.1]{ackermann_Inhomogeneous_2022}, with
    $u = (u, 0)$ and $f = 0$, the associated Riccati-Volterra equation is
    \begin{equation*}
        \psi^V(t) = (u, 0) \bar{K}(T,t) + \int_{t}^{T} \psi^V(s) B \bar{K}(s,t) \,ds.
    \end{equation*}
    Writing $\psi^V(t) = (\psi_1^V(t), \psi_2^V(t))$, this is equivalent to
    \begin{equation*}
        \psi_1^V(t) = u K(T, t) + \int_t^T \psi_1^V(s) K(s, t) \, ds, \quad \psi_2^V(t) = u + \int_t^T \psi_1^V(s) \, ds.
    \end{equation*}
 We now identify the second component $\psi_2^V$ with the solution
    $\psi$ of \eqref{riccati-volterra-remark}. To see this, write
    \[
        \beta(t,r) := \big(\beta^1(t,r)\ \cdots\ \beta^d(t,r)\big),
    \]
    so that
    \[
        K(t,r) = \int_r^t \beta(s,r) \, ds.
    \]
    Using Fubini's theorem and the identity
    $\psi_2^V(r)= u + \int_r^T \psi_1^V(s) \, ds$, we obtain
    \[
        \psi_1^V(t) = \int_t^T \psi_2^V(r) \beta(r,t) \, dr,
    \]
    and hence
    \[
        \psi_2^V(t) = u + \int_t^T\int_s^T \psi_2^V(r)\beta(r,s) \,dr \,ds
        = u + \int_t^T \psi_2^V(r) \left(\int_t^r \beta(r,s) \, ds \right)dr.
    \]
    Since $ \int_t^r \beta(r,s)\,ds = B(r,\Ind_{[t,r]})$,
    we conclude that
    \[
        \psi_2^V(t) = u + \int_t^T \psi_2^V(r) B(r,\Ind_{[t,r]}) \, dr.
    \]
 Thus, $\psi_2^V$ solves exactly the Riccati equation
    \eqref{riccati-volterra-remark}. In other words, the Riccati equation
    obtained from the affine Volterra formulation coincides with the Riccati
    equation obtained from Theorem~\ref{main_thm}, after identifying $\psi=\psi_2^V$.
    Finally, under this identification, the martingale exponent from the
    affine Volterra formulation becomes
    \[
        \phi(0)+\psi(0)X(0) + \int_0^t \psi(s) \, dW(s) - \frac{1}{2}\int_0^t \psi(s)\psi(s)^\top \, ds,
    \]
    which agrees with \eqref{process-Y-volterra-remark}.
\end{remark}

\section{Existence of Weak Solutions and Positivity of the Affine Diffusion}\label{section:existence}
We now turn to the existence of path-dependent equations.
Throughout this section the state space of the process is $\R^d$.
For $X(0) \in \R^d$, we consider the path-dependent SDE
\begin{equation}\label{path_dependent_sde_existence}
    X(t) = X(0) + \int_{0}^{t} b(s,X_{[0,s]}) \, ds + \int_{0}^{t} \sigma(s,X_{[0,s]}) \, dW(s).
    \quad 0 \leq t \leq T.
\end{equation}
Here $\sigma : \ccC_T^d \to \R^{d\times m}$ is measurable,
$b : \ccC_T^d \to \R^d$ and $a(t,x) := \sigma(t,x)\sigma(t,x)^\top$ are continuous
with respect to $d_{\ccC}$, as in Section~\ref{sec:preliminaries}.

We are particularly interested in the case where the $b$ and $a$ are affine in the past trajectory, that is, where Assumption
\ref{ass:affine-coeff} holds.
If $a(t,X_{[0,t]}) \in \bS_d^+$ along a solution, then the positive
semidefinite square root $\sqrt{a(t,X_{[0,t]})}$ is well-defined along the
trajectory. In this square-root formulation, the equation can be
written as
\begin{equation}\label{eq:existence-with-affine-coeff}
    X(t) = X(0) + \int_{0}^{t} \left( b^0(u) + \sum_{i=1}^d b^i(u) \cdot X^i_{[0,u]} \right) du
    + \int_{0}^{t} \sqrt{a^0(u) + \sum_{i=1}^d a^i(u) \cdot X^i_{[0,u]}} \, dW(u), \quad 0 \leq t \leq T,
\end{equation}
where in the square-root formulation $W$ is taken $d$-dimensional.

This formulation raises two central questions. First, since the affine
functional $a(t,\cdot)$ need not be positive semidefinite on the whole path
space $C([0,t],\R^d)$, one has to ensure that
$a(t,X_{[0,t]}) \in \bS_d^+$ almost surely along the trajectories of the solution for all $t \in [0,T]$. 
Second, due to the presence of the square root, the diffusion coefficient is generally not locally
Lipschitz continuous. As a result, standard existence results do not apply directly. 
To establish weak existence in this setting, we adapt
arguments from \cite{abijaber_Affine_2019}, originally developed for the
Volterra case. For completeness, we also present a strong existence result
under Lipschitz conditions on $b$ and $\sigma$.

Before showing existence, we provide a structural condition ensuring the non-negativity of the affine diffusion term.
\begin{remark}\label{rem:nontrivial-affine-diffusion-support}
    A nontrivial affine diffusion coefficient cannot simultaneously be positive semidefinite on the whole path space and depend nontrivially on the past trajectory.
    Indeed, fix $t \in [0,T]$ and let
    \begin{equation*}
        a(t,x) = a^0(t) + \sum_{i=1}^d a^i(t) \cdot x^i_{[0,t]},
        \quad x \in C([0,t],\R^d),
    \end{equation*}
    be affine in $x$. Suppose that
    \begin{equation*}
        a(t,x) \in \bS_d^+
        \quad \text{for all } x \in C([0,t],\R^d).
    \end{equation*}
    Then, since the path space $C([0,t],\R^d)$ is linear, also $-x \in C([0,t],\R^d)$ for every $x$, and hence
    \begin{equation*}
        a^0(t) + \sum_{i=1}^d a^i(t)\cdot x^i_{[0,t]} \in \bS_d^+, \quad a^0(t) - \sum_{i=1}^d a^i(t)\cdot x^i_{[0,t]} \in \bS_d^+.
    \end{equation*}
    By scaling $x \mapsto \lambda x$, $\lambda > 0$, this is only possible if
    \begin{equation*}
        \sum_{i=1}^d a^i(t)\cdot x^i_{[0,t]} = 0 \quad \text{for all } x \in C([0,t],\R^d).
    \end{equation*}
    Thus, $a(t,x) = a^0(t)$ is independent of the path, and the affine dependence is necessarily trivial.
    Consequently, in order to obtain a genuinely path-dependent affine diffusion coefficient, positivity of $a$ cannot be required on the entire linear path space. Instead, one must only require that
    \begin{equation*}
        a(t,X_{[0,t]}) \in \bS_d^+ \quad \text{for all } t \in [0,T], \quad P\text{-a.s.},
    \end{equation*}
    along the realized trajectories of the process $X$.
\end{remark}
To this end, we assume that the pairs $(a^i, X)$ satisfy a semimartingale differentiability similarly to Definition~\ref{def:semimartingale-diff} for each $i \in \{1,\dotsc,d\}$.
Since the diffusion term $a$ is in general a $d \times d$ matrix, the semimartingale representation takes the form
\begin{equation}\label{cond:for-diffusion-part}
    a^i(t) \cdot X^i_{[0,t]} = a^i(0)\cdot X^i_{[0,0]} - \int_{0}^{t} \Gamma^i(s,X^i_{[0,s]}) \, ds + \int_{0}^{t} \Pi^i(s) \, dX^i(s), \quad t \in [0,T],
\end{equation}
where $\Gamma^i : \ccC_T^1 \to \R^{d \times d}$ is an affine non-anticipative functional and $\Pi^i : [0,T] \to \R^{d \times d}$, for each $i \in \{1,\dotsc,d\}$.

In contrast to a Markov process it is not enough to analyze the process $X$ at the boundary but rather give inward pointing conditions directly on the affine variance term.
Therefore, we have two different regimes to explore. In the case when $d = 1$, the variance term is scalar valued, and we have to ensure that it stays in the cone $\R_{\geq 0}$.
The case $d \geq 2$ is more involved, since in this case $a(t,X)$ takes values in $\bS_d^+$ and the boundary is more complex.

We start with the one-dimensional case $d = 1$. Let $X$ be an affine one-dimensional path-dependent process.
In this case the diffusion term admits the semimartingale representation
\begin{equation}\label{eq:diffusion-semimartingale-repr-dim-1}
    da(t,X_{[0,t]}) = \left[\frac{d}{dt}a^0(t)\right]dt + \Pi(t) \, dX(t) - \Gamma(t,X_{[0,t]}) \,dt.
\end{equation}

If $\sigma$ denotes the unique positive semidefinite square root of $a$, then notice that on the boundary set $\{a(t,X_{[0,t]}) = 0\}$ Equation \eqref{eq:diffusion-semimartingale-repr-dim-1} reduces to
\begin{equation}
    da(t,X_{[0,t]}) = \left[\frac{d}{dt}a^0(t)\right] + \left( \Pi(t)b(t,X_{[0,t]}) - \Gamma(t,X_{[0,t]}) \right) dt \quad dt \otimes dP \text{-a.s.}
\end{equation}
Therefore, on this set, only the finite-variation drift
determines whether the process can move into the negative half-line. This
is the usual one-dimensional stochastic invariance condition. Thus, it is sufficient to require that
\begin{equation}
    \frac{d}{dt} a^0(t) +\Pi(t) b(t,X_{[0,t]}) -\Gamma(t,X_{[0,t]}) \geq 0 \quad \text{ on } \{a(t,X_{[0,t]})=0\}, \quad dt \otimes dP\text{-a.s.}
\end{equation}
We say in this case that $a(t,X_{[0,t]})$ is non-decreasing on the zero
set.

\begin{lemma}\label{lm:non-negativ-diffusion-dim-1}
    Let $d = 1$.
    Assume that Assumption \ref{ass:affine-coeff} holds, that $a^0(t)$ is continuously differentiable and that the pair $(a^1,X)$ is semimartingale differentiable as in \eqref{cond:for-diffusion-part} with coefficients $(\Gamma, \Pi)$, where $\Gamma$ is by assumption affine with representation $\Gamma(t,x) = \Gamma^0(t) + \Gamma^1(t) \cdot x_{[0,t]}$ for $(t,x) \in \ccC_T^1$.
    Suppose that $a(0,x(0)) \geq 0$ and there exists measurable functions $\lambda : [0,T] \to \R$, $r : [0,T] \to \R_{\geq 0}$ such that for all $t \in [0,T]$
    \begin{align}
        \left[\frac{d}{dt}a^0(t)\right] + \Pi(t)b^0(t) - \Gamma^0(t) & = \lambda(t)a^0(t) + r(t), \quad \forall t \in [0,T], \label{eq:matrix-condition-1-dim-1}                               \\
        \Pi(t)(b^1(t) \cdot x_{[0,t]}) - \Gamma^1(t) \cdot x_{[0,t]} & = \lambda(t)\bigl( a^1(t) \cdot x_{[0,t]} \bigr), \quad \forall (t,x) \in \ccC_T^1. \label{eq:matrix-condition-2-dim-1}
    \end{align}
    Then $a(t,X_{[0,t]})$ is non-decreasing on the zero set for all $t \in [0,T]$.
\end{lemma}

\begin{proof}
    By Assumption \ref{ass:affine-coeff}, the drift of the one-dimensional process $X$ has the form
    \begin{equation*}
        b(t,X_{[0,t]})=b^0(t) + b^1(t)\cdot X_{[0,t]} \in \R,
    \end{equation*}
    with $b^0(t) \in \R$ and $b^1(t) \in \ccM([0,t],\R)$.
    By the semimartingale differentiability assumption on $(a^1, X)$, there exists $\R$-valued functionals $\Gamma$ and $\Pi$, such that
    \begin{equation*}
        d \bigl( a^1(t) \cdot X_{[0,t]} \bigr) = - \Gamma(t, X_{[0,t]}) \, dt + \Pi(t) \, dX(t).
    \end{equation*}
    Thus, in the same manner as in Equation \eqref{eq:diffusion-semimartingale-repr-dim-1} we get
    \begin{equation*}
        da(t,X_{[0,t]}) = \left[\frac{d}{dt}a^0(t)\right] dt + \Pi(t) \, dX(t) - \Gamma(t, X_{[0,t]}) \, dt.
    \end{equation*}
    Therefore, at the boundary $a(t,X_{[0,t]}) = 0$, we get $dt \otimes dP$-almost surely
    \begin{align*}
        da(t,X_{[0,t]}) & = \left[\frac{d}{dt}a^0(t)\right] dt + \bigl(\Pi(t) b(t,X_{[0,t]}) - \Gamma(t, X_{[0,t]}) \bigr) dt                                                            \\
        % & = \left[\frac{d}{dt}a^0(t)\right] dt + \biggl(\Pi(t) \Bigl(b^0(t) + b^1(t)\cdot X_{[0,t]}\Bigr) - \Gamma(t, X_{[0,t]}) \biggr) dt\\
                        & = \left[\frac{d}{dt}a^0(t)\right] dt + \bigl(\Pi(t)b^0(t) - \Gamma^0(t)\bigr) dt + \big(\Pi(t) (b^1(t) \cdot X_{[0,t]}) - \Gamma^1(t) \cdot X_{[0,t]}\big) dt.
    \end{align*}
    Therefore, non-negativity requires for the drift of $a$
    \begin{equation}\label{eq:non-negative-diffusion-cond-dim-1}
        \left[\frac{d}{dt}a^0(t)\right] + \bigl(\Pi(t) b^0(t) - \Gamma^0(t)\bigr) + \big(\Pi(t) (b^1(t) \cdot X_{[0,t]}) - \Gamma^1(t) \cdot X_{[0,t]}\big) \geq 0, \quad dt \otimes dP\text{-a.s} \text{ on } a(t,X) = 0.
    \end{equation}

    If $\eqref{eq:matrix-condition-1-dim-1}\&\eqref{eq:matrix-condition-2-dim-1}$ hold then \eqref{eq:non-negative-diffusion-cond-dim-1} is satisfied and $a(t,X_{[0,t]})$ is non-decreasing on the zero set for all $t \in [0,T]$. To be able to derive a sufficient condition that ensures
    \eqref{eq:non-negative-diffusion-cond-dim-1}, we check that this condition holds for all $(t,x) \in \ccC_T^1$.
    Denote the drift of $a$ by
    $$B_a(t,x_{[0,t]}) = \left[\frac{d}{dt}a^0(t)\right] + \Pi(t)b^0(t) - \Gamma^0(t) + \Pi(t) \bigl(b^1(t) \cdot x_{[0,t]}\bigr) - \Gamma^1(t) \cdot x_{[0,t]}$$
    and notice that $B_a$ is affine. Define the affine hyperplane
    $$\cH_t = \{x \in C([0,t],\R) : a(t,x_{[0,t]}) = 0\}.$$
    We want to give conditions under which $B_a(t,x_{[0,t]}) \geq 0$ for all $t \in [0,T]$ and $x \in \cH_t$. An affine functional which is non-negative on the affine hyperplane $\cH_t$ at time $t \in [0,T]$ must be constant on that hyperplane. This in particular implies that the linear part of the functional $B_a$ is given by the linear part of $a$ multiplied by a real-valued scalar. Denote those linear parts by $m$ and $\ell$ respectively, i.e. $a(t,x) = a^0(t) + \ell(t,x)$ and $B_a(t,x) = B_a^0(t) + m(t,x)$.
    Since $m(t,x) = \lambda(t) \ell(t,x)$ we get
    \begin{equation*}
        B_a(t,x) = B_a^0(t) + \lambda(t) \ell(t,x) = \lambda(t)(a^0(t) + \ell(t,x)) + (B_a^0(t)-\lambda(t)a^0(t)) = \lambda(t) a(t,x) + r(t),
    \end{equation*}
    where $r(t) = B_a^0(t) - \lambda(t)a^0$. On the boundary $a(t,x) = 0$ it holds that $B_a(t,x) = r(t)$ and by assumption this implies $r(t) \geq 0$ on $\cH_t$.

    Thus, we conclude that the condition $B_a(t,x_{[0,t]}) \geq 0$ for all $(t,x) \in \ccC_T^1$ with $a(t,x) = 0$ is guaranteed, for instance, if there exist measurable functions $\lambda : [0,T] \to \R$ and $r : [0,T] \to \R_{\geq 0}$ such that
    \begin{equation*}
        B_{a}(t,x_{[0,t]}) = \lambda(t) a(t,x_{[0,t]}) + r(t), \quad \forall (t,x) \in \ccC_T^1.
    \end{equation*}
    This in particular implies that for all $(t,x) \in \ccC^1_T$ with $a(t,x_{[0,t]}) = 0$
    \begin{align*}
        \left[\frac{d}{dt}a^0(t)\right] + \Pi(t) b^0(t) - \Gamma^0(t) & = \lambda(t)a^0(t) + r(t)             \\
        \Pi(t)(b^1(t) \cdot x_{[0,t]}) - \Gamma^1(t) \cdot x_{[0,t]}  & = \lambda(t)(a^1(t) \cdot x_{[0,t]}).
    \end{align*}
\end{proof}

\begin{remark}
    On the boundary $a(t,x) = 0$ only the scalar term $r(t)$ remains. Thus, if $r(t) = 0$, the boundary is tangent, if $r(t) > 0$ its strictly inward pointing and if $r(t) < 0$ the drift is strictly outward pointing.
\end{remark}

Let us discuss now what changes in the case $d \geq 2$. The affine
diffusion coefficient is then matrix-valued, and the relevant state space for
the variance process is the cone $\bS_d^+$. As before let
\begin{equation*}
    Y(t) :=a(t,X_{[0,t]}) = a^0(t)+\sum_{i=1}^d a^i(t)\cdot X^i_{[0,t]}, \quad t \in [0,T].
\end{equation*}
Assume, at least formally, that $a^0$ is absolutely continuous and that the
pairs $(a^i,X^i)$ satisfy a semimartingale differentiability condition of the
form
\begin{equation*}
    a^i(t)\cdot X^i_{[0,t]} = a^i(0)\cdot X^i_{[0,0]} - \int_0^t \Gamma^i(s,X^i_{[0,s]}) \, ds + \int_0^t \Pi^i(s) \, dX^i(s), \quad t \in [0,T],
\end{equation*}
where $\Gamma^i(s,\cdot)$ and $\Pi^i(s)$ take values in
$\bS_d$. The symmetry assumption is natural here, since $Y(t)$ is
supposed to remain a symmetric matrix.
Thus, we obtain formally
\begin{equation*}
    dY(t) = \cB(t,X_{[0,t]}) \, dt + \sum_{k=1}^d \cH^k(t,X_{[0,t]}) \, dW^k(t),
\end{equation*}
with
\begin{equation*}
    \cB(t,x) = \left[\frac{d}{dt}a^0(t)\right] + \sum_{i=1}^d \Bigl(\Pi^i(t)b^i(t,x) - \Gamma^i(t,x^i_{[0,t]}) \Bigr),
\end{equation*}
and
\begin{equation*}
    \cH^k(t,X_{[0,t]}) = \sum_{i=1}^d \Pi^i(t) e_i^\top \sqrt{Y(t)} e_k.
\end{equation*}
The viability of $Y$ in $\bS_d^+$ should then be understood in terms
of the geometry of the cone. If $Y \in \partial\bS_d^+$, then
$\ker(Y) \neq \{0\}$. The tangent cone at a point $H$ is characterized by the set
\begin{equation*}
    T_{\bS_d^+}(H) := \{M \in \bS_d : v^\top M v \geq 0,\ v \in \text{Ker}(H)\}.
\end{equation*}
Thus, a natural inward-pointing drift condition is
\begin{equation*}
    \cB(t,X_{[0,t]}) \in T_{\bS_d^+}(Y(t)),
\end{equation*}
or equivalently
\begin{equation*}
    v^\top \cB(t,X_{[0,t]})v \geq 0 \quad \text{for all } v\in\ker(Y(t)).
\end{equation*}
In terms of the positive normal cone
\begin{equation*}
    N^+_{\bS_d^+}(Y) := \{U\in\bS_d^+:\ UY=0\},
\end{equation*}
this can be written as
\begin{equation*}
    \Tr(U B(t,X_{[0,t]})) \geq 0 \quad \text{for all } U\in N^+_{\bS_d^+}(Y(t)).
\end{equation*}
In addition, the martingale part has to be compatible with the boundary of the
cone. A necessary tangency condition in the normal directions is
\begin{equation*}
    \Tr\big(U\,\cH^k(t,X_{[0,t]})\big)=0, \quad U\in N^+_{\bS_d^+}(Y(t)),\quad k=1,\dots,d.
\end{equation*}
Equivalently, for every $v\in\ker(Y(t))$
\begin{equation*}
    \sum_{i=1}^d v^\top \Pi^i(t)v e_i^\top \sqrt{Y(t)}e_k = 0, \quad k=1,\dots,d
\end{equation*}
which corresponds to the condition
\begin{equation*}
    \big(v^\top \Pi^1(t)v,\dots,v^\top \Pi^d(t)v\big)^\top \in \ker(Y(t)).
\end{equation*}
This shows why the multidimensional case is substantially more delicate than
the scalar one. In dimension one, the boundary of the state space consists only
of the point $0$, and the condition $Y(t)=0$ defines a single affine
constraint. This is what allows the hyperplane argument in Lemma
\ref{lm:non-negativ-diffusion-dim-1}. For $d\geq 2$, however, the boundary of
$\bS_d^+$ is more complex. The relevant
boundary conditions are therefore not encoded by one affine hyperplane, but by
the whole family of normal directions.
Consequently, the scalar argument does not directly extend to the matrix-valued
case. A full multidimensional invariance result would require a separate
analysis of these tangent and normal cone conditions, and we do not pursue this
in the present paper.

\begin{remark}\label{remark:dimension-d-1-well-posedness}
    Lemma \ref{lm:non-negativ-diffusion-dim-1} holds similarly true if we assume that, the linear functional $a^1(t) \cdot x_{[0,t]} \in \C^{1,2}_b$ for all $(t,x) \in \ccC_T^d$, as seen in Example~\ref{ex:functional-ito-semimartingale-diff}.
\end{remark}

Lemma~\ref{lm:non-negativ-diffusion-dim-1} implicitly assumes in the one-dimensional case that the diffusion part $a(t,X) = a^0(t) + a^1(t) \cdot X_{[0,t]}$ is a path-dependent CIR process.

\begin{proposition}\label{prop:CIR-reduction}
    In the setting of Lemma~\ref{lm:non-negativ-diffusion-dim-1},
    it holds that for every $t \in [0,T]$
    \begin{equation}\label{eq:variance-CIR-1-dim}
        da(t,X_{[0,t]}) = \bigl(\lambda(t)a(t,X_{[0,t]})+r(t)\bigr) dt + \Pi(t) \sqrt{a(t,X_{[0,t]})} dW(t).
    \end{equation}
    Furthermore, let $X$ solve the path-dependent equation \eqref{eq:existence-with-affine-coeff} with diffusion coefficient
    $\sqrt{a(t,X_{[0,t]})^+}$. If $a(0,X(0)) \geq 0$, then $a(t,X_{[0,t]}) \geq 0$ for all $t \in [0,T]$ a.s.
    Consequently, $X$ is a solution of the original square-root equation. 
\end{proposition}

\begin{proof}
    As already established in Lemma~\ref{lm:non-negativ-diffusion-dim-1} the differential of $a(t,X)$ is given by
    \begin{align*}
        da(t,X_{[0,t]}) & = \left[\frac{d}{dt}a^0(t)\right] \, dt + \Pi(t) dX(t) - \Gamma(t, X_{[0,t]}) \, dt                                                   \\
                        & = \left(\left[\frac{d}{dt}a^0(t)\right] + \Pi(t)b(t,X_{[0,t]}) - \Gamma(t, X_{[0,t]}) \right)dt + \Pi(t) \sqrt{a(t,X_{[0,t]})} dW(t).
    \end{align*}
    By assumption, there exists functions $\lambda : [0,T] \to \R$ and $r : [0,T] \to \R_{\geq 0}$ such that
    \begin{align*}
        \left[\frac{d}{dt}a^0(t)\right] + \Pi(t)b^0(t) - \Gamma^0(t) & = \lambda(t)a^0(t) + r(t), \quad \forall t \in [0,T],                               \\
        \Pi(t)(b^1(t) \cdot x_{[0,t]}) - \Gamma^1(t) \cdot x_{[0,t]} & = \lambda(t)\bigl( a^1(t) \cdot x_{[0,t]} \bigr), \quad \forall (t,x) \in \ccC_T^1.
    \end{align*}
    Thus,
    \begin{equation*}
        da(t,X_{[0,t]}) = \bigl(\lambda(t)a(t,X_{[0,t]})+r(t)\bigr) dt + \Pi(t) \sqrt{a(t,X_{[0,t]})} dW(t).
    \end{equation*}
    Finally, note that $r : [0,T] \to \R_{\geq 0}$ and thus the stochastic invariance of the CIR process \eqref{eq:variance-CIR-1-dim} follows from standard results which proves the remaining claim.
\end{proof}

We now show that general path-dependent SDEs given by \eqref{path_dependent_sde_existence} admit a weak solution, under a linear growth assumption.
The existence of strong solutions under Lipschitz conditions can be found in Appendix~\ref{app:Auxiliary}.

\begin{theorem}\label{thm:weak-existence}
    Assume that the non-anticipative functionals $b$ and $\sigma$ in \eqref{path_dependent_sde_existence} are continuous and satisfy the linear growth Assumption \ref{LG-Condition}.
    Then, for any initial condition $X(0) \in \R^d$, there exists a weak solution $X = (X(t))_{t \in [0,T]}$ to equation \eqref{path_dependent_sde_existence}.
\end{theorem}
To show weak existence we need the following lemmas.

\begin{lemma}\label{lm:tightness}
    Let $\cX$ denote the set of all continuous processes $X$ solving \eqref{path_dependent_sde_existence} with a fixed initial condition $X(0) \in \R^d$,
    and with coefficients $b$ and $\sigma$ that are non-anticipative, continuous and satisfy the linear growth Assumption \ref{LG-Condition} with a fixed constant $C_{LG}$.
    Then the family of laws $\{\cL (X) : X \in \cX \}$ is tight in $C([0,T], \R^d)$.
\end{lemma}

\begin{proof}
    Fix $m \geq 1$ and let $X \in \cX$. By Lemma~\ref{lm:hoelder-continuity}, for any $\alpha \in [0, \frac{1}{2} - \frac{1}{2m})$ there exists a constant $c > 0$ depending only on $m, T, C_{LG}, \|X(0)\|^{2m}$, but is uniform over all $X \in \cX$, such that for the $\alpha$-Hölder seminorm $|\cdot|_{C^\alpha}$,
    \begin{equation}
        E\bigl[|X|_{C^\alpha}^{2m}\bigr] = E\left[\left(\sup_{0 \leq s < t \leq T} \frac{\|X(t)-X(s)\|}{|t-s|^\alpha}\right)^{2m}\right] \leq c.
    \end{equation}
    By Markov's inequality, for any $R' > 0$,
    \begin{equation*}
        \sup_{X \in \cX} P(|X|_{C^\alpha} \geq R') \leq \frac{c}{(R')^{2m}}.
    \end{equation*}
    Since $X(0)$ is deterministic, we have $\sup_{t \in [0,T]} \|X(t)\| \leq \|X(0)\| + T^\alpha |X|_{C^\alpha}$.
    Hence, if we choose $R = \|X(0)\| + T^\alpha R'$ and denoting by $B_R^\alpha$ the closed ball of radius $R$ w.r.t.~$|\cdot|_{C^\alpha}$, we obtain
    \begin{equation*}
        \sup_{X \in \cX} P \left( X \notin B_R^\alpha \right) \leq \sup_{X \in \cX} P(|X|_{C^\alpha} \geq R') \leq  \frac{c}{(R')^{2m}}.
    \end{equation*}
    For a given $\varepsilon > 0$, choose $R' = (\tfrac{c}{\varepsilon})^{\tfrac{1}{2m}}$, such that
    \begin{equation*}
        \sup_{X \in \cX} P \left( X \notin B_R^\alpha \right) \leq \varepsilon.
    \end{equation*}
    Finally, since closed balls in $C^\alpha([0,T], \R^d)$ are compact in $C([0,T], \R^d)$ by the Arzelà-Ascoli theorem, we conclude that the family $\cX$ is tight in $C([0,T], \R^d)$.

\end{proof}
For the following lemma, it is convenient to realize the non-anticipative coefficients on the fixed space $[0,T] \times C([0,T],\R^d)$ rather than on the stopped-path space $\ccC_T^d$. To this end, we define the lifted maps
\begin{equation}\label{eq:lifted-maps}
    \overline{b} : [0,T] \times C([0,T], \R^d) \to \R^d, \quad \overline{\sigma} : [0,T] \times C([0,T], \R^d) \to \R^{d \times m},
\end{equation}
by
\begin{equation*}
    \overline{b}(t,x) = b(t,x|_{[0,t]}),\quad \overline{\sigma}(t,x) = \sigma(t,x|_{[0,t]}).
\end{equation*}
In particular, $\overline{b}(t,x)$ and $\overline{\sigma}(t,x)$ depend only on the restriction of $x \in C([0,T],\R^d)$ to $[0,t]$. Thus, the lifted coefficients remain non-anticipative while being represented on the fixed path space $C([0,T],\R^d)$.

\begin{lemma}\label{lm:stability-of-weak-sol}
    For each $n \in \N$, let $X^n$ be a weak solution of
    \begin{equation*}
        X^n(t) = X^n(0) + \int_0^t \overline{b}^n(s,X^n) \, ds + \int_0^t \overline{\sigma}^n(s,X^n) \, dW^n(s), \quad t\in[0,T],
    \end{equation*}
    where $\overline{b}^n : [0,T] \times C([0,T],\R^d) \to \R^d $ and $\overline{\sigma}^n : [0,T] \times C([0,T],\R^d) \to \R^{d\times m}$ are continuous non-anticipative functionals satisfying the linear growth Assumption~\ref{LG-Condition} with a common constant $C_{LG}$.
    Suppose further that $\overline{b}^n \to \overline{b}$, $\overline{\sigma}^n \to \overline{\sigma}$,
    locally uniformly on $[0,T]\times C([0,T],\R^d)$, where $\overline{b}$ and $\overline{\sigma}$ are the lifted coefficients defined above, and that $X^n \Rightarrow X$ for some continuous process $X$.
    Then $X$ is a weak solution of \eqref{path_dependent_sde_existence}.
\end{lemma}

\begin{proof}
    For each $n \in \N$, define the continuous local martingale
    \begin{equation*}
        M^n(t) = X^n(t) - X^n(0) - \int_0^t \overline{b}^n(s,X^n)\,ds = \int_0^t \overline{\sigma}^n(s,X^n) \, dW^n(s), \quad t \in [0,T].
    \end{equation*}
    Joint continuity and the locally uniform convergence of $\overline{b}^n$ and $\overline{\sigma}^n$ now imply with the continuous mapping theorem that $M^n \Rightarrow M$ for some limit $M$ and
    \begin{equation*}
        \langle M^n \rangle = \int \overline{\sigma}^n(t,X^n)  \overline{\sigma}^{n} (t,X^n)^\top \, dt \, \Rightarrow \int \overline{\sigma}(t,X)  \overline{\sigma} (t,X)^\top \,dt, \quad \int \overline{b}^n(t,X^n) \, dt \, \Rightarrow \int \overline{b}(t,X) \, dt.
    \end{equation*}
    Fix any $0 \leq s < t \leq T$, $m\in\N$, times $0 \leq t_1 \leq \dotsc \leq t_m \leq s$, and any bounded continuous function
    $f:\R^{d \times m} \to \R$. Observe that the moment bound in Lemma \ref{lm:estimate-existence} is uniform in $n$ since the
    $X^n$ satisfy the linear growth condition \ref{LG-Condition} with a common constant.
    Since $M^n$ is a martingale w.r.t. the filtration generated by $X^n$, using \cite[Theorem 3.5]{billingsley_Convergence_1999}, one then show that
    \begin{equation*}
        E\big[f(X(t_1),\dots,X(t_m))\,(M(t)-M(s))\big] = \lim_{n \to \infty} E\bigl[f(X^n(t_1),\dotsc,X^n(t_m))\,(M^n(t)-M^n(s))\bigr]=0,
    \end{equation*}
    and similarly for the increments $M_i^nM_j^n - \langle M_i^n, M^n_j \rangle$.
    Therefore, $M$ is a martingale with respect to the filtration generated by $X$ with quadratic variation $\langle M \rangle = \int \overline{\sigma}(t,X)\overline{\sigma}(t,X)^\top dt$.
    This carries over to the usual augmentation.
    Enlarging the probability space if necessary, we may construct a $m$-dimensional Brownian motion $\overline{W}$ such that
    \begin{equation*}
        M(t) = \int_0^t \overline{\sigma}(s,X) \, d\overline{W}(s).
    \end{equation*}
    Therefore,
    \begin{equation*}
        X(t) = X(0) + \int_0^t \overline{b}(s,X) \, ds + \int_0^t \overline{\sigma}(s,X) \, d\overline{W}(s),
    \end{equation*}
    solves \eqref{path_dependent_sde_existence} with the Brownian motion $\overline{W}$.
\end{proof}

\begin{proof}[Proof of Theorem \ref{thm:weak-existence}]
    We first approximate the coefficients directly on the stopped-path space $\ccC_T^d$. Since $b$ and $\sigma$ are continuous non-anticipative functionals satisfying the linear growth condition \ref{LG-Condition}, Lemma~\ref{lem:nonanticipative-lipschitz-approximation} yields sequences of continuous non-anticipative functionals
    \begin{equation*}
        b^n : \ccC_T^d \to \R^d, \quad \sigma^n: \ccC_T^d \to \R^{d\times m}, \quad n \in \N,
    \end{equation*}
    such that each $b^n$, $\sigma^n$ is Lipschitz in the path variable uniformly in time, satisfies the linear growth condition with a common constant $C'_{LG}>0$ independent of $n$, and
    $b^n\to b$, $\sigma^n\to \sigma$ locally uniformly on $\ccC_T^d$.

    For each $n \in \N$, we lift these coefficients to the fixed path space by setting
    \begin{equation*}
        \overline{b}^n(t,x) := b^n(t,x|_{[0,t]}), \quad \overline{\sigma}^n(t,x) := \sigma^n(t,x|_{[0,t]}),
    \end{equation*}
    and likewise
    \begin{equation*}
        \overline{b}(t,x) := b(t,x|_{[0,t]}), \quad \overline{\sigma}(t,x) := \sigma(t,x|_{[0,t]}),
    \end{equation*}
    for $(t,x)\in [0,T] \times C([0,T],\R^d)$.
    By construction, the lifted coefficients $\overline{b}^n,\overline{\sigma}^n$ remain continuous and non-anticipative on $[0,T] \times C([0,T],\R^d)$. Moreover, since $b^n \to b$ and $\sigma^n \to \sigma$ locally uniformly on $\ccC_T^d$, we also have
    \begin{equation*}
        \overline{b}^n \to \overline{b}, \quad \overline{\sigma}^n \to \overline{\sigma}
    \end{equation*}
    locally uniformly on $[0,T] \times C([0,T],\R^d)$.

    Now fix $n \in \N$. Since $b^n$ and $\sigma^n$ satisfy the Lipschitz and linear growth assumptions of Theorem~\ref{thm:strong-existence}, there exists a unique strong solution $X^n$ of
    \begin{equation*}
        X^n(t) = X^n(0) + \int_0^t b^n(s,X^n_{[0,s]}) \,ds + \int_0^t \sigma^n(s,X^n_{[0,s]}) \,dW^n(s), \quad t \in [0,T].
    \end{equation*}
    Equivalently, in terms of the lifted coefficients,
    \begin{equation*}
        X^n(t) = X^n(0) + \int_0^t \overline{b}^n(s,X^n) \,ds + \int_0^t \overline{\sigma}^n(s,X^n) \,dW^n(s), \quad t \in [0,T].
    \end{equation*}

    Since the coefficients $b^n$, $\sigma^n$ satisfy the linear growth condition with a common constant $C'_{LG}$, Lemma~\ref{lm:tightness} implies that the family $\{X^n\}_{n \in \N}$ is tight in $C([0,T],\R^d)$. Hence, by Prokhorov's theorem, there exists a subsequence, again denoted by $\{X^n\}_{n \in \N}$, and a continuous process $X$ such that
    \begin{equation*}
        X^n \Rightarrow X \quad \text{ in } C([0,T],\R^d).
    \end{equation*}

    Applying Lemma~\ref{lm:stability-of-weak-sol} to the lifted coefficients, we conclude that $X$ is a weak solution of
    \begin{equation*}
        X(t) = X(0) + \int_0^t \overline{b}(s,X) \,ds + \int_0^t \overline{\sigma}(s,X) \,dW(s), \quad t \in [0,T],
    \end{equation*}
    for some Brownian motion $W$.
    Finally, by definition of the lift,
    \begin{equation*}
        \overline{b}(t,X) = b(t,X_{[0,t]}),\quad
        \overline{\sigma}(t,X) = \sigma(t,X_{[0,t]}),
    \end{equation*}
    and therefore $X$ is a weak solution of the original path-dependent SDE \eqref{path_dependent_sde_existence}. This proves the claim.
\end{proof}

\section{Examples and applications}\label{sec:examples}

\subsection{Affine processes with non-standard state spaces}
In this section we consider an affine process with path-dependent volatility whose state space depends on the realized path. This situation is surprisingly different to classical affine processes, where a standard choice of the state space is $\R_{\geq 0} ^m \times \R^n$, which in the path-dependent case will no longer be typical.

To illustrate this, we consider a weak solution of the SDE
\begin{align}\label{eq:pd-vol-model-general}
    dX(t) & =  \Big(\kappa(t) - c'(t) + X(t) - \Ind_{\{t > 1\}}X(t-1) \Big) \, dt
    +  \sqrt{c(t) + X(t) - \int_{(t-1)^+}^{t} X(s) \, ds} \, dW(t),
\end{align}
with $\kappa \in C([0,T], \R_{\geq 0})$, $c \in C^1([0,T], \R)$ and $X(0) \in \R$ with $c(0) + X(0) \geq 0$ where $c'(t)$ denotes the time-$t$ derivative of $c$.

This model introduces path-dependence through a rolling window. It
satisfies the conditions stated in Lemma~\ref{lm:non-negativ-diffusion-dim-1} and ensures non-negativity of the affine variance functional along solutions. Observe that $X$ is a path-dependent affine process with
non-anticipative functional drift $b$ and diffusion coefficient $\sigma^2 = a$ satisfying
\begin{align}\label{parameters:pd-vol}
\begin{split}
b(s,x) & = b^0(s) + \int_{[0,s]} x(u)\, b^1(s,du)
= \kappa(s)-c'(s)+x(s)-\Ind_{\{s>1\}}x(s-1),      \\
a(s,x) & = a^0(s) + \int_{[0,s]} x(u)\, a^1(s,du)
= c(s) + x(s)  -\int_{(s-1)^+}^{s}x(u) \,du ,
\end{split}
\end{align}
for $(s,x) \in \ccC_T^1$.
The corresponding coefficients are therefore
\begin{equation*}
    b^0(s)=\kappa(s)-c'(s), \quad a^0(s)=c(s),
\end{equation*}
and
\begin{equation*}
    b^1(s,du)=\delta_s(du)-\Ind_{\{s>1\}}\delta_{s-1}(du), \quad
    a^1(s,du)=\delta_s(du)-\Ind_{((s-1)^+,s]}(u)\,du.
\end{equation*}

The drift has been chosen in such a way that the affine variance term $a(t,X_{[0,t]})$ remains non-negative for all possible trajectories of the underlying process $X$. Indeed, notice first that the process $(a^1(t) \cdot X_{[0,t]})_{t \in [0,T]}$ is semimartingale differentiable in the sense of Definition \ref{def:semimartingale-diff} with coefficients $(\Gamma, 1)$, where the affine functional
$$\Gamma(t,x_{[0,t]}) = \Gamma^0(t) + \Gamma^1(t) \cdot x_{[0,t]}$$
is given by
\begin{equation*}
    \Gamma^0(t) \equiv 0, \quad \Gamma^1(t,du) = \delta_t(du) - \Ind_{\{t > 1\}}\delta_{t-1}(du).
\end{equation*}
Hence, $\Gamma(t,x_{[0,t]}) = x(t) - x(t-1)\Ind_{\{t > 1\}}$ for all $(t,x) \in \ccC_T^1$.
By Lemma~\ref{lm:non-negativ-diffusion-dim-1}, it is therefore enough to choose an affine drift of the form $b(t,x) = b^0(t) + \int_{[0,t]} x(u) \, b^1(t,du)$,
such that there exist functions $\lambda : [0,T] \to \R$ and $r : [0,T] \to \R_{\geq 0}$ satisfying, for all $(t,x) \in \ccC_T^1$,
\begin{align}\label{pd-vol-condition}
    \begin{split}
        b^0(t)                 & = \lambda(t)c(t) + r(t) - c'(t),                                                                 \\
        b^1(t) \cdot x_{[0,t]} & = x(t) - x(t-1)\Ind_{\{t > 1\}} + \lambda(t) \biggl(x(t) - \int_{(t-1)^+}^{t} x(s) \, ds\biggr).
    \end{split}
\end{align}
This leads to the following two natural specifications.
\begin{enumerate}
    \item Taking $\lambda\equiv 0$, $r(t)=\kappa(t)$, and hence
          \begin{equation*}
              b^0(t) = \kappa(t) - c'(t), \quad b^1(t,du) = \delta_t(du) - \Ind_{\{t>1\}} \delta_{t-1}(du),
          \end{equation*}
          satisfies \eqref{pd-vol-condition} and yields the model
          \begin{equation*}
              X(t) = X(0) + \int_0^t \Bigl(\kappa(s) - c'(s) + X(s) - \Ind_{\{s>1\}}X(s-1)\Bigr) \,ds
              +\int_0^t \sqrt{c(s) + X(s) - \int_{(s-1)^+}^{s}X(u) \,du} \,dW(s).
          \end{equation*}
          This specification will be used as the canonical model in this subsection.
    \item Let $\theta: [0,T] \to \R_{\geq 0}$. Taking
          \begin{equation*}
              \lambda(t) = -\kappa(t), \quad
              r(t) = \kappa(t)\theta(t),
          \end{equation*}
          and therefore
          \begin{equation*}
              b^0(t) = \kappa(t)(\theta(t)-c(t))-c'(t),
          \end{equation*}
          together with
          \begin{equation*}
              b^1(t,du) = \delta_t(du)-\Ind_{\{t>1\}}\delta_{t-1}(du) -\kappa(t)\bigl(\delta_t(du)-\Ind_{((t-1)^+,t]}(u)\,du\bigr),
          \end{equation*}
          also satisfies \eqref{pd-vol-condition} and yields the model
          \begin{align*}
              X(t) & = X(0) +  \int_0^t \Biggl( X(s)-\Ind_{\{s>1\}}X(s-1) + \kappa(s)\biggl(\theta(s)-c(s)-X(s)+\int_{(s-1)^+}^{s}X(u)\,du\biggr) - c'(s) \Biggr)ds \\
                   & \quad + \int_0^t \sqrt{\,c(s)+X(s)-\int_{(s-1)^+}^{s}X(u)\,du\,}\,dW(s).
          \end{align*}
          In this case, the variance process
          \begin{equation*}
              Y(t) := a(t,X_{[0,t]}) = c(t)+X(t)-\int_{(t-1)^+}^{t}X(u)\,du
          \end{equation*}
          satisfies the classical time-inhomogeneous CIR dynamics
          \begin{equation*}
              dY(t) = \kappa(t)\bigl(\theta(t)-Y(t)\bigr) \,dt + \sqrt{Y(t)} \,dW(t).
          \end{equation*}
\end{enumerate}

The following proposition shows that, for general choices of $c$, the
natural state constraint is not a fixed half-line constraint on $X$, but
rather a path-dependent constraint on the volatility functional.

\begin{proposition}\label{prop:whole-state-space}
    Let $T \in [1,\infty)$, $\kappa \in C([0,T], \R_{\geq 0})$, $c \in C^1([0,T], \R)$ and $X(0) \in \R$ with $c(0) + X(0) \geq 0$.
    Let $X$ solve \eqref{eq:pd-vol-model-general}.
    Then the process $X$ is in general not constrained to a state space $\R_{\geq 0}$ or $\R_{\leq 0}$ for any initial condition $X(0) \in \R$.
\end{proposition}
\begin{proof}
    Choose $\kappa \equiv 0$, and denote the variance term by
    $$Y(t) := a(t,X_{[0,t]}) = c(t) + X(t) - \int_{(t-1)^+}^{t} X(u) \, du.$$
    We first derive the dynamics of $Y$. Since
    \begin{equation*}
        dX(t) = \Bigl(-c'(t)+  X(t) - \Ind_{\{t>1\}}X(t-1)\Bigr) dt + \sqrt{Y(t)} \,dW(t),
    \end{equation*}
    and, by
    \begin{equation*}
        d\biggl(\int_{(t-1)^+}^{t}X(u)\,du\biggr) = \Bigl(X(t)-\Ind_{\{t>1\}}X(t-1)\Bigr)dt,
    \end{equation*}
    it follows that
    \begin{align*}
        dY(t) & = c'(t)\,dt + dX(t) - d\left(\int_{(t-1)^+}^{t}X(u)\,du\right)                  \\
              & = c'(t)\,dt +\Bigl(-c'(t)+X(t)-\Ind_{\{t>1\}}X(t-1)\Bigr)dt +\sqrt{Y(t)}\,dW(t) \\
              & \quad -\Bigl(X(t)-\Ind_{\{t>1\}}X(t-1)\Bigr)dt                                  \\
              & = \sqrt{Y(t)}\,dW(t).
    \end{align*}
    Therefore, $Y$ satisfies
    \begin{equation*}
        dY(t) = \sqrt{Y(t)} \, dW(t), \quad Y(0) = c(0) + X(0).
    \end{equation*}
    This is, up to scaling, a Bessel process of dimension $0$ started at $c(0) + X(0) \geq 0$. By pathwise uniqueness $\{0\}$ is an absorbing point \cite{revuz_Continuous_1999}, hence $Y(t) = 0$ for all $t \geq \tau=\inf\{s \geq 0: Y(s)=0\}$. In particular, the diffusion coefficient vanishes after time $\tau$, and the process $X$ evolves according to a deterministic delay equation
    \begin{equation*}
        X(t) = X(\tau) + \int_{\tau}^{t} \Bigl(-c'(s) + X(s) - \Ind_{\{s > 1\}}X(s-1)\Bigr) ds,
    \end{equation*}
    constrained by
    \begin{equation*}
        c(t) + X(t) - \int_{(t-1)^+}^{t} X(u) \, du = 0, \quad t \geq \tau.
    \end{equation*}
    Hence, if $c(0) + X(0) = 0$, $X$ is determined solely by
    \begin{equation*}
        c(t) + X(t) - \int_{(t-1)^+}^{t} X(u) \, du = 0, \quad t \geq \tau.
    \end{equation*}
    Now choose a sign-changing path $x(t) = 1-2t$ on $t \in [0,1]$ and set
    $c(t) := \int_{0}^{t} 1-2u \,du - \bigl(1-2t\bigr) = -1 +3t -t^2$ for $t \in [0,1]$ where we can extend $c$ arbitrarily to a $C^1$ function on the interval $[0,T]$. Then $c(0) =-1$ and choosing $X(0) = 1$ yields $c(0) + X(0) = 0$ with $X(1) = -1$. Thus, $X$ is not constrained to $\R_{\geq 0}$. Repeating the same construction for $x(t) = -1 + 2t$ yields $X(0) = -1$ and $X(1) = 1$, so $X$ is not constrained to $\R_{\leq 0}$.
    Consequently, the natural state constraint is not a fixed half-line for
    $X$. Instead, it is the path-dependent constraint
    \begin{equation*}
        c(t)+x(t)-\int_{(t-1)^+}^t x(u)\,du\ge0.
    \end{equation*}
\end{proof}

\newpage

\begin{proposition}\label{prop:pd-vol-nonnegative}
    Let $\kappa \in C([0,T], \R_{\geq 0})$, $c \equiv 0$ and $X(0) \in \R_{\geq 0}$.
    Let $X$ solve \eqref{eq:pd-vol-model-general}.
    Then $X(t) \geq 0$ for all $t \in [0,T]$.
    If in addition, $X(0) > 0$, then $X(t) > 0$ for all $t \in [0,T]$.
\end{proposition}
\begin{proof}
    Define
    \begin{equation*}
        Y(t) := a(t,X_{[0,t]}) = X(t) - \int_{(t-1)^+}^t X(s) \, ds, \quad t \in [0,T].
    \end{equation*}
    As in Proposition~\ref{prop:whole-state-space} we see that $Y$ satisfies for all $t \in [0,T]$,
    \begin{equation}\label{eq:Y-dynamics}
        dY(t) = \kappa(t)dt + \sqrt{Y(t)}\,dW(t), \quad Y(0) = X(0) \geq 0.
    \end{equation}

    Since $Y$ is an inhomogeneous Bessel process started from a non-negative initial value and $t \mapsto \kappa(t) \in \R_{\geq 0}$, it remains non-negative at all times.
    We now show that $X$ itself is non-negative. Let $X^-(t) = \max{-X(t),0}$ denote the negative part of $X$. Since $Y(t) \geq 0$ for all $t \in [0,T]$, we obtain by definition of $Y$
    \begin{equation*}
        X(t) \geq \int_{(t-1)^+}^{t} X(s) \,ds.
    \end{equation*}
    This implies in particular for $F(t) := \int_{0}^{t} X^-(s) \,ds$
    \begin{equation*}
        X^-(t) \leq \int_{(t-1)^+}^{t} X^-(s) \,ds \leq F(t).
    \end{equation*}
    Then $F$ is absolute continuous, $F(0) = 0$ and for almost every $t \in [0,T]$, $F'(t) = X^-(t) \leq F(t)$. Gronwall's inequality now implies that $F(t) = 0$ for all $t \in [0,T]$ and in particular that $X^-(t) = 0$ for all $t \in [0,T]$.

    Finally, assume that $X(0)>0$. Suppose, for a contradiction, that $X$ hits zero, and define
    \begin{equation*}
        \tau:=\inf\{t>0:X(t)=0\}.
    \end{equation*}
    By continuity and the already established non-negativity of $X$, we have
    $X(t)>0$ for all $t<\tau$. At time $\tau$,
    \begin{equation*}
        0=X(\tau) = Y(\tau)+\int_{(\tau-1)^+}^{\tau} X(s)\,ds.
    \end{equation*}
    Both terms on the right-hand side are non-negative. Hence,
    \begin{equation*}
        Y(\tau)=0, \quad \int_{(\tau-1)^+}^{\tau} X(s)\,ds=0.
    \end{equation*}
    Since $X$ is continuous and non-negative, this implies $X(s) =0$ for all $s\in[(\tau-1)^+,\tau]$.
    This contradicts the definition of $\tau$.
    Therefore, $X(t)> 0$ for all $t \in [0,T]$.
\end{proof}

\begin{remark}
    Let $\kappa : C([0,T], \R_{\leq 0})$ and $X(0) \leq 0$. Then Proposition~\ref{prop:pd-vol-nonnegative} also applies to the model
    \begin{equation*}
        X(t) = X(0) + \int_{0}^{t} \kappa(s) + X(s) - \Ind_{\{s > 1\}}X(s-1) \, ds + \int_{0}^{t} \sqrt{\int_{(s-1)^+}^{s} X(u) \, du - X(s)} \, dW(s), \quad t \in [0,T],
    \end{equation*}
    where $X$ now takes values in $\R_{\leq 0}$ and it never reaches zero whenever $\kappa : [0,T] \to \R_{< 0}$ and $X(0)< 0$.
\end{remark}

We now turn to the question of existence of solutions and to the calculation of the Fourier-Laplace transform of the model given by
\eqref{eq:pd-vol-model-general}.
Notice first that the coefficients $b$ and $\sigma = \sqrt{a}$ given by \eqref{parameters:pd-vol} satisfy the linear growth condition \ref{LG-Condition} and the existence of a weak solution is guaranteed by Theorem~\ref{thm:weak-existence}.

\begin{proposition}\label{prop:pd-vol-uniqueness-hidden-variance}
    Assume that $c\in C^1([0,T])$, $\kappa \in C([0,T], \R_{\geq 0})$, and that
    \begin{equation*}
        V(0) := c(0)+X(0) \geq 0.
    \end{equation*}
    Then the weak solution to \eqref{eq:pd-vol-model-general} is unique in law.
\end{proposition}

\begin{proof}
    Let $X$ be a weak solution of \eqref{eq:pd-vol-model-general} and define
    \begin{equation*}
        V(t) := c(t) + X(t) - \int_{(t-1)^+}^{t}X(u) \,du, \quad t \in [0,T].
    \end{equation*}
    By construction, $V(t)\geq0$, since $V(t)$ is the expression appearing
    inside the square root.
    As in Proposition~\ref{prop:whole-state-space} $V$ is a weak solution of
    \begin{equation*}
        dV(t) = \kappa(t)\,dt + \sqrt{V(t)} \,dW(t), \quad V(0) = c(0) + X(0).
    \end{equation*}
    This one-dimensional square-root SDE has uniqueness in law \cite{revuz_Continuous_1999}. Therefore, the
    law of $V$ on $C([0,T],\R_+)$ is uniquely determined.

    It remains to see that $X$ is uniquely determined by $V$. Set
    \begin{equation*}
        y(t) := V(t)-c(t).
    \end{equation*}
    Then $X$ satisfies the deterministic equation
    \begin{equation*}
        X(t) = y(t)+\int_{(t-1)^+}^{t}X(u) \,du
    \end{equation*}
    and every continuous $y$, the deterministic equation has a unique continuous solution. Indeed, setting
    \begin{equation*}
        F(t) := \int_0^t X(u)\,du,
    \end{equation*}
    the equation is equivalent to the delay differential equation
    \begin{equation*}
        F'(t) = y(t)+F(t)-F((t-1)^+), \quad F(0)=0.
    \end{equation*}
    This linear equation can be solved uniquely by the standard method of steps:
    on each interval $[n,n+1]$, the delayed term $F(t-1)$ is already known
    from the previous step, so one obtains a linear inhomogeneous ODE with given
    initial value $F(n)$. Hence $F$, and therefore
    \begin{equation*}
        X(t) = F'(t) = y(t)+F(t)-F((t-1)^+),
    \end{equation*}
    is uniquely determined by $y$.

    Therefore, there exists a continuous deterministic map $\cR$ such that
    \begin{equation*}
        X = \cR(V-c).
    \end{equation*}
    Consequently, if two weak solutions $X^1,X^2$ give rise to variance
    processes $V^1,V^2$ with the same law, then
    \begin{equation*}
        X^1 = \cR(V^1-c), \quad X^2 = \cR(V^2-c),
    \end{equation*}
    and therefore $X^1$ and $X^2$ have the same law.
    Hence, weak uniqueness holds for \eqref{eq:pd-vol-model-general}.
\end{proof}

With \eqref{parameters:pd-vol} we obtain the Riccati equation \eqref{riccati-ode} for any $\mu \in \ccM([0,T],\C)$ by
\begin{align}\label{eq:riccati-pd-vol}
    \begin{split}
        \psi(t) & = \mu([t,T]) + \int_{t}^{T} \psi(s) b^1(s,[t,s]) + \frac{1}{2} \psi(s)^2 a^1(s,[t,s]) \, ds                                             \\
                & = \mu([t,T]) + \int_{t}^{T} \psi(s) \Ind_{\{s < t+1\}}+ \frac{1}{2} \psi(s)^2 \bigl(1-s+\max((s-1)^+,t)\bigr) \, ds, \quad t \in [0,T],
    \end{split}
\end{align}
and
\begin{equation*}
    \phi(t) = \int_{t}^{T} \psi(s)\bigl(\kappa(s) - c'(s)\bigr) + \frac{1}{2}\psi(s)^2c(s)\,ds.
\end{equation*}
Then the $\C$-valued process $Y = (Y(t))_{t\in[0,T]}$ is given by
\begin{align}\label{eq:Y-pd-vol}
    \begin{split}
        Y(t) & = \phi(t) + \psi(0)X(0) +\int_0^t \psi(s) dX(s)                                                                                                                \\
             & \quad - \int_0^t \Biggl( \psi(s) \bigl(X(s)-\Ind_{\{s > 1\}}X(s-1)\bigr) + \frac{1}{2} \psi(s)^2 \biggl(X(s)-\int_{(s-1)^+}^{s} X(u) \, du \biggr) \Biggr) ds.
    \end{split}
\end{align}

\begin{proposition}\label{prop:global-sol-ricc-pd-vol}
    Let $\mu \in \ccM([0,T], \C)$ satisfy $\re(\mu([t,T])) \leq 0$ for all $t \in [0,T]$.
    Then the Riccati equation \eqref{eq:riccati-pd-vol} has a unique global solution $\psi \in L^2([0,T], \C)$, and this solution satisfies $\re(\psi(t))\le0$ for all $t\in[0,T]$ and in particular $\psi \in L^\infty([0,T],\C)$.
\end{proposition}

\begin{proof}
    We proceed as in the proof of \cite[Lemma 6.3]{abijaber_Affine_2019}.
    First we rewrite equation \eqref{eq:riccati-pd-vol} by
    \begin{equation*}
        \psi(t) = \mu([t,T]) + \int_{t}^{T} \psi(s) k(t,s) + \frac{1}{2} \psi(s)^2 q(t,s) \, ds, \quad t \in [0,T],
    \end{equation*}
    with $q(t,s) = 1-s+\max((s-1)^+,t) = (1+t-s)^+$ and $k(t,s) = \Ind_{\{s < t+1\}}$. Therefore, $0 \leq q(t,s) \leq 1$ for all $t,s \in [0,T]$ with $s \geq t$.
    Now by Theorem \ref{thm:existence-of-noncontinuable-riccati} there exists a unique noncontinuable solution $(\psi,T_{\text{min}})$ to equation
    \eqref{eq:riccati-pd-vol}. Let $\psi^r$ and $\psi^i$ denote the real and imaginary parts of $\psi$. On $t \in (T_{\text{min}},T]$ they satisfy the equations
    \begin{align*}
        \psi^r(t) & = \re \mu([t,T]) + \int_{t}^{T} \psi^r(s)k(t,s) + \frac{1}{2} q(t,s) \bigl(\psi^r(s)^2-\psi^i(s)^2\bigr) \,ds, \\
        \psi^i(t) & = \im \mu([t,T]) + \int_{t}^{T} \psi^i(s)\bigl(k(t,s) + q(t,s)\psi^r(s)\bigr) \, ds.
    \end{align*}
    Moreover, on the interval $(T_{\text{min}},T]$, $-\psi^r$ satisfies the linear equation
    \begin{equation*}
        \chi(t) = - \re \mu([t,T]) + \int_{t}^{T} \chi(s) k(t,s) + \frac{1}{2}q(t,s) \bigl(\psi^i(s)^2 + \chi(s)\psi^r(s)\bigr) \, ds.
    \end{equation*}
    Since $q$ and $k$ are non-negative and also by assumption $\re \mu([t,T])$ is non-positive for all $t \in [0,T]$, \cite[Theorem C.2]{abijaber_Affine_2019} with $K \equiv \Id$ yields $\psi^r \leq 0$. Now we conclude in the same manner as \cite[Lemma 6.3]{abijaber_Affine_2019}, where the equations for $h$ and $\ell$ become
    \begin{align*}
        h(t)    & = |\im \mu([t,T])| + \int_{t}^{T} k(t,s)h(s) \,ds,                             \\
        \ell(t) & = -\re \mu([t,T]) + \int_{t}^{T} k(t,s)\ell(s) + \frac{1}{2}q(t,s)h(s)^2 \,ds.
    \end{align*}
    By \citet[Corollary B.3]{abijaber_Affine_2019} with $K \equiv \Id$ there exists unique global solutions $h,l \in L^2([0,T], \R)$. Moreover, these solutions are in fact bounded and thus $h, \ell \in L^\infty([0,T])$. Now as in the proof of \cite[Lemma 6.3]{abijaber_Affine_2019} one shows that $-l \leq \psi^r \leq 0$ and $|\psi^i| \leq h$. This implies that there exists a unique global solution $\psi \in L^2([0,T],\C)$ of \eqref{eq:riccati-pd-vol} which satisfies $\re \psi \leq 0$. In particular, $\psi \in L^\infty([0,T],\C)$.
\end{proof}

\begin{proposition}\label{prop:pd-vol-exp-affine-formula}
    Let $\mu \in \ccM([0,T], \R)$ satisfy $\mu([t,T]) \leq 0$ for all $t \in [0,T]$ and let $\psi \in L^2([0,T], \R)$ denote the unique global solution of \eqref{eq:riccati-pd-vol}. Then $\exp(Y)$ with $Y$ defined by \eqref{eq:Y-pd-vol} is a true martingale. In particular, the exponential-affine transform formula \eqref{exp-affine-formula} holds.
\end{proposition}

\begin{proof}
    Note first that by Theorem~\ref{thm:existence-of-noncontinuable-riccati} that the unique global solution to $\psi$ is real-valued since $\mu$ is a real-valued measure.
    Note further that the process $Y$ has the representation
    \begin{align}
        \begin{split}
            Y(t) & = Y(0) + \int_0^t \psi(s) \sigma(s,X_{[0,s]}) \, dW(s) - \frac{1}{2} \int_0^t  \psi(s)^2 a(s,X_{[0,s]}) \, ds, \quad t \in [0,T], \\
            Y(0) & = \phi(0) + \psi(0)X(0),
        \end{split}
    \end{align}
    and by Theorem~\ref{main_thm} $\exp(Y)$ is a local martingale.
    Further, $\exp(Y(t)) = \exp(Y(0))\exp(U(t) - \frac{1}{2}\langle U \rangle(t))$ for $U(t) := \int_0^t \psi(s) \sigma(s,X_{[0,s]}) \, dW(s)$.
    Thus, it remains to show that the process $M(t) := \exp(U(t) - \frac{1}{2} \langle U \rangle(t))$ is a martingale since $Y(0)$ is deterministic. As $M$ is a non-negative local martingale, it is a supermartingale by Fatou's lemma, and it suffices to show that $E[M_T] \geq 1$ for any $T \in \R_{\geq 0}$. Define the stopping times $\tau_n = \inf\{t \geq 0 : |X(t)| > n\} \wedge T$. Then the stopped process $M^{\tau_n}$ is a uniformly integrable martingale for each $n$ by Novikov's condition, and we may define probability measures $Q^n$ by
    \begin{equation*}
        \frac{dQ^n}{dP} = M_{\tau_n}.
    \end{equation*}
    By Girsanov's theorem, the process $dW^n(t) = dW(t) - \Ind_{\{t \leq \tau_n\}} \psi(t)\sigma(t,X_{[0,t]})dt$ is a Brownian motion under $Q^n$, and we have for $t \leq \tau_n$
    \begin{equation*}
        X(t) = X(0) + \int_{0}^{t} \biggl( b(s,X_{[0,s]}) + \psi(s)a(s,X_{[0,s]}) \biggr) ds + \int_{0}^{t} \sigma(s,X_{[0,s]}) dW^n(s).
    \end{equation*}
    Note that under $Q^n$, $X$ satisfies a path-dependent SDE with Brownian motion $W^n$ and drift $b^n(t,x) := b(t,x) + \psi(t)a(t,x)\Ind_{\{\sup_{0 \leq u \leq t} |x(u)| < n\}}$. Observe now that, since $\psi$ is bounded that $b^n(t,x)$ satisfies a linear growth condition \ref{LG-Condition} in $x$, uniformly in $t$.
    Let $p \geq  2$. By Lemma~\ref{lm:estimate-existence}, we have the moment bound
    \begin{equation*}
        \sup_{t \leq T} E_{Q^n}[\|X(t)\|^p] \leq C
    \end{equation*}
    for some constant $C$ that does not depend on $n$. Let $|\cdot|_{C^\alpha}$ denote the $\alpha$-Hölder seminorm, we then get
    \begin{align*}
        Q^n(\tau_n < T) & \leq Q^n \Bigl(\sup_{t \leq T} |X(t)| > n\Bigr)                         \\
                        & \leq Q^n(|X(0)| + |X|_{C^0} > n)                                        \\
                        & \leq \left(\frac{1}{n-|X(0)|}\right)^p E_{Q^n}\left[|X|_{C^0}^p \right] \\
                        & \leq \left(\frac{1}{n-|X(0)|}\right)^p C'
    \end{align*}
    for a constant $C'$ that does not depend on $n$, using Lemma \ref{lm:hoelder-continuity} with $\alpha = 0$ for the last inequality. We deduce that
    \begin{equation*}
        E_{P}[M_T] \geq E_{P}[M_T \Ind_{\{\tau_n = T\}}] = Q^n(\tau_n = T) \geq 1 - \left(\frac{1}{n-|X(0)|}\right)^p C',
    \end{equation*}
    and sending $n$ to infinity yields $E_{P}[M_T] \geq 1$, which shows that the process $M$ is a true martingale. Consequently, $\exp(Y)$ is also a true martingale and Theorem $\ref{main_thm}$ applies and the exponential-affine transform formula holds.
\end{proof}

\subsection{Path-dependent inhomogeneous Heston model}\label{subsec:path-dep-heston-model}
To introduce the path-dependent Heston model, let $d = 2$, $\rho \in [-1,1]$ and let $\kappa, \theta, \overline{\sigma}, \eta : [0,T] \to [0,\infty)$ be continuous such that $\overline{\sigma}$ is strictly positive. Consider for initial values $S(0) \in (0,\infty), V(0) \in [0,\infty)$ the bivariate process
\begin{align}\label{eq:path-dep-inhom-heston-spot-price}
    S(t) & = S(0) + \int_{0}^{t} S(r) \eta(r) \sqrt{V(r)}\bigl(\sqrt{1-\rho^2}dW_1(r) + \rho dW_2(r)\bigr),                                                                                       \\
    \label{eq:path-dep-inhom-heston-vol}
    V(t) & = V(0) + \int_{0}^{t} \Bigl(\kappa(s)(\theta(s) - V(s)) + \int_{[0,s]} V(u) \, \gamma(s,du)\Bigr) \, ds + \int_{0}^{t} \overline{\sigma}(s) \sqrt{V(s)} \, dW_2(s), \quad t \in [0,T],
\end{align}
where $W = (W_1, W_2)^\top$ is a two-dimensional Brownian motion and a non-negative measure $\gamma : [0,T] \to \ccM([0,T], \R_{\geq 0})$ such that
\begin{equation*}
    \int_{0}^{T} \|\gamma(s)\|_{TV} \, ds < \infty.
\end{equation*}
Structurally, the model remains within the class of path-dependent affine processes introduced in Section~\ref{sec:reverse-direction}, since the drift term is an affine functional of the form
\begin{equation*}
    b(s,V_{[0,s]}) = b^0(s) + \int_{[0,s]} V(u) \, b^1(s,du), \quad s \in [0,T],
\end{equation*}
with $b^0(s) = \kappa(s)\theta(s)$ and $b^1(s,du) = -\kappa(s)\delta_s(du) + \gamma(s,du)$.

The path-dependent affine framework allows even signed measures for $\gamma$.
However, since we have to guarantee that the process $V$ stays non-negative at all times we must ensure that the drift remains non-negative at the boundary points of the process. Similar for the variance term of $V$ which is not path-dependent and only depends on the current state of the process.
The additional term
\begin{equation*}
    \int_{[0,s]} V(u) \, \gamma(s,du)
\end{equation*}
introduces a path-dependent memory component into the mean-reversion dynamics.

Choosing $\gamma(s,du) = c\,\delta_{(s-\tau)^+}(du)$ for some $c \in \R_{\geq 0}$ and fixed delay $\tau > 0$ leads to a stochastic delay differential equation for the variance process, as studied in \cite{flore_FeynmanKac_2019} and choosing $\gamma(s,du) \equiv 0$ leads to the classical inhomogeneous Heston model as studied in \cite{ackermann_Inhomogeneous_2022} for the case where the Volterra kernels correspond to the identity matrix.
Intuitively, $V$ still mean-reverts toward its level $\theta(t)$ with rate $\kappa(t)$, but the strength of this mean reversion is influenced by the historical path of volatility.
The measure $\gamma$ can reinforce the effect of past volatility on the current drift,
leading to more persistent volatility dynamics and stronger autocorrelation of $V$.
Other types of delayed Heston-models have also been proposed as in \cite{swishchuk_Smiling_2014}.

\begin{proposition}
    Equation \eqref{eq:path-dep-inhom-heston-vol} possesses a $[0,\infty)$-valued continuous weak solution $V$.
    Furthermore,
    \begin{equation}\label{eq:path-dep-inhom-heston-sol-of-price}
        S(t) = S(0)e^{\int_{0}^{t} \eta(s)\sqrt{V(s)}\bigl(\sqrt{1-\rho^2}dW_1(s) + \rho dW_2(s)\bigr) - \int_{0}^{t} \tfrac{\eta(s)^2V(s)}{2} ds}, \quad t \in [0,T],
    \end{equation}
    is a solution of \eqref{eq:path-dep-inhom-heston-spot-price}, and $S$ is a martingale.
\end{proposition}
\begin{proof}
    The second part of the proposition has been shown in \cite[Proposition 4.1]{ackermann_Inhomogeneous_2022}.
    For $(t,x) \in \ccC_T^1$ let $\hat{b}(t,x) = \kappa(t)\theta(t) - \kappa(t)x(t) + \int_{[0,t]} x(s) \, \gamma(t,ds)$ and $\hat{\sigma}(t,x) = \overline{\sigma}(t)\sqrt{x(t)^+}$, and consider the path-dependent equation
    \begin{equation}\label{eq:sol-of-inhom-heston}
        V(t) = V(0) + \int_{0}^{t} \hat{b}(s,V_{[0,s]})  \, ds + \int_{0}^{t} \hat{\sigma}(s,V_{[0,s]}) \, dW_2(s), \quad t \in [0,T].
    \end{equation}
    Due to the continuity of $\kappa,\theta$ and $\overline{\sigma}$, $\hat{b}$ and $\hat{\sigma}$ satisfy the linear growth condition \ref{LG-Condition}. It now follows from Theorem~\ref{thm:weak-existence} that Equation \eqref{eq:sol-of-inhom-heston} admits a weak solution $V$.
    Furthermore, for a path $x \in C([0,t],\R^d)$ that hits the boundary $0$ for the first time at that time $t \in [0,T]$, we have $\hat{\sigma}(t,x) = 0$ and $\hat{b}(t,x) = \kappa(t)\theta(t) + \int_{[0,t]} x(s) \, \gamma(t,ds) \geq 0$, since the functions $\kappa$ and $\theta$ are non-negative and the measure is also non-negative. Hence, the drift points inward at the boundary and the diffusion is tangent
    to the boundary. By a one-dimensional stochastic invariance theorem for
    closed convex sets (e.g. \cite[Theorem 2.3]{jaber_Stochastic_2019}), the weak solution with $V(0) \in [0,\infty)$ satisfies $V(t) \in [0,\infty)$ for all $0\leq t\leq T$ almost surely.
\end{proof}
For the remainder of this section let $V$ be a $[0,\infty)$-valued continuous weak solution of \eqref{eq:path-dep-inhom-heston-vol} and let $S$ be the solution of \eqref{eq:path-dep-inhom-heston-spot-price} given by \eqref{eq:path-dep-inhom-heston-sol-of-price}.
Note that It\^o's formula shows that for all $t \in [0,T]$ the log-price is given by
\begin{equation}\label{eq:log-price-dynamics}
    \log(S(t)) = \log(S(0)) + \int_{0}^{t} \eta(r) \sqrt{V(r)} \bigl(\sqrt{1-\rho^2}dW_1(r) + \rho dW_2(r)\bigr) - \int_{0}^{t} \tfrac{\eta^2(r)V(r)}{2} dr.
\end{equation}
Hence, the pair $X = (\log(S), V)^\top$ satisfies

\begin{align*}
    \begin{split}
        \begin{pmatrix}
            X^1(t) \\ X^2(t)
        \end{pmatrix}
         & =
        \begin{pmatrix}
            X^1(0) \\ X^2(0)
        \end{pmatrix}
        + \int_{0}^{t} \biggl[
            \begin{pmatrix}
                0 \\ \kappa(r)\theta(r)
            \end{pmatrix}
            +
            \begin{pmatrix}
                0 \\ 0
            \end{pmatrix} X^1(r)
            +
            \int_{[0,r]} X^2(u)
            \begin{pmatrix}
                - \tfrac{\eta^2(r)}{2}\delta_r(du) \\
                -\kappa(r)\delta_r(du) + \gamma(r,du)
            \end{pmatrix}
        \biggr]dr               \\
         & \quad + \int_{0}^{t}
        \int_{[0,r]} \sqrt{X^2(u)}
        \begin{pmatrix}
            \eta(r) \sqrt{1-\rho^2} \delta_r(du) & \eta(r)\rho \delta_r(du)          \\
            0                                    & \overline{\sigma}(r) \delta_r(du)
        \end{pmatrix}
        dW(r), \quad t \in [0,T].
    \end{split}
\end{align*}
For $(t,x) \in \ccC_T^2$ let

\begin{align}\label{eq:path-dep-inhom-heston-parameters}
    \begin{split}
        \widetilde{b}^0(t)      & =
        \begin{pmatrix}
            0 \\ \kappa(t)\theta(t)
        \end{pmatrix}, \quad
        \widetilde{b}^1(t,du) = 0, \quad
        \widetilde{b}^2(t,du) =
        \begin{pmatrix}
            - \tfrac{\eta^2(t)}{2}\delta_t(du) \\ -\kappa(t)\delta_t(du) + \gamma(t,du)
        \end{pmatrix},                                    \\
        \widetilde{b}(t,x)      & = \widetilde{b}^0(t) + \widetilde{b}^1(t) \cdot x^1 + \widetilde{b}^2(t) \cdot x^2, \\
        \widetilde{\sigma}(t,x) & = \int_{[0,t]} \sqrt{x^2(u)}
        \begin{pmatrix}
            \eta(t) \sqrt{1-\rho^2} \delta_t(du) & \eta(t)\rho \delta_t(du)          \\
            0                                    & \overline{\sigma}(t) \delta_t(du)
        \end{pmatrix}, \quad a^0(t) = 0, \quad a^1(t,du) \equiv 0,                                      \\
        a^2(t,du)               & =
        \begin{pmatrix}
            \eta^2(t) \delta_t(du)                       & \rho \eta(t)\overline{\sigma}(t) \delta_t(du) \\
            \rho \eta(t)\overline{\sigma}(t)\delta_t(du) & \overline{\sigma}^2(t) \delta_t(du)
        \end{pmatrix}, \quad a(t,x) = a^0(t) + a^1(t) \cdot x^1 + a^2(t) \cdot x^2,
    \end{split}
\end{align}

and observe that Assumption \ref{ass:affine-coeff} with state space $E = \R \times [0, \infty)$ and Assumption \ref{LG-Condition} are satisfied. Then
\begin{equation}
    X(t) = X(0) + \int_{0}^{t} \widetilde{b}(u,X_{[0,u]}) du + \int_{0}^{t} \widetilde{\sigma}(u, X_{[0,u]}) \, dW(u), \quad t \in [0,T].
\end{equation}
With \eqref{eq:path-dep-inhom-heston-parameters} we thus obtain that for the inhomogeneous path-dependent Heston model the Riccati equation \eqref{riccati-ode} for any $\mu \in \ccM([0,T], \C^2)$ reads
\begin{align}\label{eq:app-temp-1}
    \psi_1(t) & = \mu_1([t,T]) \\
    \begin{split}\label{eq:app-temp-2}
        \psi_2(t) & = \mu_2([t,T]) + \int_{t}^{T} \frac{\eta^2(s)}{2}\bigl(\psi_1^2(s)-\psi_1(s)\bigr) + \psi_2(s)\bigl(\gamma(s,[t,s]) - \kappa(s)\bigr) + \frac{\overline{\sigma}^2(s)}{2}\psi_2^2(s) \\
                  & \quad + \rho \overline{\sigma}(s)\eta(s)\psi_1(s)\psi_2(s) \, ds, \quad t \in [0,T].
    \end{split}
\end{align}
In the non-delay case, i.e. $\gamma(s, dr) = 0$ and $\mu(dr) = u\delta_T(dr)$ for $u \in \C^2$, we recover the Riccati equations of \cite{ackermann_Inhomogeneous_2022} with $K \equiv \Id$. The $\C$-valued process $Y = (Y(t))_{t\in[0,T]}$ is given by
\begin{align}\label{eq:Y-pd-heston}
    \begin{split}
        Y(t) & = \phi(t) + \psi(0)X(0) +\int_0^t \psi(s) dX(s)                                                                                                                        \\
             & \quad - \sum_{i = 1}^{2} \int_0^t \psi(s) \bigl(\widetilde{b}^i(s) \cdot X_{[0,s]}\bigr) + \frac{1}{2} \psi(s) \biggl( a^i(s) \cdot X_{[0,s]} \biggr) \psi(s)^\top ds.
    \end{split}
\end{align}

\begin{proposition}
    Suppose that $\mu \in \ccM([0,T], \C^2)$ such that $\re(\mu_1([t,T])) \in [0,1]$ and $\re(\mu_2([t,T])) \leq 0$ for all $t \in [0,T]$.
    Then the Riccati equation \eqref{eq:app-temp-2} has a unique solution $\psi_2 \in L^2([0,T], \C)$. Moreover, it satisfies $\re(\psi_2) \leq 0$.
\end{proposition}

\begin{proof}
    The proof follows closely to the one of \cite[Lemma 7.4]{abijaber_Affine_2019}.
    By Theorem \ref{thm:existence-of-noncontinuable-riccati}, there exists a unique noncontinuable solution $(\psi_2,T_{\min})$ of \eqref{eq:app-temp-2}. Let $\psi_k^r$ and $\psi_k^i$ denote the real and imaginary parts of $\psi_k$, $k = 1,2$. On the interval $(T_{\text{min}},T]$ they satisfy the equations
    \begin{align*}
        \psi_2^r(t) & = \re(\mu_2([t,T])) + \int_{t}^{T} \Biggl[ \frac{\eta(s)^2}{2} \Bigl(\psi_1^r(s)^2 - \psi_1^r(s) - \psi_1^i(s)^2\Bigr) - \rho\overline{\sigma}(s)\eta(s)\psi_1^i(s)\psi_2^i(s)                                      \\
                    & \quad - \frac{\overline{\sigma}(s)^2}{2}\psi_2^i(s)^2 + \Bigl(\rho\overline{\sigma}(s)\eta(s)\psi_1^r(s)+ \gamma(s,[t,s]) - \kappa(s)\Bigr)\psi_2^r(s) + \frac{\overline{\sigma}(s)^2}{2} \psi_2^r(s)^2 \Biggr] ds, \\
        \psi_2^i(t) & = \im(\mu_2([t,T])) + \int_{t}^{T} \Biggl[ \frac{\eta(s)^2}{2} \Bigl(2\psi_1^r(s) \psi_1^i(s) - \psi_1^i(s)\Bigr)
        +\rho\overline{\sigma}(s)\eta(s)\psi_1^i(s)\psi_2^r(s)                                                                                                                                                                            \\
                    & \quad + \Bigl( \rho\overline{\sigma}(s)\eta(s)\psi_1^r(s) + \gamma(s,[t,s]) - \kappa(s) + \overline{\sigma}(s)^2 \psi_2^r(s)\Bigr)\psi_2^i(s) \Biggr] ds.
    \end{align*}
    After some rewriting, on this interval, $-\psi^r$ satisfies the linear equation
    \begin{align*}
        \chi(t) & = - \re(\mu_2([t,T])) + \int_{t}^{T} \Biggl[ \frac{\eta(s)^2}{2} \Bigl(\psi_1^r(s) - \psi_1^r(s)^2 + (1-\rho^2) \psi_1^i(s)^2\Bigr) \\
                & \quad + \frac{\bigl(\overline{\sigma}(s)\psi_2^i(s)+\rho\eta(s)\psi_1^i(s)\bigr)^2}{2}
            +\Bigl( \rho\overline\sigma(s)\eta(s)\psi_1^r(s) +\gamma(s,[t,s])-\kappa(s) +\frac{\overline\sigma(s)^2}{2}\psi_2^r(s) \Bigr)\chi(s)\Biggr] ds.
    \end{align*}
    Since $\eta \geq 0$ and $\overline{\sigma} >0$, as well as $\psi_1^r,|\rho| \in [0,1]$ and $\re \mu_2([t,T])$ nonpositive for every $t \in [0,T]$, \cite[Theorem C.2]{abijaber_Affine_2019} with $K \equiv \Id$ yields $\psi_2^r \leq 0$ on $(T_{\text{min}},T]$. The equations $h$ and $\ell$ in the proof of \cite[Lemma 7.4]{abijaber_Affine_2019} become
    \begin{align*}
        h(t)    & = |\im \mu_2([t,T])| + |\rho \overline{\sigma}(t)\eta(t)\psi_1^i(t)| + \int_{t}^{T} \Biggl[ \Bigl|\frac{\eta(s)^2}{2}\bigl(2\psi_1^r(s)\psi_1^i(s) - \psi_1^i(s)\bigr)                                                    \\
                & \quad - \bigl(\rho \overline{\sigma}(s)^{-1}\eta(s)\psi_1^r(s) + \gamma(s,[t,s])- \kappa(s)\bigr)\rho \eta(s)\overline{\sigma}(s)^{-1} \psi_1^i(s)\Bigr|                                                                  \\
                & \quad + \Bigl(\rho \overline{\sigma}(s)^{-1}\eta(s)\psi_1^r(s) + \gamma(s,[t,s])- \kappa(s)\Bigr)h(s) \Biggr] ds,                                                                                                         \\
        \ell(t) & = \re \mu_2([t,T]) + \int_{t}^{T} \Biggl[ \frac{\eta(s)^2}{2} \Bigl(\psi_1^r(s)^2 - \psi_1^r(s) - \psi_1^i(s)^2\Bigr)                                                                                                     \\
                & \quad - |\rho \overline{\sigma}(s)\eta(s)\psi_1^i(s)|\bigl(h(s) + |\rho \overline{\sigma}(s)^{-1}\eta(s)\psi_1^i(s)|\bigr)                                                                                                \\
                & \quad - \frac{\overline{\sigma}(s)^2}{2}\bigl(h(s)+|\rho \overline{\sigma}(s)^{-1}\eta(s)\psi_1^i(s)|\bigr)^2 + \bigl(\rho \overline{\sigma}(s)\eta(s) \psi_1^r(s) + \gamma(s,[t,s]) - \kappa(s)\bigr)\ell(s) \Biggr] ds.
    \end{align*}
    By \cite[Corollary B.3]{abijaber_Affine_2019} with $K \equiv \Id$ there exists unique global solutions $h,l \in L^2([0,T], \R)$. As in the proof of \cite[Lemma 7.4]{abijaber_Affine_2019} one shows that $\ell \leq \psi_2^r \leq 0$ and $|\psi_2^i| \leq h + |\rho \overline{\sigma}(t)^{-1}\eta(t)\psi_1^i|$. From this it follows by the same arguments as in the proof of \cite[Lemma 7.4]{abijaber_Affine_2019} that there exists a unique global solution $\psi_2 \in L^2([0,T],\C)$ of \eqref{eq:app-temp-2} which satisfies $\re \psi_2 \leq 0$.
\end{proof}

\begin{proposition}
    Suppose that $\mu \in \ccM([0,T], \C^2)$ such that $\re(\mu_1([t,T])) \in [0,1]$ for all $t \in [0,T]$ and $\re(\mu_2)$ is a non-positive measure.
    Let $\psi_2 \in L^2([0,T], \C)$ denote the unique solution of \eqref{eq:app-temp-2}. Then $\exp(Y)$ with $Y$ defined by \eqref{eq:Y-pd-heston} is a true martingale. In particular, the exponential-affine transform formula \eqref{exp-affine-formula} holds and the weak solution $(\log S, V)$ to \eqref{eq:path-dep-inhom-heston-spot-price}-\eqref{eq:path-dep-inhom-heston-vol} is unique.
\end{proposition}

\begin{proof}
    The proof is similar to the proof of Proposition~\ref{prop:pd-vol-exp-affine-formula}, where we show that $\re(Y)$ is bounded from above on $[0,T]$. We use the representation \eqref{Y-tilde} of $Y$ which in this case corresponds to
    \begin{align*}
        Y(t) & = \phi(t) + \psi(t)X(t) + \mu \cdot X_{[0,t)}                                                                                                                       \\
             & \quad + \int_t^T \sum_{i=1}^{2} \Bigl( \psi(s) ( \widetilde{b}^i(s) \cdot X^i_{[0,t)}) + \frac{1}{2}\psi(s) (a^i(s) \cdot X^i_{[0,t)}) \psi(s)^\top \Bigr) ds,      \\
             & = \phi(t) + \psi_1(t)\log S(t) + \psi_2(t)V(t) + \mu_1 \cdot (\log S)_{[0,t)} + \mu_2 \cdot V_{[0,t)} + \int_t^T \psi_2(s) \int_{[0,t)} V(u) \, \gamma(s,du) \, ds,
    \end{align*}
    with
    \begin{equation*}
        \phi(t) = \int_{t}^{T} \kappa(s) \theta(s)\psi_2(s) \, ds.
    \end{equation*}
    Taking real parts and inspecting the signs $V \geq 0$, $\kappa\theta \geq 0$, $\re \mu_2 \leq 0$, $\re \psi_2 \leq 0$, and $\gamma(s, \cdot) \geq 0$, all terms involving $V$ have non-positive real part. Hence,
    \begin{equation*}
        \re Y(t) \leq \re \psi_1 \log S(t) + \re\bigl(\mu_1 \cdot (\log S)_{[0,t)}\bigr)
    \end{equation*}
    and since $\psi_1(t) = \mu_1([t,T])$ integration by parts yields
    \begin{equation*}
        \re \psi_1 \log S(t) + \re\bigl(\mu_1 \cdot (\log S)_{[0,t)}\bigr) = \re \psi_1(0)\log S(0) + \int_{[0,t]} \re \psi_1(u) d\log S(u).
    \end{equation*}
    By Equation~\eqref{eq:log-price-dynamics} the log-price dynamics are
    \begin{equation*}
        d\log S(t) = -\frac{1}{2}\eta(t)^2V(t)\,dt + \eta(t)\sqrt{V(t)}
        \bigl(\sqrt{1-\rho^2}\,dW_1(t) +\rho\,dW_2(t)\bigr).
    \end{equation*}
    Hence,
    \begin{equation*}
        \re Y(t) \leq \re\psi_1(0)\log S(0) + U(t) - \frac{1}{2}\langle U \rangle(t),
    \end{equation*}
    with
    \begin{equation*}
        U(t) = \int_{0}^{t} \re\psi_1(s) \eta(s)\sqrt{V(s)} \bigl(\sqrt{1-\rho^2}\,dW_1(s) +\rho\,dW_2(s)\bigr),
    \end{equation*}
    since $\re \psi_1 \in [0,1]$, which implies $\re \psi_1 \geq (\re \psi_1)^2$. It now follows that
    \begin{equation*}
        |\exp(Y(t))| = \exp(\re Y(t)) \leq S(0)^{\re \psi_1(0)} \exp \Bigl(U(t)-\frac{1}{2}\langle U \rangle (t)\Bigr), \quad t \in [0,T].
    \end{equation*}
    In the same manner as in Proposition~\ref{prop:pd-vol-exp-affine-formula}, using the fact that $\eta,\overline{\sigma}, \re \psi_1$ and $\re \psi_2$ are bounded, we can show that $\exp \bigl(U(t)-\frac{1}{2}\langle U \rangle (t)\bigr)$ is a true martingale and hence $\exp(Y)$ is a true martingale. It remains to prove uniqueness in law for $X$.
    The law of $X$ is determined by the characteristic function
    \begin{equation*}
        \E\Big[\exp\Big(i\sum_{j=1}^n \lambda_j^\top X(t_j)\Big)\Big], \quad n \in \N,\ 0 \leq t_1,\dots,t_n \leq T,\ \lambda_j \in \R^2.
    \end{equation*}
    Now observe that for every such choice of $n$, $(t_j)_{j=1}^n$ and
    $(\lambda_j)_{j=1}^n$, the measure $\mu := i\sum_{j=1}^n \lambda_j \delta_{t_j}$
    belongs to $\ccM([0,T],\C^2)$ and satisfies $\re(\mu_2) = 0$ and $\re(\mu_1([t,T])) = 0$ for all $t \in [0,T]$.
    Hence, the corresponding characteristic function is of the form
    \begin{equation*}
        \E\Big[\exp\Big(\int_0^T X(u)\,\mu(du)\Big)\Big],
    \end{equation*}
    which is uniquely determined by the exponential-affine transform formula.
\end{proof}

\appendix

\section{Auxiliary Results}\label{app:Auxiliary}

In this appendix we collect existence and uniqueness results for path-dependent SDEs under Lipschitz and linear growth conditions, followed by existence theory for generalized Riccati equations.

\begin{lemma}\label{lem:nonanticipative-lipschitz-approximation}
    Let $q\in\N$, and let $F:\ccC_T^d\to \R^q$
    be a continuous non-anticipative functional. Assume that there exists a constant
    $C>0$ such that
    \begin{equation*}
        \|F(t,x)\|
        \leq C \biggl(1+\sup_{s\in[0,t]}\|x(s)\|\biggr), \quad (t,x)\in\ccC_T^d .
    \end{equation*}
    Then there exist continuous non-anticipative functionals
    \begin{equation*}
        F^n:\ccC_T^d\to\R^q,\quad n\in\N,
    \end{equation*}
    such that the following properties hold.

    \begin{enumerate}
        \item $F^n\to F$ locally uniformly on $\ccC_T^d$.

        \item For each $n\in\N$, $F^n$ is Lipschitz continuous with respect
              to the metric $d_{\ccC_T}$.

        \item There exists a constant $C'>0$, independent of $n$, such that
              \begin{equation*}
                  \|F^n(t,x)\|
                  \leq C'\biggl(1+\sup_{s\in[0,t]}\|x(s)\|\biggr),
                  \quad (t,x)\in\ccC_T^d,\ n\in\N .
              \end{equation*}
    \end{enumerate}

    In particular, for each $n\in\N$ and each fixed $t\in[0,T]$, there
    exists a constant $L_n>0$ such that
    \begin{equation*}
        \|F^n(t,x)-F^n(t,y)\|
        \leq L_n \sup_{s\in[0,t]}\|x(s)-y(s)\|,
        \quad x,y\in C([0,t],\R^d).
    \end{equation*}
\end{lemma}

\begin{proof}
    Let $(E,\rho)$ be an arbitrary metric space and let $f:E \to \R$ be continuous.
    It is now well-known that there exists Lipschitz functions $h_n:E\to\R$ such that $h_n\to f$ locally uniformly on $E$. Nevertheless, we prove it here, since with the construction of those functions $h_n$ we will be able to show the properties (1)-(3).
    For $n \in \N$, define the truncation
    \begin{equation*}
        f^{(n)} := (-n) \vee (f\wedge n),
    \end{equation*}
    and set, with $\lambda_n$ such that $\frac{\lambda_n}{n} \to \infty$ as $n \to \infty$,
    \begin{equation*}
        h_n(e) := \inf_{y\in E}\left\{f^{(n)}(y) + \lambda_n\rho(e,y)\right\},
        \quad e\in E.
    \end{equation*}
    Since $f^{(n)}$ is bounded, $h_n$ is finite. Moreover, for
    $e,e'\in E$, the triangle inequality gives
    \begin{equation*}
        \begin{aligned}
            h_n(e) & \leq f^{(n)}(y)+\lambda_n\rho(e,y)                      \\
             & \leq f^{(n)}(y)+\lambda_n\rho(e,e')+\lambda_n\rho(e',y)
        \end{aligned}
    \end{equation*}
    for every $y\in E$. Taking the infimum over $y$ yields
    \begin{equation*}
        h_n(e) \leq h_n(e')+\lambda_n \rho(e,e')
    \end{equation*}
    and in particular
    \begin{equation*}
        |h_n(e)-h_n(e')| \leq \lambda_n \rho(e,e'),
    \end{equation*}
    so $h_n$ is Lipschitz continuous for each $n \in \N$.

    We now show that $h_n\to f$ locally uniformly. Let $K\subset E$ be compact and
    let $\varepsilon>0$. Since $f$ is continuous, it is bounded on $K$. Moreover,
    by compactness of $K$, there exists $\delta>0$ such that
    \begin{equation*}
        e\in K,\quad \rho(e,y)<\delta \quad \Rightarrow \quad |f(y)-f(e)|<\varepsilon .
    \end{equation*}
    Enlarging $n$ if necessary, we may assume that $f^{(n)}=f$ on the
    $\delta$-neighbourhood of $K$. For $e\in K$, we always have
    \begin{equation*}
        h_n(e)\leq f^{(n)}(e)=f(e).
    \end{equation*}
    On the other hand, if $\rho(e,y)<\delta$, then
    \begin{equation*}
        f^{(n)}(y)+\lambda_n\rho(e,y) = f(y)+\lambda_n\rho(e,y) \geq f(e)-\varepsilon.
    \end{equation*}
    If $\rho(e,y) \geq \delta$, then
    \begin{equation*}
        f^{(n)}(y)+\lambda_n\rho(e,y) \geq -n+\lambda_n\delta.
    \end{equation*}
    Since $ \lim_{n \to \infty} \lambda_n/n = \infty$, the latter quantity tends to $+\infty$. Hence, for all sufficiently large $n$,
    \begin{equation*}
        f^{(n)}(y)+\lambda_n\rho(e,y)\geq f(e)-\varepsilon
    \end{equation*}
    uniformly in $e\in K$ and $y\in E$. Taking the infimum over $y$ yields
    \begin{equation*}
        f(e)-\varepsilon\leq h_n(e)\leq f(e), \quad e\in K,
    \end{equation*}
    for all sufficiently large $n$. Hence, $h_n\to f$ uniformly on $K$.

    Applying this scalar approximation componentwise to $F=(F_1,\dots,F_q)$, we
    obtain Lipschitz continuous functions $H^n:E\to\R^q$
    such that $H^n\to F$
    locally uniformly on $E=\ccC_T^d$.

    The functions $H^n$ do not necessarily satisfy a uniform linear growth estimate.
    We now enforce such an estimate by projection. Choose $C'>C$. For $r\ge0$, let
    $\pi_r:\R^q\to\R^q$ denote the projection onto the closed
    Euclidean ball $\overline{B_r(0)}$, i.e.
    \begin{equation*}
        \pi_r(z) :=
        \begin{cases}
            z,                   & \|z\|\leq r, \\
            r\,\dfrac{z}{\|z\|}, & \|z\|>r .
        \end{cases}
    \end{equation*}
    Define for $(t,x) \in \ccC_T^d$,
    \begin{equation*}
        F^n(t,x) = \pi_{C'(1+\sup_{s\in[0,t]}\|x(s)\|)} \bigl(H^n(t,x)\bigr).
    \end{equation*}
    Then, by construction,
    \begin{equation*}
        \|F^n(t,x)\| \leq C'\left(1+\sup_{s\in[0,t]}\|x(s)\|\right), \quad (t,x) \in \ccC_T^d.
    \end{equation*}
    Next we show that $F^n$ is Lipschitz. We use the elementary estimate
    \begin{equation*}
        \|\pi_r(z)-\pi_s(w)\| \leq \|z-w\|+|r-s|, \quad r,s \geq 0,\ z,w \in \R^q.
    \end{equation*}
    Indeed, for fixed $r$, the map $\pi_r$ is $1$-Lipschitz in $z$, and for
    fixed $w$,
    \begin{equation*}
        \|\pi_r(w)-\pi_s(w)\|\leq |r-s|.
    \end{equation*}
    Thus, if $(t,x), (t',x') \in \ccC_T^d$, then
    \begin{equation*}
        \begin{aligned}
            \|F^n(t,x)-F^n(t',x')\| & \leq \|H^n(t,x)-H^n(t',x')\| + C'|\sup_{s\in[0,t]}\|x(s)\|-\sup_{s\in[0,t']}\|x'(s)\||                                       \\
             & \leq \|H^n(t,x)-H^n(t',x')\| + C' \sup_{u\in[0,T]} \bigl\|x(u\wedge t)-x'(u\wedge t')\bigr\| \\
             & \leq (L_n+C')d_{\ccC_T}((t,x),(t',x')),
        \end{aligned}
    \end{equation*}
    where $L_n$ is a Lipschitz constant of $H^n$. Hence, $F^n$ is Lipschitz
    continuous with respect to $d_{\ccC_T}$.

    It remains to prove local uniform convergence. Let $K\subset \ccC_T^d$ be compact.
    Since $H^n\to F$ uniformly on $K$, we have
    \begin{equation*}
        \sup_{(t,x)\in K} \|H^n(t,x)-F(t,x)\| \to 0.
    \end{equation*}
    Moreover,
    \begin{equation*}
        \|F(t,x)\| \leq C(1+\sup_{s\in[0,t]}\|x(s)\|), \quad (t,x) \in \ccC_T^d.
    \end{equation*}
    Because $C'>C$ and $\sup_{s\in[0,t]}\|x(s)\|\ge0$, the term
    \begin{equation*}
        \bigl(C'-C\bigr)(1+\sup_{s\in[0,t]}\|x(s)\|)
    \end{equation*}
    is bounded from below by $C'-C>0$. Therefore, for all sufficiently large $n$,
    \begin{equation*}
        \|H^n(t,x)\| \leq C'(1+\sup_{s\in[0,t]}\|x(s)\|), \quad (t,x) \in K.
    \end{equation*}
    For such $n$, the projection is the identity on $K$, and hence
    \begin{equation*}
        F^n(t,x) = H^n(t,x), \quad (t,x) \in K.
    \end{equation*}
    Consequently,
    \begin{equation*}
        \sup_{(t,x) \in K}\|F^n(t,x)-F(t,x)\| = \sup_{(t,x) \in K}\|H^n(t,x)-F(t,x)\| \to 0.
    \end{equation*}
    Thus, $F^n\to F$ locally uniformly on $\ccC_T^d$.

    Finally, since $F^n$ is defined on the stopped-path space $\ccC_T^d$, it
    is non-anticipative by construction. If $t\in[0,T]$ is fixed and
    $x,y\in C([0,t],\R^d)$, then
    \begin{equation*}
        d_{\ccC_T}((t,x),(t,y)) = \sup_{s\in[0,t]}\|x(s)-y(s)\|.
    \end{equation*}
    Therefore, Lipschitz continuity with respect to $d_{\ccC_T}$ implies
    \begin{equation*}
        \|F^n(t,x)-F^n(t,y)\| \leq L_n \sup_{s\in[0,t]}\|x(s)-y(s)\|,
    \end{equation*}
    possibly after increasing $L_n$. This proves the claim.
\end{proof}

\begin{theorem}\label{thm:strong-existence}
    Assume that there exist constants $L > 0$ and $C_{LG} > 0$ such that the following conditions hold:
    \begin{enumerate}
        \item[(i)] (uniform Lipschitz condition) For all $(t,x) \in \ccC_T^d$ and $(t,y) \in\ccC_T^d$, we have
              \begin{equation}\label{Lipschitz-condition-existence}
                  \|b(t,x) - b(t,y)\|^2 + \|\sigma(t,x) - \sigma(t,y)\|^2 \leq L \sup_{0 \leq s \leq t}\|x(s)-y(s)\|^2.
              \end{equation}
        \item[(ii)] (linear growth condition) For all $(t,x) \in\ccC_T^d$, it holds that
              \begin{equation}\label{LG-condition-existence}
                  \|b(t,x)\|^2 + \|\sigma(t,x)\|^2 \leq C_{LG}\left(1 + \sup_{s \in [0,t]}\|x(s)\|^2 \right).
              \end{equation}
    \end{enumerate}
    Then the path-dependent SDE \eqref{path_dependent_sde_existence} admits a unique strong solution $X = (X(t))_{t \in [0,T]}$.
    Moreover, this solution satisfies the integrability condition $E\left[ \int_0^T \|X(t)\|^2 dt \right] < \infty$.
\end{theorem}

To prove this theorem we first establish moment and Hölder-regularity bounds for solutions.

\begin{lemma}\label{lm:estimate-existence}
    Assume the linear growth condition \eqref{LG-condition-existence}. Suppose $X = (X(t))_{t \in [0,T]}$ is a solution to equation \eqref{path_dependent_sde_existence}
    with initial condition $X(0)$ being $\ccF_0$-measurable and satisfying $E[\|X(0)\|^{2m}] < \infty$ for some $m \geq 1$. Then,
    \begin{equation}\label{eq:estimate-existence}
        E\left[\sup_{t \in [0,T]} \|X(t)\|^{2m}\right] \leq (1+3E[\|X(0)\|^{2m}])e^{3C_{LG}T(T+4)}.
    \end{equation}
    In particular, $X$ satisfies $E[\int_{0}^{T} \|X(s)\|^{2m} dt] < \infty$.
\end{lemma}

\begin{proof}
    For each $n \geq 1$, define the stopping time
    $$\tau_n =  \inf \{t \in [0,T] : \|X(t)\| \geq n\} \wedge T,$$
    and consider the truncated process $X^n(t) = X(t \wedge \tau_n)$ for $t \in [0,T]$. Then $X^n$ satisfies the SDE
    \begin{equation*}
        X^n(t) = X(0) + \int_{0}^{t} b(s, X^n_{[0,s]}) \Ind_{\{s \leq \tau_n\}} \, ds + \int_{0}^{t} \sigma(s,X^n_{[0,s]}) \Ind_{\{s \leq \tau_n\}} \, dW(s), \quad t \in [0,T].
    \end{equation*}
    Applying Hölder's inequality, Doob's martingale inequality and the linear growth condition \eqref{LG-condition-existence}, yields
    \begin{equation*}
        E\left[\sup_{0 \leq s \leq t} \|X^n(s)\|^{2m}\right] \leq 3 E\left[\|X(0)\|^{2m}\right] + 3C_{LG}(T+4)\int_{0}^{t} \left(1+E\left[\sup_{0 \leq r \leq s} \|X^n(r)\|^{2m}\right]\right) \, ds.
    \end{equation*}
    Applying Gronwall’s inequality to this bound gives
    \begin{equation*}
        1 + E\left[\sup_{0 \leq s \leq T} \|X^n(s)\|^{2m}\right] \leq \left(1 + 3E\left[\|X(0)\|^{2m}\right]\right)e^{3C_{LG}T(T+4)}.
    \end{equation*}
    Since $X^n(t) = X(t)$ for all $t \leq \tau_n$, we conclude
    \begin{equation*}
        E\left[\sup_{0 \leq s \leq \tau_n} \|X(s)\|^{2m}\right] \leq \left(1 + 3E\left[\|X(0)\|^{2m}\right]\right)e^{3C_{LG}T(T+4)}.
    \end{equation*}
    Finally, letting $n \to \infty$ and applying monotone convergence yields the desired estimate \eqref{eq:estimate-existence}.
    The integrability of $\int_0^T \|X(s)\|^{2m} ds$ follows by the bound on the supremum.
\end{proof}

\begin{lemma}\label{lm:hoelder-continuity}
    Assume that the linear growth condition \eqref{LG-condition-existence} holds and let $X = (X(t))_{t \in [0,T]}$ be a solution to \eqref{path_dependent_sde_existence}
    with $X(0)$ a $\ccF_0$-measurable random variable such that $E[\|X(0)\|^{2m}] < \infty$ for some $m \geq 1$.
    Then, for any $0 \leq s \leq t \leq T$, there exists a constant $C > 0$ depending only on $m, T, C_{LG}$ such that
    \begin{equation}\label{eq:diff-bound-in-expec}
        E\left[\|X(t)-X(s)\|^{2m}\right] \leq C(1+3E[\|X(0)\|^{2m}])(t-s)^m.
    \end{equation}
    Moreover, $X$ admits a version (which we again denote by $X$) which is Hölder continuous of any order $\alpha \in [0, \frac{1}{2} - \frac{1}{2m})$.
    For each such $\alpha$, there exists a constant $c > 0$ depending only on $m, T, C_{LG}, E[\|X(0)\|^{2m}]$ such that
    \begin{equation}\label{eq:hoelder-bound}
        E\left[\left(\sup_{0 \leq s < t \leq T} \frac{\|X(t)-X(s)\|}{|t-s|^\alpha}\right)^{2m}\right] \leq c.
    \end{equation}
\end{lemma}

\begin{proof}
    We begin by establishing the moment estimate \eqref{eq:diff-bound-in-expec}. Using the integral representation of $X$, we get
    \begin{equation*}
        E\left[\|X(t) - X(s)\|^{2m}\right] \leq 2^{2m-1} \left(E\left[\left\| \int_{s}^{t} b(u,X_{[0,u]}) \, du \right\|^{2m} \right] + E\left[\left\| \int_{s}^{t} \sigma(u,X_{[0,u]}) \, dW(u) \right\|^{2m}\right]\right).
    \end{equation*}
    To estimate the drift term, we apply Jensen's inequality and the linear growth condition \eqref{LG-condition-existence} to get
    \begin{align*}
        E\left[\left\| \int_{s}^{t} b(u,X_{[0,u]}) \, du \right\|^{2m} \right] & \leq (t-s)^{2m-1} \int_{s}^{t} E\left[\|b(u,X_{[0,u]})\|^{2m}\right] du                                           \\
                                                                               & \leq C_{LG}^{2m} (t-s)^{2m-1} \int_{s}^{t} \left(1+ E\left[\sup_{0 \leq r \leq u} \|X(r)\|^{2m}\right]\right) du.
    \end{align*}
    For the diffusion term we apply the BDG inequality, Jensen's inequality and the linear growth condition \eqref{LG-condition-existence}
    \begin{align*}
        E\left[\left\| \int_{s}^{t} \sigma(u,X_{[0,u]}) \, dW(u) \right\|^{2m}\right] & \leq C_m E\left[ \left( \int_{s}^{t} \left\|\sigma(u,X_{[0,u]})\right\|^{2} du \right)^m\right]                     \\
                                                                                      & \leq C_m (t-s)^{m-1} \int_{s}^{t} E\left[\|\sigma(u, X_{[0,u]})\|^{2m}\right] du                                    \\
                                                                                      & \leq C_{LG}^{2m}C_m (t-s)^{m-1} \int_{s}^{t} \left(1+ E\left[\sup_{0 \leq r \leq u} \|X(r)\|^{2m}\right]\right) du.
    \end{align*}
    Combining both drift and diffusion term, we get
    \begin{align*}
        E\left[\|X(t) - X(s)\|^{2m}\right] & \leq C_1(t-s)^{2m-1} \int_{s}^{t} \left(1+ E\left[\sup_{0 \leq r \leq u} \|X(r)\|^{2m}\right]\right) du     \\
                                           & \quad + C_2 (t-s)^{m-1} \int_{s}^{t} \left(1+ E\left[\sup_{0 \leq r \leq u} \|X(r)\|^{2m}\right]\right) du,
    \end{align*}
    where $C_1, C_2$ are positive constants depending only on $m$ and $C_{LG}$. Using now the bound obtained in Lemma \ref{lm:estimate-existence} we get
    \begin{align*}
        E\left[\|X(t) - X(s)\|^{2m}\right] & \leq \left( \widetilde{C}_1(t-s)^{2m-1} + \widetilde{C}_2(t-s)^{m-1}\right)(t-s)(1+3E[\|X(0)\|^{2m}]) \\
                                           & \leq \left( \widetilde{C}_1(t-s)^{2m} + \widetilde{C}_2(t-s)^{m}\right)(1+3E[\|X(0)\|^{2m}])          \\
                                           & \leq C(t-s)^m(1+3E[\|X(0)\|^{2m}]),
    \end{align*}
    where $\widetilde{C}_1, \widetilde{C}_2$ depend on $m,T,C_{LG}$ and thus also the positive constant $C$.

    The existence of a Hölder-continuous version and the bound \eqref{eq:hoelder-bound} now follow from the Kolmogorov continuity theorem; see e.g.\ \cite[Theorem I.2.1]{revuz_Continuous_1999}.
\end{proof}

\begin{proof}[Proof of Theorem \ref{thm:strong-existence}]
    Uniqueness:
    Let $X(t)$ and $\bar{X}(t)$ be two strong solutions of \eqref{path_dependent_sde_existence}. By Lemma \ref{lm:estimate-existence},
    both processes satisfy
    $$E\left[\int_{0}^{T} \|X(s)\|^2 dt\right] < \infty, \quad E\left[\int_{0}^{T} \|\bar{X}(s)\|^2 dt\right] < \infty.$$

    Noting further
    $$X(t) - \bar{X}(t) = \int_{0}^t \left[b(s, X_{[0,s]}) - b(s, \bar{X}_{[0,s]})\right] ds
        + \int_{0}^t \left[\sigma(s,X_{[0,s]}) - \sigma(s,\bar{X}_{[0,s]})\right] dW(s),$$
    we can thus show using the uniform Lipschitz condition \eqref{Lipschitz-condition-existence}
    $$E\left[\sup_{0 \leq s \leq t} \|X(s) - \bar{X}(s)\|^2\right]
        \leq 2L(t + 4) \int_{0}^t E\left[\sup_{0 \leq u \leq s} \|X(u) - \bar{X}(u)\|^2\right] ds.$$
    An application of Gronwall’s inequality yields
    $$E \left[ \sup_{0 \leq s \leq t} \|X(s) - \bar{X}(s)\|^2 \right] = 0,$$
    implying $X(t) = \bar{X}(t)$ almost surely for all $t \in [0,T]$. This proves uniqueness.

    Existence:
    We construct the solution via Picard iteration. Let $X^0(t) := X(0)$ for all $t \in [0,T]$. For each $n \in \N$, define
    \begin{equation}\label{eq:picard-iteration}
        X^{n}(t) := X(0) + \int_0^t b(s, X^{n-1}_{[0,s]}) \, ds + \int_0^t \sigma(s, X^{n-1}_{[0,s]}) \, dW(s), \quad t \in [0,T].
    \end{equation}
    It is now straightforward to show that $X^n(\cdot)$ satisfies $E[\int_{0}^{T} \|X^n(s)\|^2 dt] < \infty$.
    We first show that the sequence $(X^n)_{n \geq 0}$ is Cauchy in $L^2$ for all $n \in \N$, i.e.
    \begin{equation}\label{eq:existence-difference}
        E\left[\sup_{0 \leq s \leq t} \|X^{n+1}(s)-X^n(s)\|^2\right] \leq \frac{K(Mt)^n}{n!}, \quad \text{for } t \in [0,T],
    \end{equation}
    by induction, where $M := 2L(t+4)$ and $K := 2C_{LG} T(T+4)(1 + E[\|X(0)\|^2])$.
    For $n = 0$, we compute directly
    \begin{align*}
        X^1(t) - X^0(t) & = \int_0^t b(s, X^0_{[0,s]}) \, ds + \int_0^t \sigma(s, X^0_{[0,s]}) \, dW(s) \\
                        & = \int_0^t b(s, X(0)) \, ds + \int_0^t \sigma(s, X(0)) \, dW(s).
    \end{align*}
    Jensen's inequality, the BDG inequality and the linear growth condition now yields
    \begin{align*}
        E\left[ \sup_{0 \leq s \leq t} \|X^1(s) - X^0(s)\|^2 \right]
         & \leq 2 E \left[ \left\| \int_0^t b(s, X(0)) \, ds \right\|^2 \right] + 2 E \left[ \sup_{0 \leq s \leq t} \left\| \int_0^t \sigma(s, X(0)) \, dW(s) \right\|^2 \right] \\
         & \leq 2t \int_0^t E[\|b(s, X(0))\|^2] ds + 8 \int_0^t E[\|\sigma(s, X(0))\|^2] ds                                                                                      \\
         & \leq K.
    \end{align*}
    Now assume that \eqref{eq:existence-difference} holds for some $n \geq 0$. Then
    \begin{align*}
        E\left[ \sup_{0 \leq s \leq t} \|X^{n+2}(s) - X^{n+1}(s)\|^2 \right]
         & \leq M \int_0^t E \left[ \sup_{0 \leq r \leq s} \|X^{n+1}(r) - X^n(r)\|^2 \right] ds \\
         & \leq M \int_0^t \frac{K(Ms)^n}{n!} ds = \frac{K (Mt)^{n+1}}{(n+1)!},
    \end{align*}
    completing the induction.
    From this, using Markov's inequality and Borel--Cantelli, we get almost sure uniform convergence of the series
    $$X^n(t) = X^0(t) + \sum_{i=0}^{n-1} \left(X^{i+1}(t) - X^i(t)\right).$$
    Define $X(t) := \lim_{n \to \infty} X^n(t)$. Then $X$ is continuous and adapted, and the bound \eqref{eq:existence-difference} ensures convergence in $L^2(\Omega)$ for each $t \in [0,T]$. Moreover,
    $$E \left[ \int_0^T \|X(t)\|^2 dt \right] < \infty.$$
    It remains to show that $X$ satisfies the integral equation \eqref{path_dependent_sde_existence}. Using Lipschitz continuity and dominated convergence, we note that
    \begin{align*}
         & E\left[ \left\|\int_{0}^{t} b(s,X^n_{[0,s]}) \, ds - \int_{0}^{t} b(s,X_{[0,s]}) \, ds\right\|^2\right]                     \\
         & + E\left[ \left\|\int_{0}^{t} \sigma(s,X^n_{[0,s]}) \, dW(s) - \int_{0}^{t} \sigma(s,X_{[0,s]}) \, dW(s)\right\|^2\right]   \\
         & \leq L(T+1) \int_{0}^{T} E\left[ \sup_{0 \leq s \leq T}\|X^n(s) - X(s)\|^2\right] ds \to 0, \quad \text{ as } n \to \infty.
    \end{align*}
    Hence, passing to the limit in \eqref{eq:picard-iteration} yields
    $$X(t) = X(0) + \int_0^t b(s, X_{[0,s]}) \, ds + \int_0^t \sigma(s, X_{[0,s]}) \, dW(s), \quad t \in [0,T].$$
    Thus, $X$ is a strong solution of \eqref{path_dependent_sde_existence}. This completes the proof.
\end{proof}

To establish existence results we follow the same path as in the Volterra case shown in \cite{abijaber_Affine_2019}.
Fix $\mu \in \ccM([0,T], \C^d)$ and $p : \{(t,s,x) \in \R_{\geq 0}^2 \times \C^d : 0 \leq t \leq s < \infty \} \to \C^d$ and consider the Riccati-type equation
\begin{equation}\label{eq:def-basline-existence}
    \psi(t) = \mu([t,T]) + \int_{t}^{T} p(t,s,\psi(s)) \, ds, \quad t \in [0,T].
\end{equation}
A \emph{noncontinuable solution} of \eqref{eq:def-basline-existence} is a pair $(\psi,T_{\min})$ with $T_{\min} \in [0,T)$ and $\psi \in L^2((T_{\min},T],\C^d)$, such that $\psi$ satisfies \eqref{eq:def-basline-existence} on $(T_{\min},T]$ and $\|\psi \|_{L^2(T_{\min},T)} = \infty$ if $T_{\text{min}} > 0$. If $T_{\min} = 0$ we call $\psi$ a global solution of \eqref{eq:def-basline-existence}. We call a noncontinuable solution $(\psi, T_{\min})$ \emph{unique} if for any $\tau \in [0,T]$ and $\widetilde{\psi} \in L^2([\tau,T],\C^d)$ satisfying \eqref{eq:def-basline-existence} on $[\tau,T]$, we have $T_{\min} < \tau$ and $\psi = \widetilde{\psi}$ on $[\tau,T]$.

\begin{theorem}\label{thm:existence-of-noncontinuable-riccati}
    Let $T \in \R_{\geq 0}$.
    Assume that $\mu \in \ccM([0,T], \C^d)$, and
    $P_T \in L^1([0,T],\R_{\geq 0})$ such that
    \begin{equation}\label{eq:riccati-zero-growth-condition}
        |p(t,s,0)| \leq P_T(s), \quad 0\leq t\leq s\leq T,
    \end{equation}
    and that there exists a positive constant $\Theta_T$ and a function $\Pi_T \in L^2([0,T], \R_{\geq 0})$ such that
    \begin{equation}\label{eq:riccati-growth-condition}
        |p(t,s,x)-p(t,s,y)| \leq \Pi_T(s)|x-y| + \Theta_T|x-y|(|x|+|y|), \quad x,y \in \C^d, 0 \leq t \leq s \leq T.
    \end{equation}
    Then the Riccati-type equation \eqref{eq:def-basline-existence} has a unique noncontinuable solution $(\psi, T_{\text{min}})$. If $\mu$ and $p$ are real-valued, then so is $\psi$.
\end{theorem}

\begin{proof}
    We focus only on the complex-valued case. The real-valued case follows by replacing $\C^d$ with $\R^d$ below.
    We first prove that a solution exists for small times.
    Let $\rho \in (0,T]$ and $\varepsilon > 0$ be constants and define
    $$B_{\rho, \varepsilon} = \{\psi \in L^2([T-\rho,T], \C^d) : \|\psi\|_{L^2(T-\rho,T)} \leq \varepsilon\}.$$
    Consider the operator $F$ acting on elements $\psi \in B_{\rho, \varepsilon}$ by
    \begin{equation*}
        (F\psi)(t) = \mu([t,T]) + \int_{t}^{T} p(t,s,\psi(s)) \, ds, \quad t \in [T-\rho,T].
    \end{equation*}
    Let $M_{\rho} = \sup_{t \in [T-\rho,T]} |\mu([t,T])|$. Since $\mu$ has finite total variation we have $M_{\rho} < \infty$.
    Using \eqref{eq:riccati-zero-growth-condition} and \eqref{eq:riccati-growth-condition} we get for $\psi \in B_{\rho, \varepsilon}$ and $t \in [T-\rho,T]$ with the Cauchy--Schwarz inequality
    \begin{align*}
        \left| \int_{t}^{T} p(t,s,\psi(s)) \, ds \right| & \leq \int_{t}^{T} |p(t,s,0)| \, ds + \int_{t}^{T} |p(t,s,\psi(s))-p(t,s,0)| \,ds                                                   \\
                                                         & \leq  \int_{T-\rho}^{T} P_T(s) \, ds + \int_{T-\rho}^{T} \Pi_{T}(s) |\psi(s)| \, ds + \Theta_T \int_{T-\rho}^{T} |\psi(s)|^2 \, ds \\
                                                         & \leq \|P_T\|_{L^1(T-\rho,T)} + \|\Pi_{T}\|_{L^2(0,T)}\|\psi\|_{L^2(T-\rho,T)} + \Theta_T\|\psi\|_{L^2(T-\rho,T)}^2                 \\
                                                         & \leq \|P_T\|_{L^1(T-\rho,T)} + \|\Pi_{T}\|_{L^2(0,T)}\varepsilon + \Theta_T \varepsilon^2.
    \end{align*}
    This now implies for $\psi \in B_{\rho, \varepsilon}$
    \begin{align}\label{eq:L^2-solution}
        \begin{split}
            \|F(\psi)\|_{L^2(T-\rho,T)} & = \left(\int_{T-\rho}^{T} |(F\psi)(s)|^2 ds\right)^\frac{1}{2}                                                                \\
            %   & \leq \left(\int_{T-\rho}^{T} \left(M_{\rho} + \|\Pi_{T}\|_{L^2}\varepsilon + \Theta_T \varepsilon^2\right)^2 ds\right)^\frac{1}{2} \\
                                        & \leq \sqrt{\rho}\left(M_{\rho} + \|P_T\|_{L^1(T-\rho,T)} + \|\Pi_{T}\|_{L^2(0,T)}\varepsilon + \Theta_T \varepsilon^2\right).
        \end{split}
    \end{align}
    Similarly, for $\psi, \widetilde{\psi} \in B_{\rho, \varepsilon}$ and every $t \in [T-\rho,T]$
    \begin{align*}
        |(F\psi)(t) - (F \widetilde{\psi})(t)| & \leq \int_{T-\rho}^{T} \Pi_{T}(s) |\psi(s) - \widetilde{\psi}(s)| \, ds + \Theta_{T} \int_{T-\rho}^{T} |\psi(s) - \widetilde{\psi}(s)|\bigl(|\psi(s)|+|\widetilde{\psi}(s)|\bigr) \, ds \\
        %    & \leq \|\Pi_{T}\|_{L^2}\|\psi - \widetilde{\psi}\|_{L^2} + \Theta_T\|\psi - \widetilde{\psi}\|_{L^2}\bigl(\|\psi\|_{L^2} + \|\widetilde{\psi}\|_{L^2}\bigr)                              \\
                                               & \leq \bigl(\|\Pi_{T}\|_{L^2(0,T)} + 2\Theta_T \varepsilon\bigr)\|\psi - \widetilde{\psi}\|_{L^2(T-\rho,T)} .
    \end{align*}
    Taking $L^2$-norm yields
    \begin{align*}
        \|F\psi - F \widetilde{\psi}\|_{L^2(T-\rho,T)} \leq \sqrt{\rho}\bigl(\|\Pi_{T}\|_{L^2(0,T)} + 2 \Theta_{T}\varepsilon\bigr) \|\psi - \widetilde{\psi}\|_{L^2(T-\rho,T)}.
    \end{align*}
    Since $\mu$ has finite variation, $M_\rho$ is bounded and therefore $\sqrt{\rho}M_\rho \to 0$ as $\rho \downarrow 0$. Moreover, since $P_T \in L^1([0,T])$, we have $\|P_T\|_{L^1(T-\rho,T)} \to 0$ as $\rho \downarrow 0$. Thus, for any given $\varepsilon > 0$ we may find a sufficiently small $\rho > 0$ such that
    \begin{align*}
        \sqrt{\rho}\left(M_{\rho} +\|P_T\|_{L^1(T-\rho,T)} +  \|\Pi_{T}\|_{L^2(0,T)}\varepsilon + \Theta_T \varepsilon^2\right) & \leq \varepsilon \\
        \sqrt{\rho}\bigl(\|\Pi_{T}\|_{L^2(0,T)} + 2 \Theta_{T}\varepsilon\bigr)                                                 & < 1.
    \end{align*}
    This implies that $F$ maps $B_{\rho, \varepsilon}$ to itself and is a contraction there, so Banach's fixed-point theorem yields that $F$ has a unique fixed point $\psi \in B_{\rho, \varepsilon}$, which is a solution of \eqref{eq:def-basline-existence}.

    We now extend this to a unique noncontinuable solution. Define the set
    \begin{equation*}
        J = \left\{ \tau \in [0, T] : \eqref{eq:def-basline-existence} \text{ admits a solution } \psi \in L^2([\tau, T], \C^d) \text{ on } [\tau, T] \right\}.
    \end{equation*}
    We know $T \in J$, and if $0 \in J$ and $0 \leq S \leq T$, then $S \in J$. So $J$ is a non-empty interval. Moreover, $J$ is open on the set of intervals ending at $T$.

    Let $\tau \in J$ with $\tau > 0$, and let $\psi$ be the corresponding solution. For $t \in [0,\tau)$, the equation can be rewritten as
    \begin{equation*}
        \psi(t) = h(t) + \int_t^\tau p(t, s, \psi(s)) \, ds,
    \end{equation*}
    where
    \begin{equation*}
        h(t) = \mu([t,T]) + \int_\tau^T p(t, s, \psi(s)) \, ds,
    \end{equation*}
    lies in $L^2_{\text{loc}}([0,T])$.
    By what we already proved, there exists $\rho > 0$ such that the solution $\psi$ can be extended to $[\tau - \rho, T]$.
    Consequently, $J$ is of the form $(T_{\min}, T]$ for some $0 \leq T_{\min} < T$.
    It remains to argue uniqueness. Pick any $\tau \in [0, T]$ and let $\widetilde{\psi} \in L^2([\tau, T])$ be another solution to \eqref{eq:def-basline-existence} on $[\tau, T]$. Therefore, $\tau \in J$ and thus $T_{\min} < \tau$. Define
    \begin{equation*}
        S = \inf \{ s \in [\tau, T] : \tilde{\psi} = \psi \text{ a.e. on } [s, T] \}.
    \end{equation*}
    Then $\tilde{\psi} = \psi$ on $[S, T]$. If $S > \tau$, consider a small interval $[S-\rho, S]$.
    On this interval both functions satisfy the same equation with the same future term
    \begin{equation*}
        h(t) = \mu([t,T]) + \int_{S}^{T} p(t,s,\psi(s)) \, ds.
    \end{equation*}
    Repeating the contraction estimate we have for $\rho > 0$ on this interval
    \begin{equation*}
        \|\psi - \tilde{\psi}\|_{L^2([S-\rho, S])} \leq \kappa \|\psi - \tilde{\psi}\|_{L^2([S-\rho, S])}
    \end{equation*}
    for some $\kappa < 1$, provided that $\rho$ is sufficiently small.
    This implies $\psi = \tilde{\psi}$ on $[S-\rho, S]$, contradicting the definition of $S$. Thus, $S = \tau$ and the solution is unique.
\end{proof}

\section{Functional It\^o calculus}\label{app:functional_ito}

This appendix collects the basic definitions and results from the functional It\^o-calculus
used in the paper.
We follow \cite{cont_Functional_2013} and restrict attention to functionals that depend
only on the path and not explicitly on the quadratic variation.

\subsection{Non-anticipative functionals and stopped paths}

Let $X:[0,T] \times \Omega \to \R^d$ be a continuous semimartingale.
Although paths are almost surely continuous, c\`adl\`ag perturbations will be required for the functional It\^o formula,
so we work on the Skorokhod space $D([0,T],\R^d)$ in this section.
A non-anticipative functional is a family $F=(F_t)_{t \in [0,T]}$ with
$$F_t : D([0,t],\R^d) \to \R$$
measurable with respect to the canonical filtration.

For a path $x \in D([0,T],\R^d)$ we denote its value at time $t$ by $x(t)$ and its \emph{path} until $t$ by $x_t=(x(s) \colon 0 \leq s \leq t)\in D([0,t],\R^d)$. For $h \geq 0$, its \emph{horizontal} extension $x_{t,h} \in D([0,t+h],\R^d)$ is given as
$$ x_{t,h}(u) = x(u \wedge t), \quad u \in [0,t+h].$$

\subsection{Continuity}

For continuity, it is useful to consider $F$ on a single space, which will be the following \emph{space of stopped paths:}
\begin{align}
    \Upsilon([0,T], \R^d) = \big\{(t,x_{t,T-t}) \colon (t,x) \in [0,T] \times D([0,T],\R^d) \big\}.
\end{align}
Intuitively, two paths stopped at $t$ coincide before $t$ and are identified with the path stopped at $t$, i.e.\ $x_{t,T-t}$.
We  notice that non-anticipative functionals can be viewed as functionals on $\Upsilon([0,T], \R^d)$.
As a distance on $\Upsilon([0,T], \R^d)$ we consider
$$ d_\infty\big( (t,x),(t',x') \big) = |t-t'| + \sup_{u \in [0,T]} \| x_{t,T-t}(u) -x'_{t',T-t'}(u) \|. $$
With this metric $(\Upsilon([0,T], \R^d),d_\infty)$ is a closed subspace of $[0,T] \times D([0,T],\R^d)$ under the product topology.

Based on this metric, we can introduce a natural notion of continuity. More precisely, a non-anticipative functional $F$ on $\Upsilon([0,T], \R^d)$ is called \emph{continuous} at $(t,x) \in \Upsilon([0,T], \R^d)$ if for all $\varepsilon > 0$  there exists $\eta >0$ such that
for all  $(t',x') \in \Upsilon([0,T], \R^d)$ with $d_\infty((t,x),(t',x')) < \eta$, it holds that
$$ \big| F_t(x) - F_{t'}(x') \big| < \varepsilon. $$ We denote by $\C^{0,0}([0,T])$ the set of continuous non-anticipative functionals over $[0,T]$.

We will also need the following weaker notion of continuity: the non-anticipative functional $F$ is called \emph{continuous at fixed times} if for all $t\in[0,T)$, all $\varepsilon > 0$ and all $(t,x) \in \Upsilon([0,T], \R^d)$ there exists $\eta >0$ such that if $d_\infty((t,x),(t,x'))< \eta$ for $(t,x') \in \Upsilon([0,T], \R^d)$ implies that
$$ \big| F_t(x) - F_{t}(x') \big| < \varepsilon. $$

\subsection{Differentiability}

The \emph{horizontal derivative} at $x \in D([0,t],\R^d)$ of the non-anticipative functional $F$ at time $t$ is defined as
$$ \cD_t F(x) = \lim_{h \to 0+} \frac{F_{t+h}(x_{t,h}) - F_t(x_t)}{h}, $$
if the limit exists.

For $h \in \R^d$, we define its \emph{vertical} perturbation $x_t^h \in D([0,t],\R^d)$ by
$$ x_t^h(u) = x_t(u) + h \ind{u=t}, \quad u \in [0,t]$$
which is obtained by shifting the endpoint of $x_t$ by $h$.

The vertical derivative is a directional derivative. We therefore denote by $(e_i,i=1,\dots,d)$ the standard basis of $\R^d$. The \emph{vertical derivative} at $x \in D([0,t],\R^d)$ of the non-anticipative functional $F$ at time $t$ is component-wise, i.e.
\ $\nabla_x F_t(x) = (\partial_i F_t(x))_{i=1,\dots,d}$ where
$$\partial_i F_t(x) = \lim_{h \to 0} \frac{F_{t}(x_{t}^{h e_i}) - F_t(x_t)}{h}, \quad i=1,\dots,d. $$
If this derivative is defined for all $(t,x) \in \Upsilon$, $\nabla_x F=(\nabla_x F_t)_{t \in [0,T]}$ defines a non-anticipative functional, the \emph{vertical derivative} of $F$.
Note that the directional derivative involves evaluating $F$ at a c\`adl\`ag perturbation, even if $x$ itself is continuous, which explains why we have to consider the space $D([0,T],\R^d)$ instead of simply working on $C([0,T],\R^d)$.

Since the space $D([0,T],\R^d)$ is not separable under the sup-norm, we need the following additional regularity, even for the continuous functionals.

A non-anticipative functional $F$ is said to be boundedness preserving if for any compact $K \subset \R^d$ and $t_0 < T$, there exists a constant $C_{K,t_0} > 0$ such that for all $t_0 \leq t$, $x \in D([0,T], \R^d)$
\begin{equation*}
    x([0,t]) \subseteq K \Rightarrow |F_t(x)| \leq C_{K,t_0}.
\end{equation*}

We denote by $\C^{1,k}([0,T])$ the space of functionals $F \in \C^{0,0}([0,T])$
that are once horizontally and $k$ times vertically differentiable,
and by $\C^{1,k}_b([0,T])$ those whose derivatives are boundedness preserving.

\subsection{The functional It\^o formula}

With these preliminaries, we are able to state the functional It\^o-formula\footnote{Note that condition (10) in Theorem 4.1 of \cite{cont_Functional_2013} is not needed in our context, since the functionals we consider do not depend on the quadratic variation.} (see Theorem 4.1 in \cite{cont_Functional_2013}).
\begin{theorem}
    Let $F$ be a non-anticipative functional such that $F \in \C^{1,2}_b$. Then, for all $t \in [0,T)$,
    \begin{equation*}
        F_t(X_t) = F_0 (X_0) + \int_0^t \cD_u F(X_u) \, du + \int_0^t \nabla_x F_u(X_u) \, dX(u)  + \frac 1 2 \int_0^t \tr(\nabla^2_x F_u(X_u)) \, d\langle X \rangle(u), \quad \text{a.s.}
    \end{equation*}
\end{theorem}

\newpage
\nocite{*}
\bibliographystyle{agsm}
\bibliography{reference.bib}
\end{document}